\documentclass[british, a4paper,11pt]{article}

\usepackage[utf8]{inputenc}
\usepackage[french, main=british]{babel}

\usepackage[mono=false]{libertine}
\usepackage[T1]{fontenc}

\usepackage{amsthm}
\usepackage[cmintegrals,libertine]{newtxmath}
\usepackage[cal=dutchcal, calscaled=.95, scr=boondoxo]{mathalfa}
\useosf
\let\mathcal\mathbcal

\usepackage{microtype}

\usepackage{numprint}

\usepackage[explicit]{titlesec}

\titleformat{\section}[block]{\normalfont\large\bfseries\boldmath\centering}{\raggedright\makebox[1em][l]{\thesection.}}{.5em}{#1}
\titleformat{\subsection}[runin]{\normalfont\bfseries\boldmath}{\raggedright\makebox[1em][l]{\thesubsection.}}{1.25em}{#1.~\hbox{---}}
\titleformat{\subsubsection}[runin]{\normalfont\bfseries\boldmath}{\raggedright\makebox[1em][l]{\thesubsubsection.}}{1.25em}{#1.~\hbox{---}}

\renewenvironment{abstract}{%
\begin{center}
\begin{minipage}{.9\textwidth}\linespread{1.05}\selectfont\small
\makebox[5em][l]{\bfseries\abstractname.~\hbox{---}}
\normalfont}
{\par\vspace{1em}
\end{minipage}
\end{center}
}

\usepackage[a4paper,vmargin={2cm,2cm},hmargin={2.25cm,2.25cm}]{geometry}
\selectfont

\usepackage{tikz}
	\usetikzlibrary{calc}
	\usetikzlibrary{decorations.markings}

\usepackage[font=sf]{caption}
\usepackage{hyperref}
\hypersetup{
	colorlinks=true,
		citecolor=blue!60!black,
		linkcolor=red!60!black,
		urlcolor=green!40!black,
		filecolor=yellow!50!black,
	breaklinks=true,
	pdfpagemode=UseNone,
	bookmarksopen=false,
}

\usepackage{enumitem}

\newtheorem{thm}{\bfseries \upshape Theorem}
\newtheorem{lem}{Lemma}
\newtheorem{prop}{Proposition}
\newtheorem{cor}{Corollary}

\theoremstyle{definition}
\newtheorem{rem}{Remark}

\renewcommand{\le}{\leqslant}
\renewcommand{\ge}{\geqslant}

\newcommand{\ensembles}[1]{\mathbf{#1}}
	\newcommand{\N}{\ensembles{N}}
	\newcommand{\Z}{\ensembles{Z}}
	\newcommand{\R}{\ensembles{R}}
	
	\renewcommand{\P}{\ensembles{P}}
	\newcommand{\E}{\ensembles{E}}
	\newcommand{\Var}{\mathrm{Var}}

	\newcommand{\1}{\ensembles{1}}
\newcommand{\ind}[1]{\1_{\{#1\}}}

\newcommand{\tildep}[1]{\widetilde{#1}\vphantom{#1}}
\newcommand{\underlinep}[1]{\underline{#1}\vphantom{#1}}

\renewcommand{\Pr}[1]{\P\left(#1\right)}
\newcommand{\Prc}[2]{\P\left(#1 \;\middle|\; #2\right)}
\newcommand{\Es}[1]{\E\left[#1\right]}
\newcommand{\Esc}[2]{\E\left[#1 \;\middle|\; #2\right]}

\newcommand{\Map}{\mathbf{M}}
\newcommand{\PMap}{\mathbf{PM}}

\newcommand{\Tree}{\mathbf{T}}
\newcommand{\LTree}{\mathbf{LT}}

\newcommand{\Fn}{\Tree_{d_n}^{\varrho_n}}

\newcommand{\fn}{T_{d_n}^{\varrho_n}}
\newcommand{\tn}{T_{d_n}}

\newcommand{\Hfn}{H_{d_n}^{\varrho_n}}

\newcommand{\Wfn}{W_{d_n}^{\varrho_n}}

\newcommand{\infWfn}{\underlinep{W}_{d_n}^{\varrho_n}}

\newcommand{\Bfn}{B_{d_n}^{\varrho_n}}
\newcommand{\Btn}{B_{d_n}}
\newcommand{\bfn}{b_{d_n}^{\varrho_n}}
\newcommand{\Lfn}{L_{d_n}^{\varrho_n}}
\newcommand{\tLfn}{\tildep{L}_{d_n}^{\varrho_n}}

\newcommand{\tX}{\tildep{X}^{\varrho}}

\newcommand{\LFn}{\LTree_{d_n}^{\varrho_n}}

\newcommand{\Mn}{\Map_{d_n}^{\varrho_n}}
\newcommand{\mn}{M_{d_n}^{\varrho_n}}
\newcommand{\Vmn}{V(M_{d_n}^{\varrho_n})}
\newcommand{\PMn}{\PMap_{d_n}^{\varrho_n}}

\newcommand{\CRT}{\mathscr{T}}
\newcommand{\dCRT}{\mathscr{d}}
\newcommand{\pCRT}{\mathscr{p}}
\newcommand{\Bmap}{\mathscr{M}}
\newcommand{\dBmap}{\mathscr{D}}
\newcommand{\pBmap}{\mathscr{p}}
\newcommand{\dgr}{d_{\mathrm{gr}}}
\newcommand{\pgr}{p_{\mathrm{unif}}}

\newcommand{\h}{\mathrm{ht}}
\newcommand{\wid}{\mathrm{wid}}

\newcommand{\bb}{\mathbf{b}}
\renewcommand{\d}{\mathrm{d}}
\newcommand{\e}{\mathrm{e}}

\newcommand{\m}{\mathbf{m}}
\newcommand{\q}{\mathbf{q}}

\renewcommand{\i}{\mathrm{i}}

\newcommand{\edges}{E}
\newcommand{\vertices}{V}

\newcommand{\Cont}{\mathsf{Cont}}
\newcommand{\LR}{\mathsf{LR}}
\newcommand{\RR}{\mathsf{R}}
\newcommand{\LL}{\mathsf{L}}

\newcommand{\cv}[1][n]{\enskip\mathop{\longrightarrow}^{}_{#1 \to \infty}\enskip}
\newcommand{\cvloi}[1][n]{\enskip\mathop{\longrightarrow}^{(d)}_{#1 \to \infty}\enskip}
\newcommand{\cvps}[1][n]{\enskip\mathop{\longrightarrow}^{a.s.}_{#1 \to \infty}\enskip}
\newcommand{\cvproba}[1][n]{\enskip\mathop{\longrightarrow}^{\P}_{#1 \to \infty}\enskip}

\newcommand{\equivalent}[1][n]{\enskip\mathop{\thicksim}^{}_{#1 \to \infty}\enskip}

\newcommand{\eqloi}{\enskip\mathop{=}^{(d)}\enskip}

\title{On scaling limits of random trees and maps with a prescribed degree sequence}

\author{Cyril \textsc{Marzouk}\thanks{CMAP, \'{E}cole polytechnique.\hfill  \href{mailto:cyril.marzouk@polytechnique.edu}{\texttt{cyril.marzouk@polytechnique.edu}}
\newline
This project has received funding from the Fondation Mathématique Jacques Hadamard as well as, thanks to Guillaume Chapuy, the European Research Council (ERC) under the European Union's Horizon 2020 research and innovation programme, grant agreement No. \texttt{ERC-2016-STG 716083} ``CombiTop''.
}}

\begin{document}

\maketitle

\begin{abstract}
We study a configuration model on bipartite planar maps in which, given $n$ even integers, one samples a planar map with $n$ faces uniformly at random with these face degrees. We prove that when suitably rescaled, such maps always admit nontrivial subsequential limits as $n \to \infty$ in the Gromov--Hausdorff--Prokhorov topology. Further, we show that they converge in distribution towards the celebrated Brownian sphere, and more generally a Brownian disk for maps with a boundary, if and only if there is no inner face with a macroscopic degree, or, if the perimeter is too big, the maps degenerate and converge to the Brownian tree. 
By first sampling the degrees at random with an appropriate distribution, this model recovers that of size-conditioned Boltzmann maps associated with critical weights in the domain of attraction of a stable law with index $\alpha\in [1,2]$. The Brownian tree and disks then appear respectively in the case $\alpha=1$ and $\alpha=2$, whereas in the case $\alpha \in (1,2)$ our results partially recover previous known ones.
Our proofs rely on known bijections with labelled plane trees, which are similarly sampled uniformly at random given $n$ outdegrees. Along the way, we obtain some results on the geometry of such trees, such as a convergence to the Brownian tree but only in the weaker sense of subtrees spanned by random vertices, which are of independent interest.
\end{abstract}

\begin{figure}[!ht] \centering
\includegraphics[height=17\baselineskip]{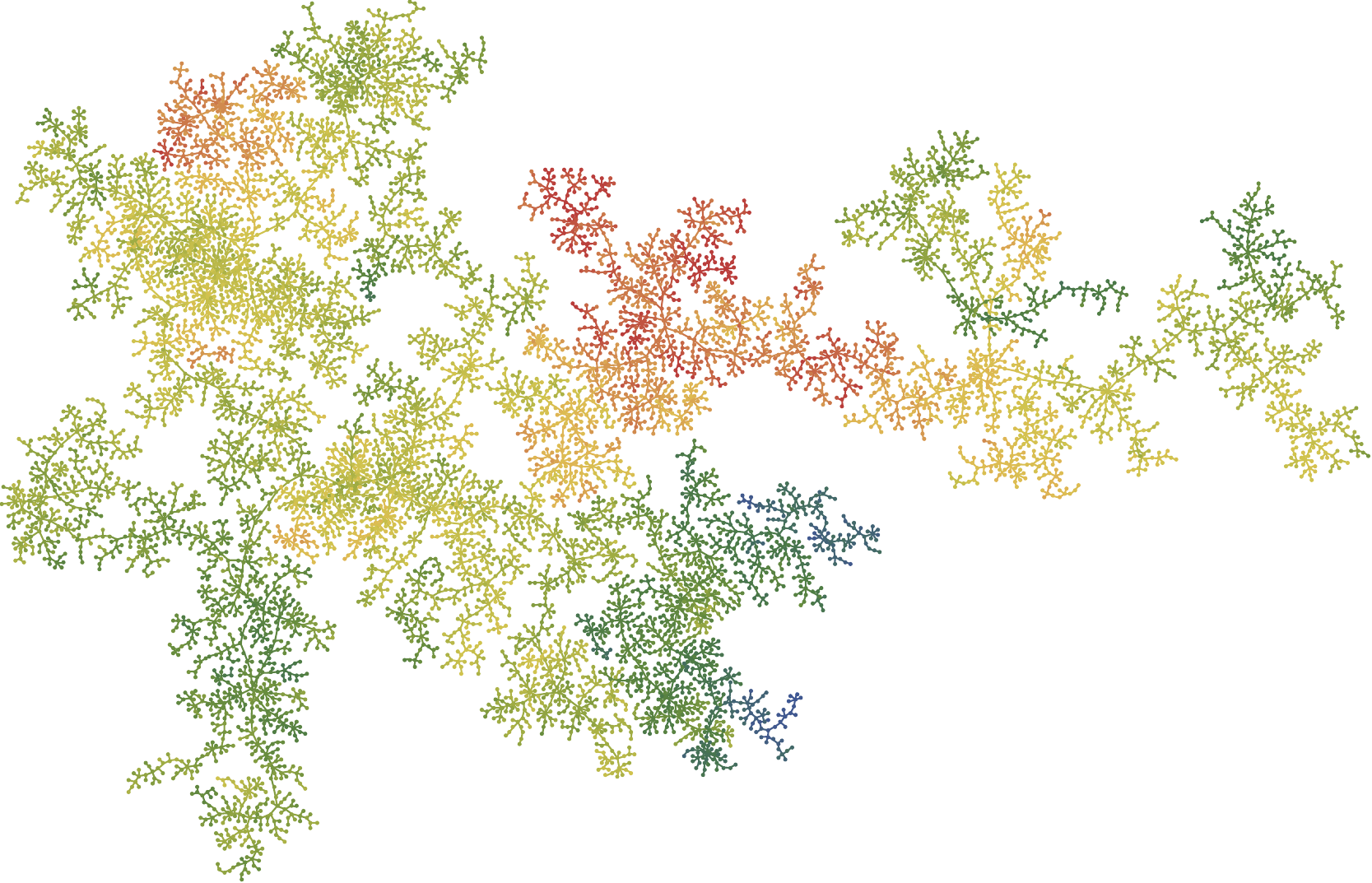}
\caption{A Brownian CRT with Brownian labels describing the Brownian map: labels are indicated by colours (red for the highest values and blue for the lowest).}
\label{fig:serpent_brownien}
\end{figure}

\newpage

\section{Introduction}

This paper deals with `continuum' limits of random planar maps as their size tends to infinity and their edge-length tends to zero appropriately. 
The most notable result in this theory has been obtained by Le~Gall~\cite{LG13} and Miermont~\cite{Mie13} who proved the convergence of large random quadrangulations towards a limit called the \emph{Brownian map} or the \emph{Brownian sphere}. 
Building on these works, the Brownian sphere has then been shown to be a universal limit of many models of discrete maps and this paper continues with a large class of distributions
introduced and studied previously in~\cite{Mar18b} in a restricted case.

\subsection{Model and main results}

A rooted planar map is a finite connected multigraph embedded in the two-dimensional sphere, in which one oriented root edge is distinguished, and viewed up to orientation-preserving homeomorphisms. 
The embedding allows to define the faces of the map, which are the connected components of the complement of the graph in the sphere; the degree of a face is then the number of edges incident to it, counted with multiplicity.
We shall consider maps with a, non-necessarily simple, boundary, given by the face incident to the right of the root edge; the degree of this face is called the perimeter of the map.
As usual in this field, for technical reasons briefly discussed later in this section, we restrict ourselves to \emph{bipartite} maps, in which all faces have even degree. 

Random maps provide discrete models of random geometry on the sphere and an important question is their convergence towards continuum random geometries when their size tends to infinity and one rescales the edge-lengths properly.
In this respect, in the case of uniformly random quadrangulations with $n$ faces, i.e. $n$ faces each with degree $4$, Le~Gall~\cite{LG07} proved that such graphs admit subsequential limits at the scaling $n^{1/4}$, and that these limits all have the same topology (later identified as the sphere~\cite{LGP08, Mie08}) and the same Hausdorff dimension $4$.
The problem of uniqueness of the subsequential limits was solved simultaneously by Le~Gall~\cite{LG13} and Miermont~\cite{Mie13} who proved that rescaled random quadrangulations converge in distribution towards a limit called the Brownian map, or sphere, which is a compact metric measured space $\Bmap^0 = (\Bmap^0,\dBmap^0,\pBmap^0)$ thus almost surely homeomorphic to the sphere and with Hausdorff dimension $4$. In the case of quadrangulations with a boundary, Bettinelli \& Miermont~\cite{BM17} proved that when the perimeter behaves like $\varrho n^{1/2}$ with $\varrho \in (0, \infty)$ fixed, they converge to the \emph{Brownian disk} with perimeter $\varrho$, denoted by $\Bmap^\varrho = (\Bmap^\varrho,\dBmap^\varrho, \pBmap^\varrho)$, which now has the topology of a disk, with Hausdorff dimension $4$, and its boundary has Hausdorff dimension $2$~\cite{Bet15}.

These works raise then the question of the universality of the limits which should not depend too much on the details of the discrete models. The present work proposes an answer by considering a more general model that we now introduce. 
For every integer $n \ge 1$, let us give ourselves an integer $\varrho_n \ge 1$ and a sequence $(d_n(k))_{k \ge 1}$ of nonnegative integers such that $\sum_{k \ge 1} d_n(k) = n$ and $d_n(1) < n$ to avoid trivialities; then let $\Mn$ denote the set of all (rooted planar bipartite) maps with perimeter $2 \varrho_n$ and $n$ inner faces, amongst which exactly $d_n(k)$ have degree $2k$ for every $k \ge 1$.
A key quantity in this work is
\[\sigma_n^2 \coloneqq \sum_{k \ge 1} k (k-1) d_n(k).\]
We stress that $\sigma_n^2$ really depends on the sequence $d_{n}$, not only on $n$, but we chose this lighter notation.
We sample $\mn$ uniformly at random in $\Mn$ and consider the asymptotic behaviour as $n \to \infty$ of its vertex set $V(\Mn)$ endowed with the graph distance $\dgr$ and the uniform probability measure $\pgr$.

\begin{thm}
\label{thm:tension_cartes}
Fix any sequence of perimeters $(\varrho_n)_{n \ge 1}$ and any degree sequence $(d_n)_{n \ge 1}$. From every increasing sequence of integers, one can extract a subsequence along which the sequence of metric measured spaces
\[\left(\Vmn, (\sigma_n + \varrho_n)^{-1/2} \dgr, \pgr\right)_{n \ge 1}\]
converges in distribution in the Gromov--Hausdorff--Prokhorov topology to a limit with a nonzero diameter.
\end{thm}

Our first main theorem thus identifies the correct scale of the random maps and provides a general tightness result; on the other hand, without any assumptions, it cannot give any further information on the subsequential limits. In our second main theorem, we prove that the Brownian sphere and disks appear when there is no macroscopic (face-)degree, in the sense that none of them dominates the others. We let
\[\Delta_n = \max\{k \ge 1 : d_n(k) \ne 0\}\]
be the largest half-degree of an inner face in $\Mn$, which we assume is always larger than or equal to $2$ in order to avoid trivialities.

\begin{thm}\label{thm:convergence_carte_disque}
Assume that 
$\lim_{n \to \infty} \sigma_n^{-1} \varrho_n = \varrho$ for some $\varrho \in [0,\infty)$. 
Then we have the convergence in distribution in the Gromov--Hausdorff--Prokhorov topology
\[\left(V(\mn), \left(\frac{3}{2 \sigma_n}\right)^{1/2} \dgr, \pgr\right) \cvloi (\Bmap^\varrho,\dBmap^\varrho, \pBmap^\varrho)\]
if and only if
\[\lim_{n \to \infty} \sigma_n^{-1} \Delta_n = 0.\]
\end{thm}

Theorem~\ref{thm:convergence_carte_disque} fully recovers~\cite[Theorem~1]{Mar18b} obtained under additional technical assumptions and in a restricted `finite variance' regime when $n^{-1} \sigma_n^2$ converges to a positive and finite limit.
An important point is that, since the characterisation of the limit spaces $\Bmap^\varrho$, showing convergence of discrete models to these objects has become much easier and we stress that, although the statement extends the results of~\cite{LG13, Mie13, BM17}, we rely on these works and we do not claim any independent proof of the convergence of quadrangulations.

The `only if' part of the statement is simple to prove: 
we shall see that an inner face with degree of order $\sigma_n$ has a diameter, in the whole map, of order $\sigma_n^{1/2}$, which either creates a pinch-point or a hole, so the space cannot converge in distribution towards a limit which has the topology of the sphere or the disk. This argument also shows that in the presence of such a macroscopic face, the diameter of the rescaled map is bounded away from zero, which proves that the subsequential limits in Theorem~\ref{thm:tension_cartes} are indeed not reduced to a single point (and without such a large face the maps converge to the Brownian sphere, which is nontrivial as well).

The behaviour drastically changes if $\varrho_n$ is much larger than $\sigma_n$; indeed as shown by Bettinelli~\cite[Theorem~5]{Bet15} for quadrangulations, in this case, the boundary takes over the rest of the map and we obtain in the limit $\CRT_{X^0} = (\CRT_{X^0}, \dCRT_{X^0}, \pCRT_{X^0})$ the Brownian Continuum Random Tree of Aldous~\cite{Ald93} encoded by the standard Brownian excursion $X^0$. 

\begin{thm}\label{thm:convergence_cartes_CRT}
Suppose that $\lim_{n \to \infty} \sigma_n^{-1} \varrho_n = \infty$. 
Then the convergence in distribution
\[\left(V(\mn), (2\varrho_n)^{-1/2} \dgr, \pgr\right) \cvloi (\CRT_{X^0}, \dCRT_{X^0}, \pCRT_{X^0})\]
holds in the Gromov--Hausdorff--Prokhorov topology, where $X^0$ is the standard Brownian excursion.
\end{thm}

Theorems~\ref{thm:convergence_carte_disque} and~\ref{thm:convergence_cartes_CRT} apply directly to models of random maps when all faces have the same degree. For an integer $p \ge 2$, a \emph{$2p$-angulation} denotes a map in which all faces have degree $2p$. In this case $\Delta_n = p$ and $\sigma_n^2 = p(p-1)n$.
Then Le~Gall~\cite[Theorem~1]{LG13} actually proved that for any $p$ fixed, such a $2p$-angulation with $n$ faces sampled uniformly at random converges in distribution towards the Brownian sphere; this was extended to $2p$-angulations with a boundary by Bettinelli \& Miermont~\cite[Corollary~6]{BM17}. Our results allow to let $p$ vary with $n$.

\begin{cor}
\label{cor:convergence_d_ang}
Let $\varrho \in [0,\infty]$ and let $(p_n)_{n \ge 1} \in \{2, 3, \dots\}^\N$ and $(\varrho_n)_{n \ge 1} \in \N^\N$ be any sequences such that $\lim_{n \to \infty} (p_n (p_n-1) n)^{-1/2} \varrho_n = \varrho$. For every $n \ge 1$, let $M^{\varrho_n}_{n, p_n}$ be a uniformly chosen random $2p_n$-angulation with $n$ inner faces and with perimeter $2\varrho_n$.
\begin{enumerate}
\item Suppose that $\varrho < \infty$, then the convergence in distribution
\[\left(V(M^{\varrho_n}_{n, p_n}), \left(\frac{9}{4 p_n (p_n-1) n}\right)^{1/4} \dgr, \pgr\right) \cvloi (\Bmap^\varrho,\dBmap^\varrho, \pBmap^\varrho)\]
holds in the Gromov--Hausdorff--Prokhorov topology.

\item Suppose that $\varrho = \infty$, then the convergence in distribution
\[\left(V(M^{\varrho_n}_{n, p_n}), (2\varrho_n)^{-1/2} \dgr, \pgr\right) \cvloi (\CRT_{X^0}, \dCRT_{X^0}, \pCRT_{X^0})\]
holds in the Gromov--Hausdorff--Prokhorov topology.
\end{enumerate}
\end{cor}

An important point about our main theorems is that the assumptions are very simple to check. We shall illustrate this with so-called size-conditioned \emph{critical $\alpha$-stable Boltzmann maps} in Section~\ref{sec:BGW_Boltzmann}.
This model can be seen as a mixture of our present one, where ($n$ itself and) the sequence $d_n$ is first sampled at random, informally as a conditioned version of some i.i.d. random variables, and then one samples a map uniformly at random given these degrees. The size of a random Boltzmann map is not fixed but one can sample such a map conditioned to have either $n$ faces, or $n$ edges, or $n$ vertices and let $n \to \infty$.

When the index $\alpha$ lies in $(1,2)$, we shall prove that for some deterministic sequence $a_n$ of order $n^{1/\alpha}$ (up to a \emph{slowly varying function}), the ratio $a_n^{-1} \sigma_n$ converges in distribution as $n\to\infty$, so we deduce from Theorem~\ref{thm:tension_cartes} tightness of these maps once rescaled by $a_n^{-1/2} \approx n^{-1/(2\alpha)}$. This model was first studied by Le~Gall \& Miermont~\cite{LGM11} who obtained more information in addition to tightness, although proving the uniqueness of the subsequential limit is still open.

When $\alpha=2$, we shall see that $a_n^{-1} \Delta_n \to 0$ in probability and $a_n^{-1} \sigma_n$ converges in probability to an explicit constant which depends on the law of the degrees and the notion of size (faces, edges, or vertices) where $a_n$ is of order $n^{1/2}$ (again up to a slowly varying function).
Then Theorem~\ref{thm:convergence_carte_disque} implies the convergence of these maps, once rescaled by the factor $a_n^{-1/2} \approx n^{-1/4}$, to a Brownian disk.
In the case of maps without boundary, this recovers~\cite[Theorem~1]{Mar18a}, whereas for maps with a boundary, this extends~\cite[Theorem~5]{BM17} which assumes small exponential moments.

Finally, in the other extreme case $\alpha=1$, relying on the recent work by Kortchemski \& Richier~\cite{KR19}, we prove that there is a unique giant face, of order some other deterministic sequence $|b_n| \approx n$ (still up to a slowly varying function), and this falls into the framework of Theorem~\ref{thm:convergence_cartes_CRT} so the maps converge, once rescaled by a factor $|b_n|^{-1/2} \approx n^{-1/2}$ to the Brownian CRT. A similar result holds for \emph{subcritical maps}, at the scaling exactly $n^{-1/2}$, as first shown by Janson \& Stef\'{a}nsson \cite{JS15}.
We refer to Section~\ref{sec:BGW_Boltzmann} for precise statements.

\subsection{Random trees with a prescribed degree sequence}
\label{sec:intro_arbres}

This model of random maps was inspired by a similar model of random trees, whose scaling limits were first studied by Broutin \& Marckert~\cite{BM14} and extended to forests recently by Lei~\cite{Lei19}.
Let $(\varrho_n)_{n \ge 1}$ and $(d_n)_{n \ge 1}$ be as above, then for each $n \ge 1$ we let $\Tree_{d_n}^{\varrho_n}$ be the set of ordered plane forests with $\varrho_n$ trees and $n$ internal vertices, amongst which $d_n(k)$ have $k$ offspring for every $k \ge 1$. Such a forest always has
\[\edges_n \coloneqq \sum_{k \ge 1} k d_n(k),
\qquad
d_n(0) \coloneqq \varrho_n + \sum_{k \ge 1} (k-1) d_n(k),
\qquad
\vertices_n \coloneqq \sum_{k \ge 0} d_n(k),\]
edges, leaves, and vertices respectively.

Under technical assumptions on the degree sequences $(d_n)_{n\ge 1}$, in particular that $\sigma_n^2 \sim \sigma^2 E_n$ for some $\sigma \in (0,\infty)$ and $\lim_{n \to \infty} \sigma_n^{-1} \Delta_n = 0$, Broutin \& Marckert~\cite{BM14} showed that if $\tn$ is a tree sampled uniformly at random in $\Tree_{d_n}^1$, then the convergence in distribution
\begin{equation}\label{eq:convergence_arbres_BM}
\left(V(\tn), \frac{\sigma_n}{2 \edges_n} \dgr, \pgr\right) \cvloi (\CRT_{X^0}, \dCRT_{X^0}, \pCRT_{X^0})
\end{equation}
holds in the Gromov--Hausdorff--Prokhorov topology, where $X^0$ is the standard Brownian excursion. 
More precisely, the \emph{contour} or \emph{height} process of the rescaled tree converges in distribution towards $X^{0}$.

Under the assumption of no macroscopic degree only, we prove a weaker convergence, in the sense of subtrees spanned by finitely many random vertices, as depicted on Figure~\ref{fig:arbre_reduit}.
Fix $q \ge 1$ and let $u_1, \dots, u_q$ be $q$ i.i.d. uniform random vertices of $\tn$ and keep only these vertices and their ancestors and remove all the other ones; further merge each chain of vertices with only one child in this new tree into a single edge with a length given by the number of edges of the former chain. The resulting tree $\mathscr{R}_{d_n}(q)$ is called a discrete tree with edge-lengths; its combinatorial structure is that of a plane tree with at most $q$ leaves and no vertex with outdegree $1$, so there are only finitely many possibilities, and thus there are a bounded number of edge-lengths to record. We equip the space of trees with edge-lengths with the natural product topology. For $x_1, \dots, x_q$ i.i.d. random points of $\CRT_{X^0}$ sampled from the mass measure $\pCRT_{X^0}$, one can construct similarly a discrete tree with edge-lengths $\mathscr{R}_{X^0}(q)$; its law is described by Aldous~\cite[Section~4.3]{Ald93}.

\begin{figure}[!ht] \centering
\includegraphics[height=10\baselineskip]{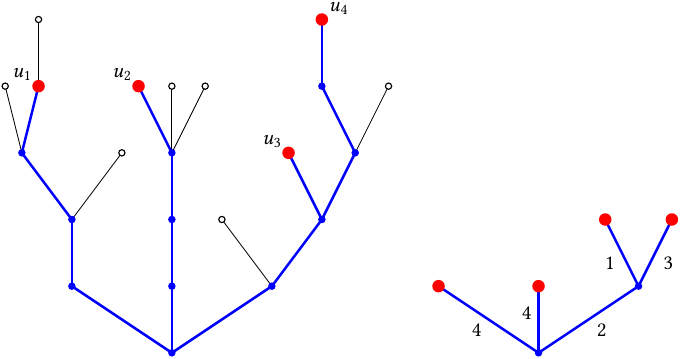}
\caption{Left: a plane tree and four distinguished vertices in red. Right: the associated reduced tree with edge-lengths.}
\label{fig:arbre_reduit}
\end{figure}

\begin{thm}\label{thm:convergence_arbre_reduit}
If $\lim_{n \to \infty} \sigma_n^{-1} \Delta_n = 0$, then for every $q \ge 1$ we have the convergence in distribution
\[\frac{\sigma_n}{2 \edges_n} \mathscr{R}_{d_n}(q) \cvloi \mathscr{R}_{X^0}(q).\]
where $X^0$ is the standard Brownian excursion.
\end{thm}

This result holds more generally when $\varrho_n \sim \varrho \sigma_n$ for some $\varrho \in [0,\infty)$, for the forest $\fn$, viewed as a single tree by attaching all the roots to an extra root vertex. The limit is a different continuum tree which involves a Brownian first-passage bridge from $0$ to $-\varrho$; we refer to Theorem~\ref{thm:marginales_hauteur} for a statement involving the height process of the forest.

Theorem~\ref{thm:convergence_arbre_reduit} characterises the possible scaling limits of the trees $\tn$ in this regime with no macroscopic degrees; in order to extend the convergence from~\eqref{eq:convergence_arbres_BM} and thus obtain an analogue of Theorem~\ref{thm:convergence_carte_disque} for trees and forests, it remains to prove tightness of this sequence (see precisely Equation~25 in~\cite{Ald93}). Here no general tightness result as in Theorem~\ref{thm:tension_cartes} holds and the maximal height of the tree can be much larger than $\edges_n / \sigma_n$. 
Indeed, similarly to Boltzmann maps, size-conditioned Bienaymé--Galton--Watson trees can be thought of as mixtures of this model where $d_n$ is itself random; then when the offspring distribution is subcritical, Kortchemski~\cite{Kor15} proved that the height of the tree grows logarithmically, whereas $\edges_n / \sigma_n$ is of constant order, and the rescaled trees are not tight.
Still, this regime, where the trees exhibit a so-called condensation phenomenon, with a unique huge degree, is extreme and we wonder wether beside such a case tightness holds; in particular, does~\eqref{eq:convergence_arbres_BM} hold as soon as $\lim_{n \to \infty} \sigma_n^{-1} \Delta_n = 0$?

There is a particular case that our method allows to treat, and which seems already new, which is that of regular trees: for $n,k\ge1$, a (strict) \emph{$k$-ary tree} with size $n$ is a tree made of $n$ vertices with $k$ offspring each, and $n(k-1) + 1$ leaves, and no other vertex. With the preceding notation, a $k$-ary tree $T_{n,k}$ with size $n$ sampled uniformly at random has the law of $\tn$ where $d_{n}(k) = n$ and $d_{n}(i) = 0$ for all $i\in \N\setminus\{k\}$, so $\edges_n = k n$ and $\sigma_{n}^{2} = k(k-1)n$.
For such trees, one can adapt our results to prove the analogue of Corollary~\ref{cor:convergence_d_ang}, that is: for any sequence $(k_{n})_{n\ge1} \in \{2, 3, \dots\}^\N$, the convergence in distribution
\begin{equation}\label{eq:arbres_k_aires}
\bigg(V(T_{n,k_{n}}), \sqrt{\frac{k_{n}-1}{4 k_{n} n}} \dgr, \pgr\bigg) \cvloi (\CRT_{X^0}, \dCRT_{X^0}, \pCRT_{X^0})\end{equation}
holds in the Gromov--Hausdorff--Prokhorov topology. Actually, the contour or height process of the rescaled tree converges in distribution towards the standard Brownian excursion $X^{0}$. See Remark~\ref{rem:k_ary} for a further discussion.

\subsection{Strategy of the proof and further discussion}

The study of the random forests $\fn$ is also an important part of the proof of our main results on maps. Indeed, the combination of the bijections from~\cite{BDG04, JS15} relates maps in $\Mn$ with forests in $\Tree_{d_n}^{\varrho_n}$ where each vertex carries an integer label.
Thanks to the work of Le~Gall~\cite{LG13} and Miermont~\cite{Mie13} our main theorems shall easily follow from the study of such random labelled forests and most of this paper is devoted to this.
Let us mention at this point that the bijection from~\cite{BDG04} holds also for maps which are not bipartite and it was used by Le~Gall~\cite{LG13} to prove that random triangulations converge to the Brownian sphere. Recently, Addario-Berry \& Albenque~\cite{ABA21} used it to prove that random $p$-angulations converge to the Brownian sphere for every $p \ge 5$ odd. 
However the labelled trees from~\cite{BDG04} are more complicated to control and extending our present results to the non-bipartite case does not seem technically easy.

The graph distance on the map can be partially encoded by the \emph{label process} which reads the labels of the forest in depth-first search order, and we prove that this process is always tight when suitably rescaled, see Theorem~\ref{thm:tension_etiquettes}. 
This tightness relies only on the so-called \emph{{\L}ukasiewicz path} of the forest, which is a very simple process: up to a discrete Vervaat transform (also known as cyclic shift), the latter is a uniformly random path which makes $d_n(k)$ jumps of size $k-1$ for each $k \ge 0$. This is the simplest example of a random path with exchangeable increments and its asymptotic behaviour is well understood.
From a technical point of view, tightness of the label process requires a precise control of the {\L}ukasiewicz path obtained in particular via Chernoff bounds for martingales. 
Going from tightness of the label process to tightness of the map as in Theorem~\ref{thm:tension_cartes} is then very standard in the theory.

Similarly, Theorems~\ref{thm:convergence_carte_disque} and~\ref{thm:convergence_cartes_CRT} follow easily from the convergence of the label process stated in Theorem~\ref{thm:convergence_etiquettes}. After tightness is proved, it only remains to consider its finite-dimensional marginals. A first step is to prove Theorem~\ref{thm:marginales_hauteur} on the convergence of the forest without the labels, in the weak sense of the subforests spanned by finitely many random vertices; for this we introduce a new \emph{spinal decomposition} which describes the genealogy of such random vertices. This shall allow us to compare the so-called height process and {\L}ukasiewicz path of the forest. 
We want to stress that only this weak convergence of the forest is needed, and not a strong convergence, such as the functional convergence of the contour process used in the previous works on random maps.
Once we control the length of the branches of the subforests, the joint convergence of the label of these random vertices follows by a Central Limit Theorem for independent but not identically distributed random variables; here again we strongly rely on the spinal decomposition.

Finally the last section is devoted to stable Boltzmann maps; in this case the {\L}ukasiewicz path of the associated forest is a conditioned first-passage bridge of a left-continuous random walk. We rely on previous works and familiar techniques in the study of stable Bienaymé--Galton--Watson trees to prove that the random degree distribution fits in our general framework.

This work leaves open several questions on maps (in addition to those on trees and on the non-bipartite case discussed above). Let us only briefly mention that of the asymptotic behaviour of (planar bipartite) maps in the case of large degrees, of order $\sigma_n$, that we are currently investigating. Under suitable assumptions, the {\L}ukasiewicz path converges towards the excursion of a process with exchangeable increments which makes no negative jump. Aldous, Miermont, \& Pitman~\cite{AMP04} constructed the analogue of the height process of the associated `Inhomogeneous Continuum Random Tree' which is a family of random excursions with continuous-path which extends the Brownian excursion. One can then try to define `Inhomogeneous Continuum Random Maps'  by adding random labels on such trees in a similar way as in~\cite{LGM11} and to prove convergence of the discrete maps towards these objects (after extraction of a subsequence). In this regard our tightness results from Theorems~\ref{thm:tension_cartes} and~\ref{thm:tension_etiquettes} represent a first step towards such a convergence.

\subsection{Organisation of the paper}
In Section~\ref{sec:preliminaires}, we first recall the definition of labelled plane forests and their encoding by paths and we briefly discuss the bijection with planar maps; then we state 
the three results whose proof will occupy most of this paper: first Theorem~\ref{thm:convergence_arbre_reduit} on the convergence of reduced forest in Section~\ref{sec:arbres_etiquetes}, and then Theorems~\ref{thm:tension_etiquettes} and~\ref{thm:convergence_etiquettes} on the label process in Section~\ref{sec:enonce_tension_convergence_etiquettes}.
In Section~\ref{sec:convergence_cartes} we prove the main theorems from the introduction by relying on the aforementioned results. Section~\ref{sec:forets} is devoted to the study of the random forests, we prove in particular Theorem~\ref{thm:convergence_arbre_reduit}there. We study the label process in Section~\ref{sec:etiquettes} where we prove Theorems~\ref{thm:tension_etiquettes} and~\ref{thm:convergence_etiquettes}.
Finally in Section~\ref{sec:BGW_Boltzmann} we describe the model of stable Boltzmann planar maps, we state and prove our results on these models by relating them to our general setup.

\subsection*{Acknowledgement}

I am grateful to Igor Kortchemski who spotted an error in a first draft, as well as to all the referees involved for their constructive remarks.

\section{Planar maps as labelled trees}
\label{sec:preliminaires}

As alluded in the introduction, we study planar maps via a bijection with labelled forests. Let us first recall formally the definition of the latter and set the notation we shall need. Then we state the main contributions of this paper on scaling limits of the paths which code these forests; the proofs are differed to Sections~\ref{sec:forets} and~\ref{sec:etiquettes}.

\subsection{The key bijection}
\label{sec:bijection}

A (rooted plane) tree is a planar map with a unique face. We shall interpret it as the genealogical tree of a population. First, the origin of the root edge is the ancestor of the family, denoted by $\varnothing$, then for any given vertex $x$, its neighbour $pr(x)$ closer to the ancestor is its parent, and all the other neighbours are its offspring; the distance of $x$ to the root is its generation and is denoted by $|x|$. A vertex with no offspring is called a leaf, and the other ones are called internal vertices. Note that due to the embedding on the sphere, the offspring of a given internal vertex are ordered from left to right and we denote them by $x1, \dots, xk_x$ where $k_x$ is the offspring number of $x$; finally the root edge points to the left-most offspring of the ancestor. 
For any vertex $x \ne \varnothing$, we let $\chi_x \in \{1, \dots, k_{pr(x)}\}$ be the only index such that $x=pr(x) \chi_x$, which is the relative position of $x$ amongst its siblings. We shall denote by $\llbracket x , y \rrbracket$ the unique geodesic path between $x$ and $y$. Finally, we shall always list the vertices of a tree as $\varnothing = x_0 < x_1 < \dots$ in the so-called \emph{depth-first search} order, which is constructed recursively as follows: suppose that $x_0 < x_1 < \dots < x_i$ are constructed for some $i\ge 0$, then let $j \le i$ denote the largest index such that $x_j$ has an offspring which does not belong to the list $(x_0, \dots, x_i)$ and let $x_{i+1}$ denote the left-most of these missing offspring.

A (plane) forest is a finite ordered list of plane trees; we shall alternatively view a forest as a single tree by attaching all the roots to an extra root vertex (when focusing on its geometry), or as another connected graph by linking two consecutive roots in a chain (when focusing on labels as defined below).

For every $k \ge 1$, let us consider the following set of discrete bridges:
\begin{equation}\label{eq:def_pont}
\mathscr{B}_k^{\ge -1} = \left\{(b_1, \dots, b_k): b_1, b_2-b_1, \dots, b_k-b_{k-1} \in \Z_{\ge-1},\, b_k=0\right\}.
\end{equation}
Then a labelling of a plane forest, with say, $\varrho \ge 1$ trees rooted at the vertices $r_1 < \dots < r_\varrho$ respectively, is a function $\ell$ from its vertices to $\Z$ such that:
\begin{enumerate}
\item the sequence $(\ell(r_1), \ell(r_2), \dots, \ell(r_\varrho))$ belongs to $\mathscr{B}_{\varrho}^{\ge -1}$,
\item for every vertex $x$ with $k_x \ge 1$ offspring, the sequence of increments $(\ell(x1)-\ell(x), \dots, \ell(xk_x)-\ell(x))$ belongs to $\mathscr{B}_{k_x}^{\ge -1}$.
\end{enumerate}
We emphasise the asymmetry of the model: we have $\ell(x1) \ge \ell(x)-1$ in general, but $\ell(xk_x)=\ell(x)$. Also, in the case of a single tree, the first condition reduces to the fact that its root has label $0$.

Recall from the introduction that for every integer $n$, we consider an integer $\varrho_n \in \N$ and a sequence $d_n = (d_n(k))_{k \ge 1} \in \Z_+^{\N}$ which sums up to $n$, and we denote by $\Tree_{d_n}^{\varrho_n}$ the set of all plane forests with $\varrho_n$ trees and $n$ internal vertices, amongst which $d_n(k)$ have $k$ offspring for every $k \ge 1$. Such a forest always has
\[\edges_n \coloneqq \sum_{k \ge 1} k d_n(k),
\qquad
d_n(0) \coloneqq \varrho_n + \sum_{k \ge 1} (k-1) d_n(k),
\qquad
\vertices_n \coloneqq \sum_{k \ge 0} d_n(k),\]
edges, leaves, and vertices respectively.
We let $\LFn$ denote the set of forests in $\Fn$ equipped with a labelling as above. For a single tree, we shall drop the exponent $1$.

Let $\PMn$ be the set of \emph{pointed maps} $(M_n, x_\star)$ where $M_n$ is a map in $\Mn$ and $x_\star$ is a distinguished vertex of $M_n$. If $(M_n, x_\star)$ is a pointed map, then the tip of the root edge is either farther (by one) to $x_\star$ than its origin, in which case the map is said to be \emph{positive} by Marckert \& Miermont~\cite{MM07}, or it is closer (again by one), in which case the map is said to be \emph{negative}.
Let us immediately note that every map in $\Mn$ has $\varrho_n + \sum_{k \ge 1} k d_n(k) = \vertices_n$ edges in total and so $\vertices_n - n + 1 = d_n(0) + 1$ vertices by Euler's formula. 
In particular a uniformly random map in $\Mn$ in which we further distinguish a vertex $x_\star$ independently and uniformly at random has the uniform distribution in $\PMn$. Moreover, half of the $2 \varrho_n$ edges on the boundary are `positively oriented' and half of them are `negatively oriented', so if $\Mn$ is positive, we may re-root it to get a negative map. Therefore it is equivalent to work with random negative maps in $\PMn$ instead of maps in $\Mn$.

Combining the bijections due to Bouttier, Di Francesco, \& Guitter~\cite{BDG04} and to Janson \& Stef\'{a}nsson~\cite{JS15}, we obtain that the set $\LFn$ is in one-to-one correspondence with the set of negative maps in $\PMn$. Let us refer to these papers as well as to~\cite{Mar18a} for a direct construction of the bijection. In a few words the map is constructed from the forest in two steps: first as in the standard Schaeffer bijection and more generally as in~\cite{BDG04} we link every vertex of the forest with the next one in \emph{depth-first search order} (as opposed to the contour order on corners for the other bijections) with a smaller label, or to an extra vertex for those whose label is minimal; then we merge every internal vertex of the forest with its last offspring (recall that $\ell(xk_x)=\ell(x)$); see Figure~\ref{fig:bijection_arbre_carte} for an illustration.
The bijection enjoys the following properties:
\begin{enumerate}
\item The leaves of the forest are in one-to-one correspondence with the vertices different from the distinguished one in the map, and the label of a leaf minus the infimum over all labels, plus one, equals the graph distance between the corresponding vertex of the map and the distinguished vertex.
\item The internal vertices of the forest are in one-to-one correspondence with the inner faces of the map, and the outdegree of the vertex is half the degree of the face.
\item The root face of the map corresponds to the collection of roots of the forest, and the number of trees is half the perimeter of the map.
\end{enumerate}
The third property only holds for negative maps, which is the reason why we restricted ourselves to this case.

\begin{figure}[!ht] \centering
\includegraphics[width=.475\linewidth, page = 2]{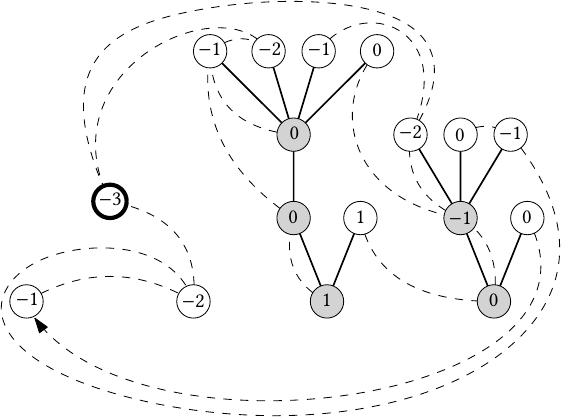}
\quad
\includegraphics[width=.475\linewidth, page = 3]{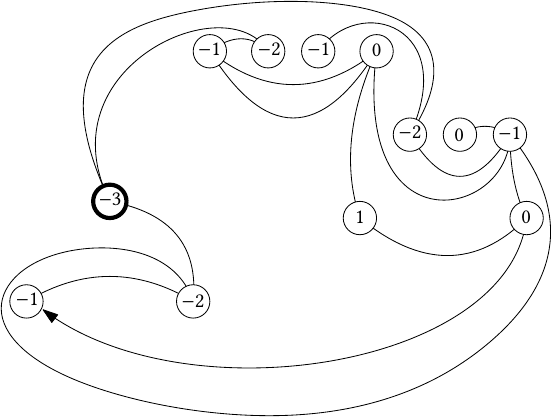}
\caption{Left: a labelled forest in solid lines with an extra vertex labelled $-3$ and in dashed the links between each vertex and the next one with smaller label (in cyclic order).
Right: the negative pointed map obtained by removing the edges of the forest and merging each internal vertex (in gray) with its last offspring.}
\label{fig:bijection_arbre_carte}
\end{figure}

\subsection{Labelled forests and discrete paths}
\label{sec:def_arbres_marche}

Fix a forest $F$ with $\varrho \ge 1$ trees and $V \ge 1$ vertices in total, listed $x_0 < x_1 < \dots < x_{V-1}$ in lexicographical order. It is well known that $F$ is described by each of the following two discrete paths. First, its \emph{{\L}ukasiewicz path} $W = (W(j) ; 0 \le j \le V)$ is defined by $W(0) = 0$ and
\[W(j+1) = W(j) + k_{x_j}-1
\qquad(0 \le j \le V-1).\]
One easily checks that $W(V) = -\varrho$ and $W(j) > -\varrho$ for every $0 \le j \le V-1$. 
Next the \emph{height process} $H = (H(j); 0 \le j \le V)$ is defined by setting $H(V)=0$ for convenience and
\[H(j) = |x_j| \qquad(0 \le j \le V-1);\]
see Figure~\ref{fig:foret_etiquetee} for an illustration.
Clearly the maps $F \mapsto W$ and $F \mapsto H$ are injective, and in order to recover the forest from one of those paths, one may first observe that the infimum of $W$ up to time $j$, or the number of passages of $H$ by $0$ up to time $j$ tells to which tree of the forest the vertex $x_j$ belongs, and then the parent of this vertex is $x_k$, where $k = \sup\{i < j : H(i) < H(j)\}$ for the height process, and $k = \sup\{i < j : W(i) \le W(j)\}$ for the {\L}ukasiewicz path. 
Let us refer to e.g. Le~Gall~\cite{LG05} for a thorough discussion of such encodings.

In the case of a labelled forest, we encode the labels into the \emph{label process} given by $L(V)=0$ for convenience and
\[L(j) = \ell(x_j) \qquad(0 \le j \le V-1);\]
see Figure~\ref{fig:foret_etiquetee}.
The labelled forest is encoded both by the pair $(H, L)$ and by the pair $(W,L)$.

\begin{figure}[!ht] \centering
\begin{minipage}{.425\linewidth}
\includegraphics[width=\linewidth, page = 8]{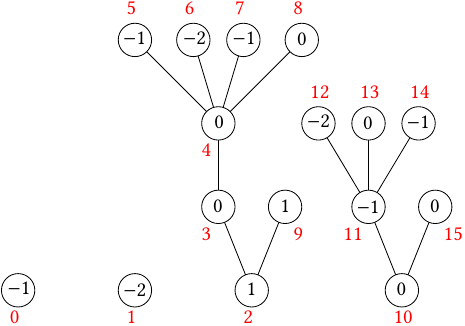}
\end{minipage}
\hfill
\begin{minipage}[c]{.55\linewidth}
\includegraphics[width=\linewidth, page = 9]{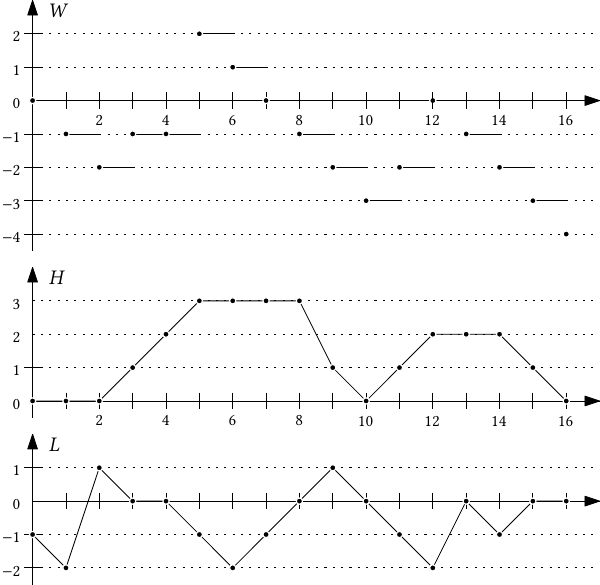}
\end{minipage}
\caption{Left: A labelled forest on the left, with labels indicated on the nodes and next to them in red their lexicographical order.
Right, from top to bottom: its {\L}ukasiewicz path, its height process, and its label process.}
\label{fig:foret_etiquetee}
\end{figure}

Without further notice, throughout this work, every {\L}ukasiewicz path shall be viewed as a step function, jumping at integer times, whereas the height and label processes shall be viewed as continuous functions after interpolating linearly between integer times.

Let us end with another useful construction of the {\L}ukasiewicz path.
The next lemma, whose proof is left as an exercise, gathers some deterministic results that we shall need. In order to simplify the notation, we shall often identify the vertices of a tree with their index in the lexicographic order: here if $x$ is the $i$'th vertex of a tree whose {\L}ukasiewicz path is $W$, then we write $W(x)$ for $W(i)$.

\begin{lem}\label{lem:codage_marche_Luka}
Let $T$ be a plane tree and $W$ be its {\L}ukasiewicz path. Fix a vertex $x \in T$, then
\[W(x k_x) = W(x),
\qquad
W(xj') = \inf_{[xj, xj']} W,
\qquad\text{and}\qquad
j' - j = W(xj) - W(xj')\]
for every $1 \le j \le j' \le k_x$.
\end{lem}

For a vertex $x$ of a tree $T$, let us denote by $\LL(x)$ and $\RR(x)$ respectively the number of vertices $y$ whose parent is a strict ancestor of $x$ and which lie strictly to the left, respectively to the right, of the ancestral line $\llbracket \varnothing, x\llbracket$; then we put $\LR(x) = \LL(x) + \RR(x)$. Then a consequence of the preceding lemma is that, in a tree, we have $\RR(x) = W(x)$. In the case of a forest, we define $\LL(x)$ and $\RR(x)$ (and so $\LR(x)$) as the same quantities in the tree containing $x$, and then we have more generally
\begin{equation}\label{eq:def_X_discret}
\RR(x) = W(x) - \min_{0 \le y \le x} W(y).
\end{equation}
For example, if $x$ is the vertex number 7 in Figure~\ref{fig:foret_etiquetee}, then $\RR(x) = W(7) - \min_{0 \le i \le 7} W(i) = 2$, which corresponds to the vertices 8 and 9.

\subsection{Geometry of random forests}
\label{sec:arbres_etiquetes}

The first key step to understand the asymptotic behaviour of a random map in $\PMn$ is to study the associated random labelled forest $(\fn, \ell)$ in $\LFn$. Let us first discuss the behaviour of the forest, and more precisely of its \L ukasiewicz path and height process, before stating our main results on the labels. For every integer $n$, we are thus given a degree sequence $d_n = (d_n(k))_{k \ge 0} \in \Z_+^{\Z_+}$ and $\fn$ is sampled uniformly at random in the set $\Fn$ of forests which possess $\varrho_n = \sum_{k \ge 0} (1-k) d_n(k)$ trees and $\vertices_n = \sum_{k \ge 0} d_n(k)$ vertices, amongst which $d_n(k)$ have $k$ children for every $k \ge 0$.
Recall also the notation
$\edges_n \coloneqq \sum_{k \ge 1} k d_n(k) = \vertices_n - \varrho_n$ for the number of edges, $\Delta_n \coloneqq \max\{k \ge 0 : d_n(k) > 0\}$ the largest offspring, and finally
$\sigma_n^2 \coloneqq \sum_{k \ge 0} k (k-1) d_n(k)$,
so $\sigma_n^2/\vertices_n$ is the second factorial moment of a vertex chosen uniformly at random.

It is well known that the {\L}ukasiewicz path $\Wfn$ of our random forest $\fn$ can be constructed as a cyclic shift of a bridge, also called the discrete Vervaat transform, see e.g. Pitman~\cite[Chapter~6]{Pit06}. More precisely, let $\Bfn = (\Bfn(i))_{0 \le i \le \vertices_n}$ be a discrete path sampled uniformly at random amongst all those started from $\Bfn(0) = 0$ and which make exactly $d_n(k)$ jumps with value $k-1$ for every $k \ge 0$, so $\Bfn$ is a bridge from $0$ to $\Bfn(\vertices_n) = \sum_{k \ge 0} (k - 1) d_n(k) = - \varrho_n$. Independently, sample $p_n$ uniformly at random in $\{0, \dots, \varrho_n-1\}$ and set 
\[i_n = \inf\left\{i \in\{1, \dots, \vertices_n\} : \Bfn(i) = p_n + \inf_{1 \le j \le \vertices_n} \Bfn(j)\right\}.\]
Then $\Wfn$ has the law of 
the concatenation
\[\left((\Bfn(i_n+k) - \Bfn(i_n))_{0 \le k \le \vertices_n-i_n}, (\Bfn(k) + \Bfn(\vertices_n) - \Bfn(i_n))_{1 \le k \le i_n}\right).\]
Moreover, the time $i_n$ has the uniform distribution on $\{1, \dots, \vertices_n\}$ and is independent of the cyclically shifted path $\Wfn$. Given the excursion $\Wfn$ and $i_n$, one recovers the bridge $\Bfn$ by the same cyclic shift.

The next result states that vertices with a given outdegree are homogeneously spread in the forest in the sense that, asymptotically, the number of vertices with degree in a given set of integers, when read in depth-first search order, grows linearly. We shall apply it with $A = \{-1\}$ in order to deal with maps.
For a subset $A \subset \Z_{\ge -1} = \{-1, 0, 1, 2, \dots\}$ and a real number $r \in [0, \vertices_n]$, let $\Lambda_A(r; \Wfn)$ denote the number of jumps with value in $A$ amongst the first $\lfloor r\rfloor$ jumps of $\Wfn$, and consider its inverse $\zeta_A(t; \Wfn) = \inf\{r : \Lambda_A(r; \Wfn) = \lfloor t\rfloor\}$ for all $t \in [0, \infty)$ the time needed to see at least $\lfloor t\rfloor$ jumps with value in $A$. Let us set $d_n(A + 1) = \sum_{k \in A} d_n(k+1)$.

\begin{lem}\label{lem:proportion_feuilles}
Fix any subset $A \subset \Z_{\ge -1}$, if $d_n(A+1) \to \infty$, then we have the convergence in probability
\[\left(\frac{\Lambda_A(\vertices_n t; \Wfn)}{d_n(A+1)}, \frac{\zeta_A(d_n(A+1) t; \Wfn)}{\vertices_n}\right)_{t \in [0,1]} \cvproba (t, t)_{t \in [0,1]}.\]
\end{lem}

\begin{proof}
Note first that since the two functions are inverse of one another and non-decreasing, the two convergences are equivalent. Second, by constructing the bridge $\Bfn$ by cyclically shifting $\Wfn$ at an independent uniformly random time, it is sufficient to prove the claim when the {\L}ukasiewicz path $\Wfn$ is replaced by the bridge $\Bfn$. 

For an integer $j \le V_n$, the quantity $\Lambda_A(j; \Bfn)$ is the sum of \emph{dependent} Bernoulli random variables: we start with an urn containing $\vertices_n$ balls in total, amongst which $d_n(A+1)$ are labelled $A$, we sample $j$ balls without replacement and count the number of balls labelled $A$ picked. If those picks were with replacement, then the first convergence would simply follow from the law of large numbers; actually when sampling without replacement, $\Lambda_A(j; \Bfn)$ is even more concentrated around its mean $j d_n(A+1)/\vertices_n$ and e.g. the Chernoff bound applies in the same way it does with i.i.d. Bernoulli random variables (this is already shown in Hoeffding’s seminal paper~\cite{Hoe63}); we deduce the convergence of $t \mapsto \Lambda_A(\vertices_n t; \Bfn) / d_n(A+1)$ to the identity and thus also that of $t \mapsto \zeta_A(d_n(A+1) t; \Bfn)/\vertices_n$.
\end{proof}

Let us next focus on the convergence of $\Wfn$. For $\varrho \in [0,\infty)$, let us denote by $B^\varrho = (B^\varrho_t)_{t \in [0,1]}$ the standard Brownian bridge from $0$ to $-\varrho$ with duration $1$. Analogously one can construct $X^\varrho$ the \emph{first-passage Brownian bridge} from $0$ to $-\varrho$ (which reduces to the standard Brownian excursion when $\varrho = 0$) by cyclically shifting $B^\varrho$, see \cite[Theorem~7]{BCP03}.
The starting point of this work is the following result, which extends \cite[Lemma~7]{BM14} and \cite[Theorem~1.6]{Lei19}.

\begin{prop}\label{prop:convergence_Luka}
Assume that $\lim_{n \to \infty} \sigma_n = \infty$ and $\lim_{n \to \infty} \sigma_n^{-1} \varrho_n = \varrho$ for some $\varrho \in [0, \infty]$.
\begin{enumerate}
\item Suppose that $\varrho < \infty$. Then from every sequence of integers, one can extract a subsequence along which the sequence $\sigma_n^{-1} \Bfn(\lfloor \vertices_n \cdot \rfloor)$ converges in distribution in the Skorokhod's $J_1$ topology.

\item Furthermore this sequence converges in distribution towards $B^\varrho$ if and only if $\lim_{n \to \infty} \sigma_n^{-1} \Delta_n = 0$. In this case, the sequence of processes $\sigma_n^{-1} \Wfn(\lfloor \vertices_n \cdot \rfloor)$ converges in distribution towards $X^\varrho$.

\item Suppose that $\varrho = \infty$, then both processes $\varrho_n^{-1} \Bfn(\lfloor \vertices_n \cdot \rfloor)$ and $\varrho_n^{-1} \Wfn(\lfloor \vertices_n \cdot \rfloor)$ converge in probability towards $t \mapsto -t$.
\end{enumerate}
\end{prop}

From~\eqref{eq:def_X_discret}, by continuity, when $\lim_{n \to \infty} \sigma_n^{-1} \Delta_n = 0$ and $\varrho < \infty$, the process $(\sigma_n^{-1} \RR(\lfloor \vertices_n t\rfloor))_{t \in [0,1]}$ converges in distribution towards $\tX = (X^\varrho_t - \inf_{s \in [0,t]} X^\varrho_s)_{t \in [0,1]}$.

\begin{proof}
Let us first start with the case $\varrho < \infty$.
For every $1 \le i \le \vertices_n$, let $b_i = \Bfn(i) - \Bfn(i-1)$.
Then $\sum_{i = 1}^{\vertices_n} b_i = \sum_{k \ge 0} (k - 1) d_n(k) = -\varrho_n$ and
\[\sum_{i = 1}^{\vertices_n} b_i^2
= \sum_{k \ge 0} (k - 1)^2 d_n(k)
= \sum_{k \ge 0} k (k - 1) d_n(k) - \sum_{k \ge 0} (k - 1) d_n(k)
= \sigma_n^2 + \varrho_n.\]
Therefore $\sum_{i = 1}^{\vertices_n} (b_i + \vertices_n^{-1} \varrho_n) = 0$ and 
$\sum_{i = 1}^{\vertices_n} (b_i + \vertices_n^{-1} \varrho_n)^2 
= \sigma_n^2 + \varrho_n - \vertices_n^{-1} \varrho_n^2 = \sigma_n^2 (1 + o(1))$.
Then the sequence given by
\[x_i = \frac{b_i + \vertices_n^{-1} \varrho_n}{(\sigma_n^2 + \varrho_n - \vertices_n^{-1} \varrho_n^2)^{1/2}},
\qquad 1 \le i \le \vertices_n,\]
is called a `normalised urn' by Aldous~\cite[Chapter~20]{Ald85}. For every $t \in [0,1]$, let us define
\[S_n(t) = \sum_{i \le \lfloor \vertices_n t \rfloor} x_i = 
\frac{\Bfn(\lfloor \vertices_n t \rfloor)}{\sigma_n (1 + o(1))} + \frac{\varrho_n \lfloor \vertices_n t \rfloor}{\vertices_n \sigma_n (1 + o(1))}.\]
Then by~\cite[Proposition~20.3]{Ald85}, the sequence $(S_n)_{n \ge 1}$ is always tight in the Skorokhod's $J_1$ topology, this implies the first claim. Furthermore, by~\cite[Theorem~20.7]{Ald85}, this sequence converges in distribution towards a Brownian bridge $B^0$ as soon as there is no macroscopic jump, i.e. $\lim_{n \to \infty} \sigma_n^{-1} \Delta_n = 0$. Of course if there is such a large jump, a limit cannot have continuous paths. In the absence of jumps, we deduce that $\sigma_n^{-1} \Bfn(\lfloor \vertices_n \cdot \rfloor)$ converges in distribution towards $(-\varrho t + B^0_t)_{t \in [0,1]}$ which has the same law as $B^\varrho$. Since $\Wfn$ and $X^\varrho$ are obtained by cyclically shifting $\Bfn$ and $B^\varrho$ respectively, this implies further the convergence in distribution of $(\sigma_n^{-1} \Wfn(\lfloor \vertices_n t \rfloor))_{t \in [0,1]}$ towards $X^\varrho$ from the continuity of this operation, see Lei~\cite[Section~4]{Lei19}.

Let us finally suppose that $\varrho_n \gg \sigma_n$. Then similarly, for every $t \in [0,1]$, let us set
\[S'_n(t) = \frac{\sigma_n}{\varrho_n} S_n(t)
= \frac{\Bfn(\lfloor \vertices_n t \rfloor)}{\varrho_n (1 + o(1))} + \frac{\varrho_n \lfloor \vertices_n t \rfloor}{\vertices_n \varrho_n (1 + o(1))}.\]
Since the sequence $(S_n)_{n \ge 1}$ is tight, then $S'_n$ converges in probability to the null process and therefore $\varrho_n^{-1} \Bfn(\lfloor \vertices_n \cdot \rfloor)$ converges in probability towards $t \mapsto -t$. The convergence of $\varrho_n^{-1} \Wfn(\lfloor \vertices_n \cdot \rfloor)$ follows again by cyclic shift.
\end{proof}

Let us next turn to the height process $\Hfn$ associated with the random forest $\fn$.
Let us fix a continuous function $g : [0,1] \to \R$, then for every $0 \le s \le t \le 1$, set
\[\dCRT_g(s,t) = \dCRT_g(t,s) = g(s) + g(t) - 2 \min_{r \in [s,t]} g(r).\]
One easily checks that $\dCRT_g$ is a continuous pseudo-metric on $[0,1]$. Consider the quotient space $\CRT_g = [0,1] / \{\dCRT_g = 0\}$; we let $\pi_g$ be the canonical projection $[0,1] \to \CRT_g$, then $\dCRT_g$ induces a metric on $\CRT_g$ that we still denote by $\dCRT_g$, and the Lebesgue measure on $[0,1]$ induces a measure $\pCRT_g$ on $\CRT_g$. The space $\CRT_g = (\CRT_g, \dCRT_g, \pCRT_g)$ is a so-called compact measured real-tree, naturally rooted at $\pi_g(0)$.
When $g = X^0$, the space $\CRT_{X^0}$ is the celebrated Brownian Continuum Random Tree of Aldous~\cite{Ald93}. More generally, for $\varrho \in (0,\infty)$, 
the space $\CRT_{X^\varrho}$ describes a forest of Brownian trees, attached by their root along an interval of length $\varrho$ as in~\cite{BM17}.
For every $t \in [0,1]$, set $\underlinep{X}^\varrho_t = \min_{0 \le s \le t} X^\varrho_s$ and then $\tX_t = X^\varrho_t-\underlinep{X}^\varrho_t$; then the continuum tree $\CRT_{\tX}$ describes the same forest of Brownian trees, but now after all the roots have been merged together. 

Similarly, in the discrete setting, let us view a plane forest as a single tree by attaching all the roots to an extra root vertex. The next proposition shows that, under the assumption of no macroscopic degree, the forest $\fn$ spanned by i.i.d. uniform random vertices, i.e. the smallest connected subset of $\fn$ viewed as a tree containing these vertices,
converges towards the analogue for $\CRT_{\tX}$ and i.d.d. points sampled from its mass measure $\pCRT_{\tX}$. This result implies Theorem~\ref{thm:convergence_arbre_reduit} from the introduction.

\begin{thm}\label{thm:marginales_hauteur}
Assume that $\lim_{n \to \infty} \sigma_n^{-1} \varrho_n = \varrho$ for some $\varrho \in [0,\infty)$ and that $\lim_{n \to \infty} \sigma_n^{-1} \Delta_n = 0$. 
Fix $q \ge 1$ and let $U_1, \dots, U_q$ be i.i.d. uniform random variables in $[0,1]$ independent of the rest and denote by $0 = U_{(0)} < U_{(1)} < \dots < U_{(q)}$ their ordered statistics. Then the convergence in distribution
\[\frac{\sigma_n}{2 \edges_n} \left(\Hfn(\vertices_n U_{(i)}), \inf_{U_{(i-1)} \le t \le U_{(i)}} \Hfn(\vertices_n t)\right)_{1 \le i \le q}
\cvloi
\left(\tX_{U_{(i)}}, \inf_{U_{(i-1)} \le t \le U_{(i)}} \tX_t\right)_{1 \le i \le q}\]
holds jointly with that of $\sigma_n^{-1} \Wfn(\vertices_n \cdot)$ towards $X^\varrho$ in Proposition~\ref{prop:convergence_Luka}.
\end{thm}

The proof is more involved and is differed to Section~\ref{sec:marginales_hauteur}. We already mentioned that the question of tightness of the process $\Hfn$ (or even its maximum, as discussed in the introduction), once suitably rescaled, was not easy; at least we can say that it is simply not true in general, 
but we wonder wether it is the case under the assumptions of Theorem~\ref{thm:marginales_hauteur}.

\subsection{Random labelled forests}
\label{sec:enonce_tension_convergence_etiquettes}

Let us next consider random labelled forests $(\fn, \ell)$ in $\LFn$. The law of the latter can be constructed in the following way: first $\fn$ is sampled uniformly at random in $\Fn$ as previously, and then given this forest, we sample the label of the roots uniformly at random in $\mathscr{B}_{\varrho_n}^{\ge -1}$ and, independently for every branchpoint with, say, $k \ge 1$ offspring, the label increments between these offspring and the branchpoint are sampled uniformly at random in $\mathscr{B}_k^{\ge -1}$.
Let $\Lfn$ be the associated label process.

\begin{thm}
\label{thm:tension_etiquettes}
Fix any sequence $(\varrho_n)_{n \ge 1}$ and any sequence of degree sequence $(d_n)_{n \ge 1}$.
Then from every increasing sequence of integers, one can extract a subsequence along which the label processes
\[\left((\sigma_n + \varrho_n)^{-1/2} \Lfn(\vertices_n t)\right)_{t \in [0,1]}\]
converge for the uniform topology.
\end{thm}

In order to deduce Theorems~\ref{thm:convergence_carte_disque} and~\ref{thm:convergence_cartes_CRT} we need to identify these subsequential limits.
We shall need to deal with the root vertices of $\fn$ separately; let us therefore denote by $(\bfn(k))_{1 \le k \le \varrho_n}$ the labels of the roots of $\fn$, and set $\bfn(0) = \bfn(\varrho_n) = 0$. Recall that $\Wfn$ is the {\L}ukasiewicz path of $\fn$, for every $0 \le k \le \vertices_n$, let us set $\infWfn(k) = \min_{0 \le i \le k} \Wfn(i)$ and write
\begin{equation}\label{eq:decomposition_labels}
\Lfn(k) = \tLfn(k) + \bfn(1-\infWfn(k)).
\end{equation}
Then $\bfn(1-\infWfn(k))$ gives the value of the label of the root vertex of the tree containing the $k$'th vertex, so $\tLfn(k)$ gives the value of the label of the $k$'th vertex minus the value of the label of the root vertex of its tree; in other words, $\tLfn$ is the concatenation of the label process of each tree taken individually, so where all labels have been shifted so that each root receives label $0$. The point is that, conditionally given $\fn$, the two processes $\tLfn(\cdot)$ and $\bfn(1-\infWfn(\cdot))$ are independent.

The Brownian sphere and disks are described similarly by `continuum labelled trees' which we next recall, following Bettinelli \& Miermont~\cite{BM17}. Fix $\varrho \in [0, \infty)$ and recall that $X^\varrho$ denotes the standard Brownian first passage bridge from $0$ to $-\varrho$, and that for every $t \in [0,1]$, we set $\underlinep{X}^\varrho_t = \min_{0 \le s \le t} X^\varrho_s$ and then $\tX_t = X^\varrho_t-\underlinep{X}^\varrho_t$; the processes $X^\varrho$ and $\tX$ each encode in a different way a Brownian forest as explained previously.
For every $y \in [0,\varrho]$, let us set $\tau_y = \inf\{t \in [0,1] : X^\varrho_t = -y\}$.
We construct next another process $Z^\varrho = (Z^\varrho_t)_{t \in [0,1]}$ on the same probability space as $X^\varrho$. First, conditionally on $X^\varrho$, let $\tildep{Z}^\varrho$ be a centred Gaussian process with covariance 
\[\Esc{\tildep{Z}^\varrho_s \tildep{Z}^\varrho_t}{X^\varrho} = \min_{r \in [s,t]} \tX_r
\qquad\text{for every}\qquad 0 \le s \le t \le 1.\]
It is classical that $\tildep{Z}^\varrho$ admits a continuous version and, without further notice, we shall work throughout this paper with this version. 
In the case $\varrho = 0$, we simply set $Z^0 = \tildep{Z}^0$. If $\varrho > 0$, independently of $\tildep{Z}^\varrho$, let $\bb^\varrho$ be a standard Brownian bridge from $0$ to $0$ with duration $\varrho$, which has the law of $(\varrho^{1/2} \bb(\varrho t))_{t \in [0,1]}$ where $\bb$ is a standard Brownian bridge on the time interval $[0,1]$, and set
\[Z^\varrho_t = \tildep{Z}^\varrho_t + \sqrt{3} \cdot \bb^\varrho_{- \underlinep{X}^\varrho_t}
\qquad\text{for every}\qquad 0 \le t \le 1.\]
This construction is the continuum analogue of the decomposition of the process $\Lfn$ in~\eqref{eq:decomposition_labels}.
Observe that, almost surely, $Z^\varrho_s = Z^\varrho_t$ whenever $\dCRT_{X^\varrho}(s, t) = 0$ so $Z^\varrho$ can be seen as a process indexed by $\CRT_{X^\varrho}$ by setting $Z^\varrho_x = Z^\varrho_t$ if $x = \pi_{X^\varrho}(t)$. We interpret $Z^\varrho_{x}$ as the label of an element $x \in \CRT_{X^\varrho}$; the pair $(\CRT_{X^\varrho}, (Z^\varrho_x; x \in \CRT_{X^\varrho}))$ is a continuum analogue of labelled plane forests.

For $\varrho = 0$, the process $Z^0$ is interpreted as a Brownian motion on the Brownian tree $\CRT_{X^0}$ started from $0$ on the root.
For $\varrho > 0$, the space $\CRT_{X^\varrho}$ is interpreted as a collection of Brownian trees glued at their root along the interval $[0,\varrho]$; for every $y \in [0,\varrho]$, it holds that $- \underlinep{X}^\varrho_{\tau_y} = y$ so each point $\pi_{X^\varrho}(\tau_y)$ `at position $y$' on this interval receives label $\sqrt{3} \, \bb_{y}$ and then the labels on each of the trees $\pi_{X^\varrho}((\tau_{y-}, \tau_y])$ evolve like independent Brownian motions on them. The Brownian sphere and disks are constructed from these processes, see Section~\ref{sec:carte_brownienne} below. The second main result on the label process is the following.

\begin{thm}
\label{thm:convergence_etiquettes}
Suppose that 
$\lim_{n \to \infty} \sigma_n^{-1} \varrho_n = \varrho$ for some $\varrho \in [0,\infty]$.
\begin{enumerate}
\item\label{thm:convergence_etiquettes_arbre} If $\varrho = \infty$, then the convergence in distribution
\[\left((2\varrho_n)^{-1/2} \Lfn(\vertices_n t); t \in [0,1]\right) \cvloi (\bb_t; t \in [0,1])\]
holds in $\mathscr{C}([0,1],\R)$.

\item\label{thm:convergence_etiquettes_carte_disque} 
If $\varrho < \infty$ and furthermore $\lim_{n \to \infty} \sigma_n^{-1} \Delta_n = 0$, then the convergence in distribution
\[\left(\left(\frac{3}{2\sigma_n}\right)^{1/2} \Lfn(\vertices_n t) ; t \in [0,1]\right) \cvloi (Z^\varrho_t; t \in [0,1])\]
holds in $\mathscr{C}([0,1],\R)$.
\end{enumerate}
\end{thm}

Let us point out that in Theorem~\ref{thm:convergence_etiquettes}~\eqref{thm:convergence_etiquettes_arbre}, when $\lim_{n \to \infty} \sigma_n^{-1} \varrho_n = \infty$, we have more precisely:
\[\left((2\varrho_n)^{-1/2} \left(\bfn(\varrho_n t), \tLfn(\vertices_n t)\right); t \in [0,1]\right) \cvloi ((\bb_t, 0); t \in [0,1]),\]
which, combined with Proposition~\ref{prop:convergence_Luka} and the decomposition~\eqref{eq:decomposition_labels}, implies the claim of the theorem.

\section{Convergence of random maps}
\label{sec:convergence_cartes}

The proof of Theorems~\ref{thm:marginales_hauteur},~\ref{thm:tension_etiquettes}, and~\ref{thm:convergence_etiquettes} on labelled forests will occupy most of this paper. Before proving them, let us first in this section deduce from them the results on random maps stated in Theorems~\ref{thm:tension_cartes}, \ref{thm:convergence_carte_disque} and~\ref{thm:convergence_cartes_CRT}. The argument finds its root in the work of Le~Gall~\cite{LG13} and has already been adapted in many contexts~\cite{LG13, Abr16, BJM14, BM17, ABA17, ABA21, Mar18b, Mar18a} so we shall be very brief and refer to the preceding references for details.

Let us first define the topology we use in these theorems. Recall that we view a planar map as a compact metric space equipped with a Borel probability measure. 
In words, two such spaces $(X, d_X, p_X)$ and $(Y, d_Y, p_Y)$ are close to each other if one can find a subset of each which carries most of the mass and which are close to be isometric. 
Formally, a \emph{correspondence} between $X$ and $Y$ is a subset $R \subset X \times Y$ such that for every $x \in X$, there exists $y \in Y$ such that $(x,y) \in R$ and vice-versa. The \emph{distortion} of $R$ is defined as
\[\mathrm{dis}(R) = \sup\left\{\left|d_X(x,x') - d_Y(y,y')\right| ; (x,y), (x', y') \in R\right\}.\]
Then the \emph{Gromov--Hausdorff--Prokhorov distance} between these spaces is the infimum of all the values $\varepsilon > 0$ such that there exists a coupling $\nu$ between $p_X$ and $p_Y$ and a compact correspondence $R$ between $X$ and $Y$ such that
\[\nu(R) \ge 1-\varepsilon \qquad\text{and}\qquad \mathrm{dis}(R) \le 2 \varepsilon.\]
This is only a pseudo-distance, but after taking the quotient by measure-preserving isometries, one gets a genuine distance which is separable and complete, see Miermont~\cite[Proposition~6]{Mie09}.

Recall that starting from a uniformly random map in $\Mn$, we may sample a vertex $x_\star$ independently and uniformly at random to obtain a uniformly random pointed map in $\PMn$, which we may then re-root at one of the $\varrho_n$ possible edges on the boundary chosen uniformly at random which make this pointed map $(\mn, x_\star)$ negative. Let $\mn\setminus\{x_\star\}$ be the metric measured space given by the vertices of $\mn$ different from $x_\star$, their graph distance \emph{in $\mn$} and the uniform probability measure and note that the Gromov--Hausdorff--Prokhorov distance between $\mn$ and $\mn\setminus\{x_\star\}$ is bounded by one so it suffices to prove our claims for $\mn\setminus\{x_\star\}$. This will enable us to rely on the bijection with a labelled forest $(\fn, \ell)$.

We shall implicitly assume that $d_{n}(0) \to \infty$, otherwise we may extract a subsequence along which the number of vertices in our maps is uniformly bounded, so if one removes all the double edges (i.e. glue together both sides of a face with degree $2$), which play no role in the geometry of the associated metric measured space, then only a bounded number of edges remain and tightness is clear.

\subsection{Tightness of planar maps}
\label{sec:tension_cartes}

Let us first construct the possible subsequential limits of $\mn$, relying on a subsequential limit of the label process provided by Theorem~\ref{thm:tension_etiquettes} before proving Theorem~\ref{thm:tension_cartes}. 

Recall that in the bijection relating $(\mn, x_\star)$ to $(\fn, \ell)$, the vertices of the former different from $x_\star$ correspond to the leaves of the latter, and the internal vertices of $\fn$ are identified with their last child. Therefore, for every vertex $x$ of $\fn$, we let $\varphi(x)$ be the vertex of $\mn\setminus\{x_\star\}$ in one-to-one correspondence with the right-most leaf amongst the descendants of $x$ in $\fn$. Let us list the vertices of $\fn$ as $x_0 < x_1 < \dots < x_{\vertices_n-1}$ in lexicographical order, set $x_{\vertices_n} = x_{\vertices_n-1}$, and for every $i,j \in \{0, \dots, \vertices_n\}$, let us set
\[d_n(i,j) = \dgr(\varphi(x_i), \varphi(x_j)),\]
where $\dgr$ is the graph distance in $\mn$. We then extend $d_n$ to a continuous function on $[0, \vertices_n]^2$ by `bilinear interpolation' on each square of the form $[i,i+1] \times [j,j+1]$ as in~\cite[Section~2.5]{LG13} or~\cite[Section~7]{LGM11}. For every $0 \le s \le t \le 1$, let us set
\[d_{(n)}(s, t) = (\sigma_n + \varrho_n)^{-1/2} d_n(\vertices_n s, \vertices_n t)
\quad\text{and}\quad
L_{(n)}(t) = (\sigma_n + \varrho_n)^{-1/2} \Lfn(\vertices_n t).\]

For a continuous function $g : [0, 1] \to \R$, let us set for every $0 \le s \le t \le 1$,
\begin{equation}\label{eq:distance_fonction_continue}
D_g(s,t) = D_g(t,s) = g(s) + g(t) - 2 \max\left\{\min_{r \in [s, t]} g(r); \min_{r \in [0, s] \cup [t, 1]} g(r)\right\}.
\end{equation}
For $0 \le i < j \le \vertices_n$, let $[i, j]$ denote the set of integers from $i$ to $j$, and let $[j,i]$ denote $[j, \vertices_n] \cup [0, i]$. Recall that we construct our map from a labelled forest, using a Schaeffer-type bijection; following the chain of edges drawn starting from two points of the forest to the next one with smaller label until they merge, one obtains the following upper bound on distances:
\begin{equation}\label{eq:borne_sup_cactus}
d_n(\vertices_n s, \vertices_n t) \le D_{\Lfn}(\vertices_n s, \vertices_n t) + 2,
\end{equation}
see Le~Gall~\cite[Lemma~3.1]{LG07} for a detailed proof in a different context.

According to Theorem~\ref{thm:tension_etiquettes}, from every increasing sequence of integers, one can extract a subsequence along which the processes $L_{(n)}$ converge in distribution to some limit process, say $L$. From the previous bound, one can extract a further subsequence along which we have
\begin{equation}\label{eq:convergence_distances_sous_suite}
\left(L_{(n)}(t), D_{L_{(n)}}(s, t), d_{(n)}(s, t)\right)_{s,t \in [0,1]}
\cvloi
(L_t, D_L(s,t), d_\infty(s,t))_{s,t \in [0,1]},
\end{equation}
where $(d_\infty(s,t))_{s,t \in [0,1]}$ depends a priori on the subsequence and, by~\eqref{eq:borne_sup_cactus}, satisfies $d_\infty \le D_L$, see~\cite[Proposition~3.2]{LG07} for a detailed proof in a similar context.

The fonction $d_\infty$ is continuous on $[0,1]^2$ and is a pseudo-distance (as limit of $d_{(n)}$). 
Let then $M_\infty$ be the quotient $[0,1] / \{d_\infty=0\}$ equipped with the metric induced by $d_\infty$, which we still denote by $d_\infty$. We let $\Pi_\infty$ be the canonical projection from $[0,1]$ to $M_\infty$ which is continuous (since $d_\infty$ is) so $(M_\infty, d_\infty)$ is a compact metric space, which finally we endow with the Borel probability measure $p_\infty$ given by the push-forward by $\Pi_\infty$ of the Lebesgue measure on $[0,1]$.

Theorem~\ref{thm:tension_cartes} directly follows from the following result.

\begin{prop}\label{prop:limite_sous_suite_carte}
On a subsequence along which~\eqref{eq:convergence_distances_sous_suite} holds we have the convergence
\[\left(\Vmn\setminus\{x_\star\}, (\sigma_n + \varrho_n)^{-1/2} \dgr, \pgr\right) \cvloi (M_\infty, d_\infty, p_\infty)\]
for the Gromov--Hausdorff--Prokhorov topology. Furthermore, the following identity in law holds:
\begin{equation}\label{eq:distance_points_unif_carte_sous_suite}
d_\infty(U,U') \eqloi L_U - \min_{t \in [0,1]} L_t,
\end{equation}
where $U$, $U'$ are i.i.d. uniform random variables on $[0,1]$ and independent of everything else.
\end{prop}
 
\begin{proof}
Let us implicitly restrict ourselves to a subsequence along which~\eqref{eq:convergence_distances_sous_suite} holds.
Appealing to Skorokhod's representation theorem, let us assume furthermore that it holds almost surely. 
Recall that for a vertex $x$ of $\fn$, we denote by $\varphi(x)$ the vertex of $\mn\setminus\{x_\star\}$ in one-to-one correspondence with the leaf at the extremity of the right-most ancestral line starting from $x$ in $\fn$.
Then the sequence $(\varphi(x_i))_{0 \le i \le \vertices_n}$ lists \emph{with redundancies} the vertices of $\mn$ different from $x_\star$. For every $1 \le i \le d_n(0)$, let us denote by $\lambda(i) \in \{0, \dots, \vertices_n-1\}$ the index such that $x_{\lambda(i)}$ is the $i$'th leaf of $\fn$, and extend $\lambda$ linearly between integer times. The function $\lambda$ corresponds to $\zeta_{\{-1\}}(\,\cdot\,; \Wfn)$ with the notation of Lemma~\ref{lem:proportion_feuilles}. According to this lemma, we have
\begin{equation}\label{eq:approximation_sites_aretes_carte}
\left(\vertices_n^{-1} \lambda(\lfloor d_n(0) t\rfloor) ; t \in [0,1]\right) \cvproba (t ; t \in [0,1]),
\end{equation}
Observe that the sequence $(\varphi(x_{\lambda(i)}))_{1 \le i \le d_n(0)}$ now lists \emph{without redundancies} the vertices of $\mn$ different from $x_\star$. The set
\[\mathscr{R}_n = \left\{\left(\varphi(x_{\lambda(\lfloor d_n(0) t\rfloor)}), \Pi_\infty(t)\right) ; t \in [0,1]\right\}.\]
is a correspondence between $\mn\setminus\{x_\star\}$ and $M_\infty$. Let further $\nu$ be the coupling between $\pgr$ and $p_\infty$ given by
\[\int_{(\mn\setminus\{x_\star\}) \times M_\infty} f(x, z) \d\nu(x, z) = \int_0^1 f\left(\varphi(x_{\lambda(\lfloor d_n(0) t\rfloor)}), \Pi_\infty(t)\right) \d t,\]
for every test function $f$. Then $\nu$ is supported by $\mathscr{R}_n$ by construction. Finally, the distortion of $\mathscr{R}_n$ is given by
\[\sup_{s,t \in [0,1]} \left|d_{(n)}\left(\frac{\lambda(\lfloor d_n(0) s\rfloor)}{\vertices_n}, \frac{\lambda(\lfloor d_n(0) t\rfloor)}{\vertices_n}\right) - d_\infty(s,t)\right|,\]
which, appealing to~\eqref{eq:approximation_sites_aretes_carte}, tends to $0$ whenever the convergence~\eqref{eq:convergence_distances_sous_suite} holds, which concludes the proof of the convergence of the maps.

Let us next consider the identity~\eqref{eq:distance_points_unif_carte_sous_suite} which we shall need shortly in the proof of Theorem~\ref{thm:convergence_carte_disque}. If $U$, $U'$ are i.i.d. uniform random variables on $[0,1]$ and independent of everything else, then $X_n = \varphi(x_{\lambda(\lfloor d_n(0) U\rfloor)})$ and $Y_n = \varphi(x_{\lambda(\lfloor d_n(0) U'\rfloor)})$ are uniform random vertices of $V(\mn) \setminus \{x_\star\}$, they can therefore be coupled with two independent uniform random vertices $X_n'$ and $Y_n'$ of $\mn$ in such a way that the conditional probability given $\mn$ that $(X_n,Y_n) \ne (X_n', Y_n')$ converges to $0$; we implicitly assume in the sequel that $(X_n,Y_n) = (X_n', Y_n')$. Since $x_\star$ is also a uniform random vertex of $\mn$, we obtain that
\[\dgr(X_n,Y_n) \eqloi \dgr(x_\star, Y_n).\]
By definition we have $\dgr(X_n,Y_n) = d_n(\lambda(\lfloor d_n(0) U\rfloor), \lambda(\lfloor d_n(0) U'\rfloor))$ and by construction of the labels on $\fn$, we have
\[\dgr(x_\star, Y_n) = \Lfn(\lambda(\lfloor d_n(0) U'\rfloor)) - \min_{0 \le j \le \vertices_n} \Lfn(j) + 1.\]
Letting $n \to \infty$ along the same subsequence as in~\eqref{eq:convergence_distances_sous_suite} and appealing to~\eqref{eq:approximation_sites_aretes_carte}, we obtain~\eqref{eq:distance_points_unif_carte_sous_suite}.
\end{proof}

\subsection{Convergence towards a Brownian disk}
\label{sec:carte_brownienne}

Let us next turn to the proof of Theorem~\ref{thm:convergence_carte_disque}; in view of the proof of Theorem~\ref{thm:tension_cartes} it suffices to identify the subsequential limit $d_\infty$. 
Fix $\varrho \in [0, \infty)$ and let us recall the distance $\dBmap^\varrho$ of the Brownian disk, following Le~Gall~\cite{LG07} and Bettinelli \& Miermont~\cite{BM17} to which we refer for details. Recall the process $Z^\varrho$ defined in Section~\ref{sec:enonce_tension_convergence_etiquettes} and let us construct the function $D_{Z^\varrho}$ on $[0,1]^2$ as in~\eqref{eq:distance_fonction_continue}; first, let us view $D_{Z^\varrho}$ as a function on the forest $\CRT_{X^\varrho}$ by setting
\[D_{Z^\varrho}(x,y) = \inf\left\{D_{Z^\varrho}(s,t) ; s,t \in [0,1], x=\pi_{X^\varrho}(s) \text{ and }  y=\pi_{X^\varrho}(t)\right\},\]
for every $x, y \in \CRT_{X^\varrho}$, where we recall the notation $\pi_{X^\varrho}$ for the canonical projection. This function is not a pseudo-distance on $\CRT_{X^\varrho}$ but one can construct one as follows: for every $x, y \in \CRT_{X^\varrho}$, set
\[\dBmap^\varrho(x,y) = \inf\left\{\sum_{i=1}^k D_{Z^\varrho}(a_{i-1}, a_i) ; k \ge 1, (x=a_0, a_1, \dots, a_{k-1}, a_k=y) \in \CRT_{X^\varrho}\right\}.\]
The function $\dBmap^\varrho$ is in fact the largest pseudo-distance on $\CRT_{X^\varrho}$ which is bounded by $D_{Z^\varrho}$. Indeed if $D$ is another pseudo-distance with $D \le D_{Z^\varrho}$, then for every $x,y \in \CRT_{X^\varrho}$, for every $k \ge 1$ and every chain $x=a_0, a_1, \dots, a_{k-1}, a_k=y$ in $\CRT_{X^\varrho}$, by the triangle inequality, $D(x, y) \le \sum_{i=1}^k D(a_{i-1}, a_i) \le \sum_{i=1}^k D_{Z^\varrho}(a_{i-1}, a_i)$ and so $D(x, y) \le \mathscr{D}^\varrho(x, y)$.
Furthermore, it can be seen as a pseudo-distance on $[0,1]$ by setting $\dBmap^\varrho(s,t) = \dBmap^\varrho(\pi_{X^\varrho}(s),\pi_{X^\varrho}(t))$ for every $s,t \in [0,1]$. Then for all $s,t \in [0,1]$ such that $\dCRT_{X^\varrho}(s,t) = 0$ we have $\pi_{X^\varrho}(s) = \pi_{X^\varrho}(t)$ and so $\mathscr{D}^\varrho(\pi_{X^\varrho}(s),\pi_{X^\varrho}(t)) = 0$. We deduce from the previous maximality property that $\mathscr{D}^\varrho$ is the largest pseudo-distance $D$ on $[0,1]$ satisfying the following two properties:
\[D \le D_{Z^\varrho}
\qquad\text{and}\qquad
\dCRT_{X^\varrho}(s,t) = 0 \quad\text{implies}\quad D(s,t) = 0.\]
The Brownian disk is then given by the quotient $[0,1] / \{\dBmap^\varrho=0\}$, endowed with the metric $\dBmap^\varrho$ and the push-forward $\pBmap^\varrho$ of the Lebesgue measure on $[0,1]$.

Let us now prove Theorem~\ref{thm:convergence_carte_disque}.

\begin{proof}[Proof of Theorem~\ref{thm:convergence_carte_disque}]
Let us first assume that $\lim_{n \to \infty} \sigma_n^{-1} \Delta_n = 0$ and that $\lim_{n \to \infty} \sigma_n^{-1} \varrho_n = \varrho$ with $\varrho \in [0,\infty)$ and let us prove that our maps converge to the Brownian disk.
Let us set for all $t \in [0,1]$,
\[W_{[n]}(t) = \frac{1}{\sigma_n} \Wfn(\lfloor \vertices_n t \rfloor)
\qquad\text{and}\qquad
L_{[n]}(t) = \left(\frac{3}{2\sigma_n}\right)^{1/2} \Lfn(\vertices_n t),\]
and for all $s,t \in [0,1]$
\[d_{[n]}(s, t) = \left(\frac{3}{2\sigma_n}\right)^{1/2} d_n(\vertices_n s, \vertices_n t).\]
Then we infer from Proposition~\ref{prop:convergence_Luka}, Theorem~\ref{thm:convergence_etiquettes}~\eqref{thm:convergence_etiquettes_carte_disque}, and the preceding subsection that from every increasing sequence of integers, one can extract a subsequence along which we have
\[\left(W_{[n]}(t), L_{[n]}(t), D_{L_{[n]}}(s, t), d_{[n]}(s, t)\right)_{s,t \in [0,1]}
\cvloi
(X^\varrho, Z^\varrho_t, D_{Z^\varrho}(s,t), d_\infty(s,t))_{s,t \in [0,1]},\]
where $d_\infty$ depends a priori on the subsequence.
Again by Skorokhod's representation theorem, let us assume that this holds almost surely.
It remains to prove that $d_\infty = \dBmap^\varrho$ almost surely.
Our argument is adapted from the work of Bettinelli \& Miermont~\cite[Lemma~32]{BM17}. 

First recall from the previous proof that $d_\infty$ is a pseudo-distance on $[0,1]$ which satisfies $d_\infty \le D_{Z^\varrho}$ almost surely. Let $0 \le s < t \le 1$ be such that $\dCRT_{X^\varrho}(s,t) = 0$, i.e. such that $X^\varrho_s = X^\varrho_t = \min_{r \in [s,t]} X^\varrho_r$. Suppose first that $X^\varrho_r > X^\varrho_t$ for all $r \in (s,t)$, then by the preceding convergence one can find two sequences of integers $(i_n)_{n \ge 1}$ and $(j_n)_{n \ge 1}$ such that $i_n / \vertices_n \to s$ and $j_n / \vertices_n \to t$ and for all $n$ large enough $\Wfn(i_n) = \Wfn(j_n) \le \Wfn(k)$ for every $i_n+1 \le k \le j_n-1$. In terms of the forest $\fn$, this means that the vertex $x_{i_n}$ is an ancestor of $x_{j_n}$ and that furthermore the latter lies on the right-most ancestral line amongst the descendants of $x_{i_n}$ in $\fn$, so in particular $\varphi(x_{i_n}) = \varphi(x_{j_n})$, hence $d_n(i_n, j_n) = 0$. Letting $n \to \infty$ along the same subsequence as above, we conclude that $d_\infty(s,t) = 0$. If there exists $r \in (s,t)$ with $X^\varrho_s = X^\varrho_t = X^\varrho_r$, then it is unique and we conclude similarly that $d_\infty(s,r) = d_\infty(r,t) = 0$.
From the maximality property of $\dBmap^\varrho$, we deduce the bound
\[d_\infty \le \dBmap^\varrho
\qquad\text{almost surely}.\]

Next let $U$, $U'$ be i.i.d. uniform random variables on $[0,1]$ and independent of everything else. Recall the identity~\eqref{eq:distance_points_unif_carte_sous_suite} which reads here $d_\infty(U,U') = Z^\varrho_U - \min Z^\varrho$ in distribution. The key point is that, according to Le~Gall~\cite[Corollary 7.3]{LG13} for $\varrho = 0$ and Bettinelli \& Miermont~\cite[Lemma~17 \& Corollary~21]{BM17} for $\varrho > 0$, the right-hand side is also distributed as $\dBmap^\varrho(U,U')$. We conclude that the identity
\[d_\infty(U,U') = \dBmap^\varrho(U,U')\]
holds in distribution and thus almost surely by the previous bound. The identity $d_\infty = \dBmap^\varrho$ follows by a density argument.

We end this proof by arguing that our assumption on the largest degree is necessary. Indeed, extracting a subsequence if necessary, let us assume that $\sigma_n^{-1} \varrho_n$ and $\sigma_n^{-1} \Delta_n$ converge respectively to $\varrho \ge 0$ and $\delta > 0$ and that $(V(\mn), \sigma_n^{-1/2} \dgr, \pgr)$ converges in distribution to some random space $(M, D, p)$ and let us prove that the latter does not have the topology of $\Bmap^\varrho$, which is that of the sphere if $\varrho = 0$~\cite{LGP08, Mie08} or the disk if $\varrho > 0$~\cite{Bet15}. Let us label all the vertices by their graph distance to the distinguished vertex $x_\star$. Let $\Phi_n$ be an inner face with degree $2\Delta_n \sim 2 \delta \sigma_n$, then the labels of its vertices read in clockwise order form a (shifted) bridge with $\pm1$ steps which, when rescaled by a factor of order $\Delta_n^{-1/2}$ converges towards a Brownian bridge. Let $x_n^-, x_n^+$ be two vertices of $\Phi_n$ such that their respective labels are the minimum and the maximum over all labels on $\Phi_n$. Then when $\varepsilon > 0$ is small, with high probability we have that $\ell(x_n^+) - \ell(x_n^-) > 6 \varepsilon \sigma_n^{1/2}$.

Our argument is depicted on Figure~\ref{fig:grande_face}; let us describe it. Let us read the vertices on the face $\Phi_n$ from $x_n^-$ to $x_n^+$ in clockwise order, there is a vertex which is the last one with label smaller than $\ell(x_n^-) + \varepsilon \sigma_n^{1/2}$ and another one which is the first one with label larger than $\ell(x_n^+) - \varepsilon \sigma_n^{1/2}$; let us call `blue vertices' all the vertices visited between these two. Let us similarly call `green vertices' the vertices defined similarly when going from $x_n^-$ to $x_n^+$ in counter-clockwise order. A vertex may be simultaneously green and blue, but in this case, and more generally if there exists a pair of blue and green vertices at graph distance $o(\sigma_n^{1/2})$ in the whole map, then this creates at the limit a pinch-point separating the map into two parts each with diameter larger than $\varepsilon \sigma_n^{1/2}$.

\begin{figure}[!ht] \centering
\includegraphics[height=9\baselineskip]{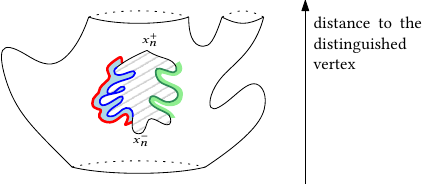}
\caption{A portion of the pointed map represented as a surface (the so-called `cactus' representation), seen from the distinguished vertex. If the grey face is macroscopic and the light blue and light green regions are disjoint, then the red simple path separates the map into two macroscopic parts.}
\label{fig:grande_face}
\end{figure}

We assume henceforth that the distance between green and blue vertices is larger than $4\eta \sigma_n^{1/2}$ for some $\eta > 0$; the light blue and light green regions in Figure~\ref{fig:grande_face} represent the hull of the set of vertices at distance smaller than $\eta \sigma_n^{1/2}$ from the blue and green vertices respectively. Consider the simple red path obtained by taking the boundary of the light blue region: its extremities lie at macroscopic distance and it separates two parts of the map with macroscopic diameter. A sphere cannot be separated by a simple curve with distinct extremities so this yields our claim when $\varrho = 0$.
This is possible on a disk, if the path touches the boundary twice. Therefore, in order to avoid a contradiction, the red path must contain (at least) two points at macroscopic distance from each other, and both at microscopic distance from the boundary of the map. If none of these points is at microscopic distance from an extremity of the path, then the preceding argument still applies, so both extremities, which belong to $\Phi_n$ must lie at microscopic distance from the boundary, but then this creates pinch-points at the limit.
\end{proof}

\subsection{Convergence to the Brownian tree}
\label{sec:convergence_carte_arbre}

Let us finally prove Theorem~\ref{thm:convergence_cartes_CRT} relying on Theorem~\ref{thm:convergence_etiquettes}~\eqref{thm:convergence_etiquettes_arbre}. The idea is that the greatest distance in the map to the boundary is small compared to the scaling so only the boundary remains in the limit, and furthermore this boundary, whose distances are related to a discrete bridge, converges to the Brownian tree, which is encoded by the Brownian bridge.

\begin{proof}[Proof of Theorem~\ref{thm:convergence_cartes_CRT}]
Let us denote by $\bb$ and $X^0$ respectively the standard Brownian bridge and the standard Brownian excursion. Recall the construction of the Brownian tree $(\CRT_{X^0}, \dCRT_{X^0}, \pCRT_{X^0})$ and let us define a random pseudo-distance $D_{\bb}$ as in~\eqref{eq:distance_fonction_continue}; let further $(\CRT_{\bb}, D_{\bb}, p_{\bb})$ be the space constructed as $M_\infty$ where $d_\infty$ is replaced by $D_{\bb}$, so $\CRT_{X^0}$ and $\CRT_{\bb}$ are obtained by taking the quotient of $[0,1]$ by  $\{\dCRT_{X^0} = 0\}$ and $\{D_{\bb} = 0\}$ respectively.
Comparing the definition of $D_{\bb}$ and $\dCRT_{X^0}$, since $\bb$ and $X^0$ are related by the Vervaat transform, it is easy to prove that the spaces $(\CRT_{X^0}, \dCRT_{X^0}, \pCRT_{X^0})$ and $(\CRT_{\bb}, D_{\bb}, p_{\bb})$ are isometric so it is equivalent to prove the convergence in distribution
\[(V(\mn)\setminus\{x_\star\}, (2\varrho_n)^{-1/2} \dgr, \pgr) \cvloi (\CRT_{\bb}, D_{\bb}, p_{\bb})\]
in the Gromov--Hausdorff--Prokhorov topology. From the previous proofs it actually suffices to prove the convergence in distribution for the uniform topology
\begin{equation}\label{eq:cv_distance_carte_arbre}
\left((2\varrho_n)^{-1/2} d_n(\vertices_n s, \vertices_n t)\right)_{s,t \in [0,1]} \cvloi (D_{\bb}(s,t))_{s,t \in [0,1]}.
\end{equation}

Here we may adapt the argument of the proof of Theorem~5 in~\cite{Bet15} to which we refer for details. 
First, 
recall the upper bound~\eqref{eq:borne_sup_cactus}:
\[d_n(i,j) \le \Lfn(i) + \Lfn(j) + 2 - 2 \max\left\{\min_{k \in [i, j]} \Lfn(k); \min_{k \in [j, i]} \Lfn(k)\right\},\]
where for $0 \le i < j \le \vertices_n$, the interval $[i, j]$ denotes the set of integers from $i$ to $j$, and $[j,i]$ denotes $[j, \vertices_n] \cup [0, i]$.
We have a very similar lower bound, see the proof of Corollary~4.4 in~\cite{CLGM13}: 
\begin{equation}\label{eq:borne_inf_cactus}
d_n(i,j) \ge \Lfn(i) + \Lfn(j) - 2 \max\left\{\min_{k \in \llbracket i, j\rrbracket} \Lfn(k); \min_{k \in \llbracket j, i\rrbracket} \Lfn(k)\right\},
\end{equation}
where the intervals are defined as follows. Let us link the roots of two consecutive trees in the forest, as well as the first and last one in order to create a cycle. Then $\mathopen{\llbracket}i, j\mathclose{\rrbracket}$ denotes the set of all indices $k$ such that $x_k$ lies in the geodesic path between $x_i$ and $x_j$ in this graph; in other words, $k \ge 1$ belongs to $\mathopen{\llbracket}i, j\mathclose{\rrbracket}$ if either $x_k$ is an ancestor of $x_i$ or of $x_j$ (and it is an ancestor of both if and only if it is their last common one), or if it is the root of a tree which lies between $x_i$ and $x_j$ in the original forest. We define $\mathopen{\llbracket}j , i\mathclose{\rrbracket}$ similarly as the set of all indices $k \ge 1$ such that either $x_k$ is an ancestor of $x_i$ or of $x_j$, or it is the root of a tree which \emph{does not} lie between $x_i$ and $x_j$.

Let us suppose that $\lim_{n \to \infty} \sigma_n^{-1} \varrho_n = \infty$; according to 
Theorem~\ref{thm:convergence_etiquettes}~\eqref{thm:convergence_etiquettes_arbre} the convergence in distribution
\[\left((2\varrho_n)^{-1/2} \Lfn(\vertices_n t); t \in [0,1]\right) \cvloi (\bb_t; t \in [0,1])\]
holds in $\mathscr{C}([0,1],\R)$. 
Then the right-hand side of~\eqref{eq:borne_sup_cactus} with $i=\vertices_n s$ and $j = \vertices_n t$ converges in distribution to $D_{\bb}(s,t)$ once rescaled by a factor $(2\varrho_n)^{-1/2}$. The same holds for the right-hand side of~\eqref{eq:borne_inf_cactus}; indeed, the root vertices visited in $[i,j]$ and in $\mathopen{\llbracket}i, j\mathclose{\rrbracket}$ are the same, so the only difference lies in the non-root vertices, but their label differ by that of the root of their tree by at most $\max \tLfn - \min \tLfn = o(\varrho_n^{1/2})$ in probability by the proof of Theorem~\ref{thm:convergence_etiquettes}~\eqref{thm:convergence_etiquettes_arbre}. This proves the convergence~\eqref{eq:cv_distance_carte_arbre} and hence our claim.
\end{proof}

\section{Geometry of random forests with a prescribed degree sequence}
\label{sec:forets}

As we already mentioned, most of this paper is devoted to the study of labelled forests. In this section we focus on the random forest $\fn$, the labels are studied in Section~\ref{sec:etiquettes}. 
We first state in Section~\ref{sec:epine} a technical spinal decomposition whose proof is postponed to the appendix, which approximates the offspring distribution of the ancestors of random vertices by random picks, without replacement, from the \emph{size-biased} offspring distribution 
$\sum_{k \ge 1} \frac{k d_n(k)}{\edges_n} \delta_{k}$. 
Then in Section~\ref{sec:Luka} we state and prove some exponential concentration for the random bridge $\Bfn$ and the excursion $\Wfn$ which are combined with the spinal decomposition in Section~\ref{sec:queues_exp} to prove exponential bounds for the height of a random vertex.
Finally we prove Theorem~\ref{thm:marginales_hauteur} in Section~\ref{sec:marginales_hauteur} on the convergence of the reduced trees to reduced Brownian trees in the case of no macroscopic degrees, relying again on the spinal decomposition.

Throughout this section we are given for every $n \ge 1$ a degree sequence $d_n = (d_n(k))_{k \ge 0} \in \Z_+^{\Z_+}$ and $\fn$ is sampled uniformly at random in the set $\Fn$ of forests which possess $\varrho_n = \sum_{k \ge 0} (1-k) d_n(k)$ trees and $\vertices_n = \sum_{k \ge 0} d_n(k)$ vertices, amongst which $d_n(k)$ have $k$ children for every $k \ge 0$; it therefore has $\edges_n = \sum_{k \ge 1} k d_n(k)$ edges and we recall the notation $\Delta_n = \max\{k \ge 0 : d_n(k) > 0\}$ for the largest offspring and $\sigma_n^2 = \sum_{k \ge 1} k (k-1) d_n(k)$.

\subsection{A spinal decomposition}
\label{sec:epine}

We describe in this section the ancestral lines of i.d.d. random vertices of the random forest $\fn$. A related result on trees was first obtained by Broutin \& Marckert~\cite{BM14} for a single random vertex and it was extended in~\cite{Mar18b} to several vertices. The present one is different, we shall compare them after the statement.

Suppose that an urn contains initially $k d_n(k)$ balls labelled $k$ for every $k \ge 1$, so $\edges_n$ balls in total; let us pick balls repeatedly one after the other \emph{without} replacement; for every $1 \le i \le \edges_n$, we denote the label of the $i$'th ball by $\xi_{d_n}(i)$. Conditionally on $(\xi_{d_n}(i))_{1 \le i \le \edges_n}$, let us sample independent random variables $(\chi_{d_n}(i))_{1 \le i \le \edges_n}$ such that each $\chi_{d_n}(i)$ is uniformly distributed in $\{1, \dots, \xi_{d_n}(i)\}$.
We shall use the fact that the $\xi_{d_n}(i)$'s are identically distributed, with
\begin{equation}\label{eq:moyenne_biais_par_la_taille}
\begin{gathered}
\Es{\xi_{d_n}(i) - 1} = \sum_{k \ge 1} (k-1) \frac{k d_n(k)}{\edges_n} = \frac{\sigma_n^2}{\edges_n},
\\
\Var\left(\xi_{d_n}(i)-1\right) \le \Es{(\xi_{d_n}(1)-1) \cdot \max_i (\xi_{d_n}(i)-1)} \le \Delta_n \frac{\sigma_n^2}{\edges_n}.
\end{gathered}
\end{equation}

Let us fix a plane forest $F$. Recall that we denote by $\chi_{x}$ the relative position of a vertex $x$ amongst its siblings; for every $0 \le i \le |x|$, let us also denote by $a_i(x)$ the ancestor of $x$ at height $i$, so $a_{|x|}(x) = x$ and $a_0(x)$ is the root of the tree containing $x$. Define next for every vertex $x$ the \emph{content} of the branch $\mathopen{\llbracket}a_0(x), x\mathclose{\llbracket}$ as
\begin{equation}\label{eq:content}
\Cont(x) = \left((k_{a_{i-1}(x)}, \chi_{a_i(x)}); 1 \le i \le |x|\right).
\end{equation}
In words, $\Cont(x)$ lists the number of children of each strict ancestor of $x$ and the position of the next ancestor amongst these children. More generally, let $x_1, \dots, x_q$ be $q$ vertices of $F$ and let us consider the forest $F$ \emph{reduced} to its root and these vertices: $F(x_1, \dots, x_q)$ contains only the vertices $x_1, \dots, x_q$ and their ancestors, and it naturally inherits a plane forest structure from $F$.
Now let $\Cont(x_1, \dots, x_q)$ be defined as the collection of pairs $(k_{pr(y)}, \chi_y)$ where these quantities are those \emph{in the original forest $F$}, and $y$ ranges in lexicographical order over the set of vertices of the reduced forest whose parent has only one offspring in this reduced forest.
We refer to Figure~\ref{fig:contenu_foret} for an example. Let us stress that the branchpoints of the reduced forest are excluded from $\Cont(x_1, \dots, x_q)$.

\begin{figure}[!ht] \centering
\includegraphics[height=10\baselineskip]{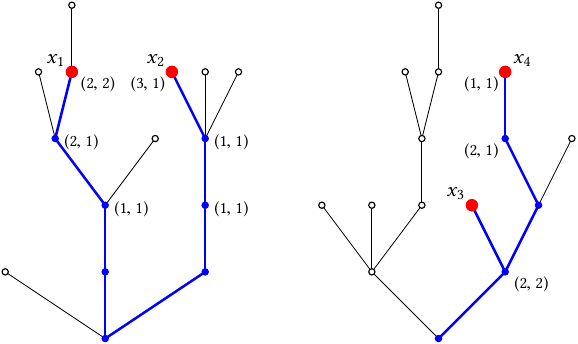}
\caption{A forest $F$ consisting of two trees with four distinguished red vertices, in blue is the reduced forest $F(x_1, \dots, x_4)$; the pairs $(k_{pr(y)}, \chi_y)$ which form the associated content $\Cont(x_1, \dots, x_4)$
are indicated next to their node $y$, the children of the blue branchpoints are excluded.}
\label{fig:contenu_foret}
\end{figure}

\begin{lem}\label{lem:multi_epines_sans_remise}
Fix $n \in \N$ and a degree sequence $d_n = (d_n(k))_{k\ge 1}$.
Let $q \ge 1$ and sample $x_{n,1}, \dots, x_{n,q}$ independently uniformly at random in $\fn$. Let $(k_i, j_i)_{i=1}^h$ be positive integers such that $1 \le j_i \le k_i$ for each $i$. If $q \ge 2$, assume that $h, q \le \vertices_n/4$. Then for every integers 
$b \ge 0$ and $c \ge 1$ with $b+c \le q$, 
the probability that $\Cont(x_{n,1}, \dots, x_{n,q}) = (k_i, j_i)_{i=1}^h$, and that the reduced forest ${T}_{d_n}^{\varrho_n}(x_{n,1}, \dots, x_{n,q})$ possesses $c$ trees, $q$ leaves, and $b$ branchpoints is bounded above by
\[q^2 2^{q - 1} \left(\frac{\sigma_n}{\vertices_n}\right)^{q + b} \left(\frac{\varrho_n}{\sigma_n}\right)^{c-1} \frac{\varrho_n + (q-1) \Delta_n + \sum_{i=1}^h (k_i - 1)}{\sigma_n}\]
times
\[\P\bigg(\bigcap_{i \le h} \left\{(\xi_{d_n}(i), \chi_{d_n}(i)) = (k_i, j_i)\right\}\bigg).\]
\end{lem}

The proof consists in straightforward but long combinatorial calculations derived from the decomposition of the forest in the spirit of~\cite{BM14} and is given in Appendix~\ref{sec:preuve_epine}. Let us make two comments.
First in the case of a single random vertex, we have $b=0$ and $c=1$ so the upper bound reads simply
\begin{equation}\label{eq:epine_un_sommet}
\frac{\varrho_n + \sum_{i=1}^h (k_i - 1)}{\vertices_n}
\cdot \Pr{\bigcap_{i \le h} \left\{(\xi_{d_n}(i), \chi_{d_n}(i)) = (k_i, j_i)\right\}}.
\end{equation}
Second, the reduced forest ${T}_{d_n}^{\varrho_n}(x_{n,1}, \dots, x_{n,q})$ possesses $q$ leaves when no $x_{n,i}$ is an ancestor of another; we shall see below that the height of a random vertex is of order $\edges_n / \sigma_n$, so this occurs with high probability as soon as $\sigma_n \to \infty$. Also, if the reduced forest has $b$ branchpoints, then $q+b$ denotes the number of branches once we remove these branchpoints; these branches typically have length of order $\edges_n/\sigma_n$ in the original forest, so the factor $(\sigma_n/\vertices_n)^{q+b}$ is important. The other factors will be typically bounded in our applications so the upper bound will read
\[\mathrm{Cst} \cdot \left(\frac{\sigma_n}{\vertices_n}\right)^{q + b} \cdot \Pr{\bigcap_{i \le h} \left\{(\xi_{d_n}(i), \chi_{d_n}(i)) = (k_i, j_i)\right\}}.\]

In words, roughly speaking, along distinguished paths (removing the branchpoints), up to a multiplicative factor, the individuals reproduce according to the \emph{size-biased law} $(kd_n(k)/\edges_n)_{k \ge 1}$, and conditionally on the offspring, the paths continue via offspring chosen uniformly at random. Note that these size-biased picks are not independent, since we are sampling without replacement. In the case of a single tree $\varrho_n = 1$, an analogous result when sampling \emph{with} replacement was obtained by Broutin \& Marckert~\cite{BM14} for a single random vertex and it was extended in~\cite{Mar18b} to several vertices. The significant difference is that when comparing to sampling with replacement, an extra factor of order $\e^{h^2/\edges_n}$ appears in the upper bound, and one cannot remove it. This was not an issue in~\cite{BM14, Mar18b} which focus on the `finite-variance regime', when $\sigma_n^2$ is of order $\edges_n$ since then the bound~\eqref{eq:queue_exp_hauteur} below provided by Addario-Berry~\cite{AB12} ensures that $h^2$ is at most of order $\edges_n$ with high probability, so the exponential factor does not explode. We mentioned already that the height of uniform random vertices is of order $\edges_n / \sigma_n$ so it does not explode anyway, but we shall prove this fact by relying on this lemma.
Moreover, we feel that sampling without replacement is more natural in this model.

\subsection{Concentration for bridges and excursions}
\label{sec:Luka}

Let us prove two technical bounds which we shall need 
in the next subsection, as well as for the proof of Theorem~\ref{thm:tension_etiquettes} on tightness of the label process in Section~\ref{sec:etiquettes}.
Recall the notation $\Bfn = (\Bfn(i))_{0 \le i \le \vertices_n}$ for a bridge from $\Bfn(0) = 0$ to $\Bfn(\vertices_n) = -\rho_n$ sampled uniformly at random amongst all those which make exactly $d_n(k)$ jumps with value $k-1$ for every $k \ge 0$. Recall the construction of the {\L}ukasiewicz path $\Wfn$ of our random forest $\fn$ as a cyclic shift of $\Bfn$. 
In Proposition~\ref{prop:convergence_Luka} we argued that then $\Bfn$ and so $\Wfn$ scale like $\sigma_n+\varrho_n$. The next proposition provides a precise H\"older-continuity like bound for the infimum process which is tailored for our applications.

\begin{prop}\label{prop:Holder_marche_Luka}
Fix a sequence $(d_n)_{n \ge 1}$ and $\varepsilon \in (0,1)$, then there exists $C > 0$ such that, for every $n \ge 1$, with probability at least $1-\varepsilon$, it holds
\[\Wfn(\lfloor\vertices_ns\rfloor) - \min_{s \le r \le t} \Wfn(\lfloor\vertices_nr\rfloor)
\le C \cdot (\sigma_n + \varrho_n) \cdot |t-s|^{(1-\varepsilon)/2}\]
uniformly for $0 \le s < t \le 1$.
\end{prop}

The proof relies on the following exponential tail bound for the random bridge $\Bfn$, adapted from~\cite[Section~3]{AB12} in the case $\varrho_n=1$.

\begin{lem}\label{lem:queue_exp_min_pont}
For every $\alpha \in (0,1)$ and every $z \ge 0$, we have
\[\Pr{\min_{i \le \alpha \vertices_n} \Bfn(i) + \alpha \varrho_n \le - z}
\le \exp\left(- \frac{(1 - \alpha)^3 z^2}{2 \alpha (\sigma_n^2+\varrho_n) + z}\right).\]
\end{lem}

\begin{proof}
It is classical that the `remaining sequence' (remaining space divided by the remaining time)
\[M_i = \frac{- \varrho_n - \Bfn(i)}{\vertices_n-i}
\qquad (0 \le i \le \vertices_n - 1)\]
is a martingale for the natural filtration $(F_i)_{0 \le i \le \vertices_n - 1}$, started from $M_0 = -\varrho_n / \vertices_n$.
Indeed, for each time $i \ge 1$, set $\Delta\Bfn(i) = \Bfn(i) - \Bfn(i-1)$ and let $d_n^{i}$ be the remaining degree sequence given by $d_n^{i}(k) = d_n(k) - \#\{j \le i : \Delta\Bfn(j) = k-1\}$ for every $k \ge 0$; note that $d_n^{i}$ is random and $F_i$-measurable. Fix $0 \le i \le \vertices_n-2$; then we have
\[\Prc{\Delta\Bfn(i+1) = k-1}{F_i} = \frac{d_n^{i}(k)}{\vertices_n-i},\qquad k \ge 0,\]
from which we obtain
\[\Esc{\Delta\Bfn(i+1)}{F_i} = \sum_{k \ge 0} (k-1) \frac{d_n^{i}(k)}{\vertices_n-i} 
= - \frac{\varrho_n + \Bfn(i)}{\vertices_n-i},\]
which shows the martingale property. Furthermore,
\[\Esc{\Delta\Bfn(i+1)^2}{F_i} = \sum_{k \ge 0} (k-1)^2 \frac{d_n^{i}(k)}{\vertices_n-i}
\le \frac{\sigma_n^2+\varrho_n}{\vertices_n-i}.\]
Since $\Bfn(i+1) \ge \Bfn(i) - 1$, then for every $0 \le i \le \vertices_n - 2$,
\[M_{i+1} - M_i
\le \frac{\vertices_n-i - \Bfn(i) - \varrho_n}{(\vertices_n-(i+1)) (\vertices_n-i)}
\le \frac{\edges_n}{(\vertices_n-(i+1))^2}
\]
and 
\[\Var(M_{i+1} \mid F_i) 
\le \frac{1}{(\vertices_n-(i+1))^2} \Esc{\Delta\Bfn(i+1)^2}{F_i}
\le \frac{\sigma_n^2 + \varrho_n}{(\vertices_n-(i+1))^3}.\]

Then the martingale Chernoff-type bound, see e.g. McDiarmid~\cite{McD98}, Theorem~3.15 and the remark at the end of Section~3.5 there about the maximum, shows that for every $z \ge 0$ and every $1 \le k \le \vertices_n-1$,
\begin{align*}
\Pr{\max_{i \le k} M_i - M_0 \ge z}
&\le \exp\left(- \frac{z^2}{2 \frac{k (\sigma_n^2+\varrho_n)}{(\vertices_n - k)^3} + \frac{2}{3} z \frac{\edges_n}{(\vertices_n - k)^2}}\right)
\\
&\le \exp\left(- \frac{(\vertices_n - k)^3 z^2}{2k (\sigma_n^2+\varrho_n) + \edges_n (\vertices_n-k) z}\right).
\end{align*}
Observe that for every $i \le k$,
\[- (M_i - M_0)
= \frac{\Bfn(i) + \varrho_n}{\vertices_n-i} - \frac{\varrho_n}{\vertices_n}
= \frac{\Bfn(i) + i \varrho_n / \vertices_n}{\vertices_n-i}
\le \frac{\Bfn(i) + k \varrho_n / \vertices_n}{\vertices_n-i},\]
we infer that for every $z \ge 0$ and every $1 \le k \le \vertices_n-1$,
\begin{align*}
\Pr{\min_{i \le k} \left(\Bfn(i) + \frac{k \varrho_n}{\vertices_n}\right) \le - z}
&\le \Pr{\min_{i \le k} \frac{\Bfn(i) + i \varrho_n / \vertices_n}{\vertices_n-i} \le - \frac{z}{\vertices_n}}
\\
&\le \exp\left(- \frac{(\vertices_n - k)^3 z^2 / \vertices_n^2}{2k (\sigma_n^2+\varrho_n) + \edges_n (\vertices_n-k) z / \vertices_n}\right).
\end{align*}
The claim follows by taking $k = \alpha \vertices_n$ with $\alpha \in (0,1)$.
\end{proof}

We may now prove Proposition~\ref{prop:Holder_marche_Luka}.

\begin{proof}[Proof of Proposition~\ref{prop:Holder_marche_Luka}]
Let us first focus on the bridge $\Bfn$ and observe that, with no jump smaller than $-1$, we have $\min_{i \le \alpha \vertices_n} \Bfn(i) + \alpha \varrho_n \ge - \alpha \vertices_n + \alpha \varrho_n =  - \alpha \edges_n$. Therefore, in Lemma~\ref{lem:queue_exp_min_pont}, the probability on the left-hand side is zero as soon as $z > \alpha \edges_n$ so for every $\alpha \in (0,1/2]$ and every $z \ge 0$,
\begin{align*}
\Pr{\min_{i \le \alpha \vertices_n} \Bfn(i) + \alpha \varrho_n \le - z}
&\le \exp\left(- \frac{(1 - \alpha)^3 z^2}{\alpha (2 (\sigma_n^2+\varrho_n) + \edges_n)}\right)
\\
&\le \exp\left(- \frac{z^2}{16 \alpha (\sigma_n^2+\varrho_n + \edges_n)}\right).
\end{align*}
By exchangeability, for any integers $i < j$ in $[0, \vertices_n]$ with $|j-i| \le \vertices_n/2$, we have
\[\P\bigg(\Bfn(i) - \min_{i \le k \le j} \Bfn(k) - \varrho_n \frac{|j-i|}{\vertices_n} > \sqrt{(\sigma_n^2+\varrho_n + \edges_n) \frac{|j-i|}{\vertices_n}} x\bigg)
\le \exp\left(- \frac{x^2}{16}\right).\]
By integrating this tail bound applied to $x^{1/p}$, we obtain that for every $0 \le s < t \le 1$ with $|t-s| \le 1/2$,
\[\Es{\left(\frac{\Bfn(\lfloor\vertices_ns\rfloor) - \min_{s \le r \le t} \Bfn(\lfloor\vertices_nr\rfloor) - \varrho_n |t-s|}{(\sigma_n^2+\varrho_n + \edges_n)^{1/2} |t-s|^{1/2}}\right)^p} 
\le \int_0^\infty \exp\left(- \frac{x^{2/p}}{16}\right) \d x
.\]
Let $c(p)$ denote the right hand side. Since $a^p + b^p \le (a+b)^p \le 2^{p-1} (a^p + b^p)$ for every $a,b \ge 0$, we conclude that
\begin{align*}
&\Es{\left(\Bfn(\lfloor\vertices_ns\rfloor) - \min_{s \le r \le t} \Bfn(\lfloor\vertices_nr\rfloor)\right)^p}
\\
&\le 2^{p-1} \left(\Es{\left(\Bfn(\lfloor\vertices_ns\rfloor) - \min_{s \le r \le t} \Bfn(\lfloor\vertices_nr\rfloor) - \varrho_n |t-s|\right)^p} + \left(\varrho_n |t-s|\right)^p\right)
\\
&\le 2^{p-1} \left(c(p) (\sigma_n^2+\varrho_n + \edges_n)^{p/2} |t-s|^{p/2} + \varrho_n^p |t-s|^{p/2}\right)
\\
&\le C(p) ((\sigma_n^{2}+\edges_{n})^{1/2} + \varrho_n)^p |t-s|^{p/2},
\end{align*}
for some $C(p) > 0$. Upon changing the constant $C(p)$, the restriction $|t-s|\le 1/2$ can be lifted by the triangle inequality.

Note that the scaling $(\sigma_n^{2}+\edges_{n})^{1/2} + \varrho_n$ that appears here is larger than the claimed one $\sigma_n + \varrho_n$. However, in the case $d_{n}(1) = 0$, it holds that $\sigma_{n}^{2} = \sum_{k \ge 2} k(k-1)d_{n}(k) \ge \sum_{k \ge 2} kd_{n}(k) = \edges_{n}$ so both scalings are of the same order. In this case, for some other constant $C(p) > 0$, for every $s<t$, it holds
\[\Es{\left(\Bfn(\lfloor\vertices_ns\rfloor) - \min_{s \le r \le t} \Bfn(\lfloor\vertices_nr\rfloor)\right)^p}
\le C(p) (\sigma_n + \varrho_n)^p |t-s|^{p/2}.\]

In the general case, let us replace the degree sequence $d_n$ by $d_n'(k) = d_n(k) \ind{k\ne1}$ and couple the two forests with degree sequence $d_n$ and $d_n'$, by removing all the vertices with outdegree $1$ in $\fn$. Let us denote by $\vertices_{n}' = \vertices_{n}-d_{n}(1)$ and $\edges_{n}' = \edges_{n}-d_{n}(1)$ the number of vertices and edges respectively of the new forest ${\fn}\vphantom{T}'$; note that both $\sigma_n$ and $\varrho_{n}$ are unchanged. Note also that the increments which are removed from $\Bfn$ in this operation are all null.
Let $s<t$ and assume for notational convenience that $\vertices_{n}s$ and $\vertices_{n} t$ are both integers. In the forest ${\fn}\vphantom{T}'$, these two instants correspond to, say, $\vertices_{n}'s_{n}$ and $\vertices_{n}'t_{n}$ respectively, and $\vertices_{n}' |t_{n}-s_{n}|$ has the binomial distribution with parameters $\vertices_{n} |t-s|$ and $\vertices_{n}'/\vertices_{n}$ and is independent of the path $\Bfn$ to which we have removed the null increments. Then the moment of order $p/2$ of $\vertices_{n}' |t_{n}-s_{n}|$ is bounded above by some constant times 
$(\vertices_{n}' |t-s|)^{p/2}$, 
so we conclude from the preceding case that
\begin{align*}
&\Es{\left(\Bfn(\lfloor\vertices_ns\rfloor) - \min_{s \le r \le t} \Bfn(\lfloor\vertices_nr\rfloor)\right)^p}
\\
&= \Es{\Esc{\left(\Bfn\vphantom{B}'(\lfloor\vertices_n' s_{n}\rfloor) - \min_{s_{n} \le r \le t_{n}} \Bfn\vphantom{B}'(\lfloor\vertices_n' r\rfloor)\right)^p}{s_{n}, t_{n}}}
\\
&\le C(p) (\sigma_n + \varrho_n)^p \Es{|t_{n}-s_{n}|^{p/2}}
\\
&\le C'(p) (\sigma_n + \varrho_n)^p |t-s|^{p/2}.
\end{align*}
The proof of the standard Kolmogorov criterion then shows that the claim of the proposition holds when we replace $\Wfn$ by $\Bfn$.

We finally want to transfer this bound to $\Wfn$ by cyclic shift; some care is needed here. 
Indeed, recall that we may couple $\Wfn$ and $\Bfn$ in such a way that $\Bfn$ is obtained by cyclically shifting $\Wfn$ at a uniform random time $\vertices_n - i_n$ independent of $\Wfn$.
Fix again $s<t$ and assume that $i = \vertices_{n} s$ and $ j= \vertices_{n} t$ are integers. 
There are two cases: either $\vertices_n - i_n$ falls (strictly) between $i$ and $j$, or it does not. If it does not, then the path of $\Wfn$ between $i$ and $j$ is moved without change in $\Bfn$ so the claim follows from the preceding bound applied to the image of $i$ and $j$ after the cyclic shift operation, which are still at distance $|j-i|$ from each other. Assume henceforth that $\vertices_n - i_n$ does fall between $i$ and $j$, then the part of $\Wfn$ between $i$ and $\vertices_n - i_n$ is moved to the part of $\Bfn$ between $i + i_n$ and $\vertices_n$, and the part of $\Wfn$ between $\vertices_n - i_n$ and $j$ is moved to the part of $\Bfn$ between $0$ and $j + i_n - \vertices_n$, so
\begin{align*}
&\Wfn(i) - \min_{i \le k \le j} \Wfn(k) 
\\
&\le\left (\Wfn(i) - \min_{i \le k \le \vertices_n - i_n} \Wfn(k)\right) + \left(\Wfn(\vertices_n - i_n) - \min_{\vertices_n - i_n \le k \le j} \Wfn(k)\right)
\\
&= \left(\Bfn(i+i_{n}) - \min_{i+i_{n} \le k \le \vertices_n} \Bfn(k)\right) + \left(\Bfn(0) - \min_{0 \le k \le j +i_{n} - \vertices_n} \Wfn(k)\right).
\end{align*}
Using again that for every $a, b \ge 0$, we have $a^p + b^p \le (a+b)^p \le 2^{p-1} (a^p + b^p)$, and since $|j-i| = |\vertices_n - (i + i_n)| + |j + i_n - \vertices_n|$, we conclude from the bound on $\Bfn$.
\end{proof}

\subsection{Exponential tails for the height and width}
\label{sec:queues_exp}

Recall from Section~\ref{sec:def_arbres_marche} that for a vertex $x$ in a tree, we denote by $\LL(x)$, resp. $\RR(x)$, the number of vertices whose parent is a strict ancestor of $x$ and which lie strictly before, resp. strictly after, $x$ in depth-first search order, so $\LR(x) = \LL(x) + \RR(x)$ denotes the total number of vertices different from $x$ branching off the path $\llbracket \varnothing, x\llbracket$. In the case of a forest, we consider these quantities in the tree containing $x$. Recall finally the identity~\eqref{eq:def_X_discret}: if $W$ is the {\L}ukasiewicz path of the forest, then $\RR(x) = W(x) - \min_{0 \le y \le x} W(y)$.

Let us note that, as in Lemma~\ref{lem:queue_exp_min_pont}, our argument is not optimal since our bounds get worse when the number of trees $\varrho_n$ grows, but in most cases, this factor is negligible.

\begin{prop}\label{prop:queues_exp_Luka}
Let $x_n$ be a uniformly random vertex of $\fn$, then 
\[\Pr{\LR(x_n) \ge z (\sigma_n^2+\varrho_n)^{1/2}} \le 4 \exp\left(- \frac{z}{288}\right)\]
uniformly in $z \ge 1/2$ and $n \in \N$.
\end{prop}

\begin{proof}
Observe that the `mirror forest' obtained from $\fn$ by flipping the order of the children of every vertex has the same law as $\fn$ so the random variables $\LL(x_n)$ and $\RR(x_n)$ have the same law; of course they are not independent. Still, since $\LR(x_n) = \LL(x_n)+\RR(x_n)$, it suffices to consider the tail of $\RR(x_n)$ with $z/2$ in place of $z$ and to use a union bound. Recall also that $\Wfn$ can be obtained by cyclically shifting the bridge $\Bfn$ at the random time $i_n$ which is uniformly distributed in $\{1, \dots, \vertices_n\}$ and is independent of $\Wfn$. In this coupling, we have that
\begin{align*}
\Wfn(\vertices_n - i_n) - \min_{0 \le j \le \vertices_n - i_n} \Wfn(j)
&= \Bfn(\vertices_n) - \min_{0 \le j \le \vertices_n} \Bfn(j)
\\
&= -\varrho_n - \min_{0 \le j \le \vertices_n} \Bfn(j).
\end{align*}
Note that $\vertices_n - i_n$ is uniformly distributed in $\{0, \dots, \vertices_n - 1\}$ and is independent of $\Wfn$, so the vertex visited at this time has the same law as $x_n$ and therefore $\RR(x_n)$ has the same law as $-\varrho_n - \min_{0 \le j \le \vertices_n} \Bfn(j)$. 
By considering the two cases where the minimum of $\Bfn$ is achieved on the first half or on the second half, we 
see that this is bounded above by
\[\left(- \frac{\varrho_n}{2} - \min_{0 \le j \le \vertices_n/2} \Bfn(j)\right) + \left(- \frac{\varrho_n}{2} - \min_{\vertices_n/2 \le j \le \vertices_n} \Bfn(j) + \Bfn\left(\frac{\vertices_n}{2}\right)\right),\]
and the two terms on the right have the same law. 
Hence for every $z \ge 1/2$ we have after two union bounds
\[\Pr{\LR(x_n) \ge z (\sigma_n^2+\varrho_n)^{1/2}}
\le 4 \cdot \Pr{\min_{0 \le j \le \vertices_n/2} \Bfn(j) + \frac{\varrho_n}{2} \le - \frac{z (\sigma_n^2+\varrho_n)^{1/2}}{4}}.\]
Finally, by Lemma~\ref{lem:queue_exp_min_pont},
\begin{align*}
\Pr{\min_{i \le \vertices_n/2} \Bfn(i) + \frac{\varrho_n}{2} \le - \frac{z (\sigma_n^2+\varrho_n)^{1/2}}{4}}
&\le \exp\left(- \frac{(1/2)^3 (\frac{z (\sigma_n^2+\varrho_n)^{1/2}}{4})^2}{(\sigma_n^2+\varrho_n) + \frac{z (\sigma_n^2+\varrho_n)^{1/2}}{4}}\right)
\\
&\le \exp\left(- \frac{2^{-5} (\sigma_n^2+\varrho_n)^{1/2} z^2}{4(\sigma_n^2+\varrho_n)^{1/2} + z}\right)
.\end{align*}
Since both $z \ge 1/2$ and $(\sigma_n^2+\varrho_n)^{1/2} \ge 1$, then 
$4(\sigma_n^2+\varrho_n)^{1/2} \le 8 (\sigma_n^2+\varrho_n)^{1/2} z$ and $z \le (\sigma_n^2+\varrho_n)^{1/2} z$ so the denominator in the last exponential is bounded above by $9 (\sigma_n^2+\varrho_n)^{1/2} z$, which yields our claim.
\end{proof}

Let us next focus on the case of a single tree $\varrho_n = 1$. Let us denote by $\wid(\tn)$ the maximum over all $i \ge 1$ of the number of vertices of $\tn$ at distance $i$ from the root, called the \emph{width} of $\tn$, and by $\h(\tn)$ the greatest $i \ge 1$ such that there exists at least one vertex of $\tn$ at distance $i$ from the root, called the \emph{height} of $\tn$. The preceding result gives a universal tail bound on the first quantity. Indeed, with a similar reasoning as in the preceding proof, one can check that
\begin{equation}\label{eq:queue_min_pont}
\Pr{\min_{0 \le j \le \vertices_n/2} \Btn(j) + \frac{1}{2} \le - z \sigma_n}
\le \exp\left(- \frac{z}{48}\right),
\end{equation}
for all $z \ge 1$. Then following~\cite{AB12}, by replacing Equation~2 there by this bound, this yields the existence of two universal constants $c_1, c_2 > 0$ such that for every $z \ge 1$,
\[\Pr{\wid(\tn) \ge z \sigma_n} \le c_1 \e^{- c_2 z}.\]
By observing that $\wid(\tn) \times \h(\tn) \ge \edges_n$, we also get
\[\Pr{\h(\tn) \le \edges_n / (z \sigma_n)} \le c_1 \e^{-c_2 z}.\]
Following~\cite{AB12} again we may obtain the following upper bound for the height of the tree $\tn$:
\begin{equation}\label{eq:queue_exp_hauteur}
\Pr{\h(\tn) \ge z (\sigma_n^2+\edges_n)^{1/2}} \le c_1 \e^{-c_2 z},
\end{equation}
for some other constants $c_1, c_2 > 0$. 
However this scaling may be much larger than $\edges_n / \sigma_n$ outside the `finite-variance' regime, when $\sigma_n$ is of order $\edges_n^{1/2}$. We already mentioned in Section~\ref{sec:intro_arbres} that an upper-bound at the scaling $\edges_n / \sigma_n$ in full generality is not possible; nonetheless, the scaling $\edges_n / \sigma_n$ is correct for typical vertices.

\begin{prop}\label{prop:queues_exp_hauteur_typique_arbre}
There exists two universal constants $c_1, c_2 > 0$ such that the following holds: if $x_n$ is a uniformly random vertex of $\tn$, then its height $|x_n|$ satisfies
\[\Pr{|x_n| \ge z \edges_n / \sigma_n} \le c_1 \e^{- c_2 z}\]
uniformly for $z \ge 1$ and $n \in \N$.
\end{prop}

Our argument is the following: if $x_n$ is at height at least $z \edges_n / \sigma_n$ and if its ancestors reproduce according to the size-biased law as the $\xi_{d_n}$'s in the preceding section, then, according to~\eqref{eq:moyenne_biais_par_la_taille}, the number of vertices $\LR(x_n)$ branching off its ancestral line is in average at least $(z \edges_n / \sigma_n) \times (\sigma_n^2 / \edges_n) = z \sigma_n$, and we know from Proposition~\ref{prop:queues_exp_Luka} that this occurs with a sub-exponential probability.

\begin{proof}
According to Proposition~\ref{prop:queues_exp_Luka}, it is sufficient to bound
\[\Pr{|x_n| \ge z \frac{\edges_n}{\sigma_n} \text{ and } \LR(x_n) \le \frac{z}{2} \sigma_n}
= \sum_{h \ge z \edges_n / \sigma_n} \Pr{|x_n| = h \text{ and } \LR(x_n) \le \frac{z}{2} \sigma_n}.\]
Recall the definition of $\Cont(x_n) = ((k_{a_{i-1}(x_n)}, \chi_{a_i(x_n)}))_{1 \le i \le |x_n|}$ from~\eqref{eq:content} and note that we have $\LR(x_n) = \sum_{1 \le i \le |x_n|} (k_{a_{i-1}(x_n)} - 1)$. 
Let us fix $h \ge z \edges_n / \sigma_n$. We deduce from Lemma~\ref{lem:multi_epines_sans_remise} with a single random vertex that
\begin{align*}
& \Pr{|x_n| = h \text{ and } \LR(x_n) \le \frac{z}{2} \sigma_n}
\\
&= \sum_{(k_i, j_i)_{i=1}^h} \Pr{\Cont(x_n) = (k_i, j_i)_{i=1}^h} \ind{\sum_{i=1}^h (k_i - 1) \le \frac{z}{2} \sigma_n}
\\
&\le \sum_{\substack{(k_i, j_i)_{i=1}^h \\ \sum_{i=1}^h (k_i - 1) \le \frac{z}{2} \sigma_n}} 
\frac{1 + \sum_{i=1}^h (k_i - 1)}{\vertices_n} \cdot 
\Pr{\bigcap_{i=1}^h \left\{(\xi_{d_n}(i), \chi_{d_n}(i)) = (k_i, j_i)\right\}}
\\
&\le \frac{\frac{z}{2} \sigma_n + 1}{\edges_n} \Pr{\sum_{i=1}^h \left(\xi_{d_n}(i) - 1\right) \le \frac{z}{2} \sigma_n}.
\end{align*}
Let us write $X_n(i) = \xi_{d_n}(i) - 1$ in order to simplify the notation. Recall from~\eqref{eq:moyenne_biais_par_la_taille} that these random variables have mean $\sigma_n^2 / \edges_n$ so for $h \ge z \edges_n / \sigma_n$, 
the probability in the last line is bounded above by
\[\Pr{\sum_{i=1}^h \left(X_n(i) - \E[X_n(i)]\right) \le - \frac{h \sigma_n^2}{2 \edges_n}}.\]
The $X_n(i)$'s come from sampling balls without replacement; as we already mentioned, by~\cite[Proposition~20.6]{Ald85}, their sum satisfies any concentration inequality based on controlling the Laplace transform the similar sum when sampling with replacement does. In particular, we may apply~\cite[Theorem~2.7]{McD98}, and get
\begin{align*}
\Pr{\sum_{i=1}^h (X_n(i) - \E[X_n(i)]) \le - \frac{h \sigma_n^2}{2 \edges_n}}
&\le \exp\left(- \frac{(\frac{h \sigma_n^2}{2 \edges_n})^2}{2 h \frac{\Delta_n \sigma_n^2}{\edges_n} + \frac{2}{3} \frac{h \sigma_n^2}{2 \edges_n} \frac{\sigma_n^2}{\edges_n}}\right)
\\
&= \exp\left(- \frac{h \sigma_n^2}{8 \edges_n \Delta_n + \frac{4}{3} \sigma_n^2}\right).
\end{align*}
Observe that $\sigma_n^2 \le \edges_n \Delta_n$ so $8 \edges_n \Delta_n + \frac{4}{3} \sigma_n^2 \le 10 \edges_n \Delta_n$, hence
\begin{align*}
\Pr{|x_n| \ge z \frac{\edges_n}{\sigma_n} \text{ and } \LR(x_n) \le \frac{z}{2} \sigma_n}
&\le \frac{z \sigma_n + 2}{2 \edges_n} \sum_{h \ge z \edges_n / \sigma_n} \exp\left(- \frac{h \sigma_n^2}{10 \edges_n \Delta_n}\right)
\\
&\le \frac{z \sigma_n + 2}{2 \edges_n} 
\frac{\exp\left(- \frac{z \sigma_n}{10 \Delta_n}\right)}{1 - \exp\left(- \frac{\sigma_n^2}{10 \edges_n \Delta_n}\right)}.
\end{align*}
We next appeal to the following two bounds: first $(1-\e^{-t}) \ge t(1 - t/2) \ge 19t/20$ for every $0 \le t \le 1/10$, second $t \e^{-t} \le \e^{-t/2}$ for all $t \ge 0$. We thus have
\begin{align*}
\Pr{|x_n| \ge z \text{ and } \LR(x_n) \le \frac{z}{2} \sigma_n}
&\le \frac{z \sigma_n + 2}{2 \edges_n}
\frac{10 \Delta_n}{z \sigma_n} \exp\left(- \frac{z \sigma_n}{20 \Delta_n}\right)
\frac{200 \edges_n \Delta_n}{19 \sigma_n^2}
\\
&\le \frac{1000 \Delta_n^2}{19 \sigma_n^2} \left(1 + \frac{2}{z \sigma_n}\right)
\exp\left(- \frac{z \sigma_n}{20 \Delta_n}\right).
\end{align*}
Recall that we assume that $\Delta_n \ge 2$, which implies $\Delta_n^2 \le 2 \sigma_n^2$ and also $\sigma_n \ge 1$. We thus obtain for every $z \ge 1$,
\[\Pr{|x_n| \ge z \text{ and } \LR(x_n) \le \frac{z}{2} \sigma_n}
\le \frac{2000 \times 3}{19} \exp\left(- \frac{z}{20 \sqrt{2}}\right).\]
Jointly with the exponential bound from Proposition~\ref{prop:queues_exp_Luka}, this completes the proof.
\end{proof}

\begin{rem}\label{rem:queues_exp_hauteur_typique_foret}
In the more general case of forests, the same argument applies and the only difference lies in~\eqref{eq:epine_un_sommet} where in the numerator of the ratio in front of the probability, the term `$+1$' is more generally `$+\varrho_n$'. 
Then in the last display of the preceding proof, the term `$3$' is replaced by `$1+\varrho_n/\sigma_n$'. If this ratio is bounded, as in Theorem~\ref{thm:marginales_hauteur}, then the result of Proposition~\ref{prop:queues_exp_hauteur_typique_arbre} stills holds, except that $c_1$ now depends on the sequences $(\varrho_n)_{n \ge 1}$ and $(d_n)_{n \ge 1}$, but it can be chosen independently of $n$ and $z \ge 1$.
\end{rem}

\begin{rem}\label{rem:k_ary}
Recall the particular case of uniformly random $k_{n}$-ary trees discussed in the introduction, in which every internal vertex has exactly $k_{n}$ offspring, so $\edges_{n} = n k_{n}$ and $\sigma_{n}^{2} = n k_{n} (k_{n}-1)$. In this case, Proposition~\ref{prop:queues_exp_hauteur_typique_arbre} extends to the maximal height of the tree, without appealing to Lemma~\ref{lem:multi_epines_sans_remise} since in this case, if a vertex $x$ lies at generation $z \edges_{n}/\sigma_{n}$, say, then $\LR(x) = z (k_{n}-1) \edges_{n}/\sigma_{n} = z \sigma_{n}$. Then, by~\eqref{eq:queue_min_pont} and the Vervaat transform, the probability that there exists such a vertex is bounded by $c_1 \e^{- c_2 z}$, where $c_{1}, c_{2} > 0$ are universal.
With this bound at hand, it is not difficult then to extend the convergence of the reduced trees provided by Theorem~\ref{thm:convergence_arbre_reduit} to the convergence~\eqref{eq:arbres_k_aires} by proving Equation~25 in~\cite{Ald93}. Precisely, one can show that for every $\varepsilon>0$, one can fix $q$ large enough so that with high probability as $n\to\infty$, each tree in the forest consisting of the complement in $\fn$ of the ancestors of the $q$ random vertices has height smaller $\varepsilon \edges_n/\sigma_n$.
We hope to obtain a more general result in the future and therefore refrain to provide the details here in this case.
\end{rem}

\subsection{Convergence of reduced forests}
\label{sec:marginales_hauteur}

Let us close this section with the proof of Theorem~\ref{thm:marginales_hauteur}.
Recall that for $\varrho \in [0,\infty)$, we denote by $X^\varrho$ the first-passage Brownian bridge from $0$ to $-\varrho$ with duration $1$ and that we denote by $\tX_t = X^\varrho_t - \min_{0 \le s \le t} X^\varrho_s$ for every $t \in [0,1]$.
Let us assume that $\lim_{n \to \infty} \sigma_n^{-1} \varrho_n = \varrho$ and that $\lim_{n \to \infty} \sigma_n^{-1} \Delta_n = 0$. We have shown with Proposition~\ref{prop:convergence_Luka} the convergence
\[\left(\frac{1}{\sigma_n} \Wfn(\vertices_n t) ; t \in [0,1]\right)
\cvloi
\left(X^\varrho(t) ; t \in [0,1]\right).\]
We aim at showing that, jointly with this convergence, if for $q \ge 1$, we sample $U_1, \dots, U_q$ i.i.d. uniform random variables in $[0,1]$ independently of the rest, and if we denote by $0 = U_{(0)} < U_{(1)} < \dots < U_{(q)}$ their ordered statistics, then we have
\[\frac{\sigma_n}{2 \edges_n} \left(\Hfn(\vertices_n U_{(i)}), \inf_{U_{(i-1)} \le t \le U_{(i)}} \Hfn(\vertices_n t)\right)_{1 \le i \le q}
\cvloi
\left(\tX_{U_{(i)}}, \inf_{U_{(i-1)} \le t \le U_{(i)}} \tX_t\right)_{1 \le i \le q}.\]
The proof is inspired from the work of Broutin \& Marckert~\cite{BM14} which itself finds its root in the work of Marckert \& Mokkadem~\cite{MM03}; the ground idea is to compare the process $\Hfn$ which describes the height of the vertices with $\Wfn$ which counts the number of vertices branching off to the right of their ancestral lines. As opposed to these works, these two processes have different scaling here so one has to be more careful.

\begin{proof}[Proof of Theorem~\ref{thm:marginales_hauteur}]
Let $U$ have the uniform distribution on $[0,1]$ independently of the forest and let $x_n$ be the $\lfloor \vertices_n U\rfloor$'th vertex of $\fn$ in lexicographical order, so it has the uniform distribution in $\fn$. Then $\Hfn(\lfloor \vertices_n U\rfloor) = |x_n|$ denotes its generation in its tree, whereas $\RR(x_n)$ as defined in~\eqref{eq:def_X_discret} is the number of individuals branching off strictly to the right of its ancestral line in its tree. According to Proposition~\ref{prop:convergence_Luka} and~\eqref{eq:def_X_discret}, the process $\sigma_n^{-1} \RR(\lfloor \vertices_n \cdot\rfloor)$ converges in distribution towards $\tX$; we claim that
\begin{equation}\label{eq:ecart_longueur_LR}
\left|\frac{1}{\sigma_n} \RR(x_n) - \frac{\sigma_n}{2 \edges_n} |x_n| \right| \cvproba 0.
\end{equation}
Recall the notation $\LR(x_n) = \LL(x_n) + \RR(x_n)$ for the total number of individuals branching off the ancestral line of $x_n$ and recall that, by symmetry, $\LL(x_n)$ and $\RR(X_n)$ have the same law (but they are not independent in general). 
Fix $\varepsilon, \eta > 0$. Let $K > 0$ be such that
\[\Pr{\LR(x_n) \le K \sigma_n \text{ and } |x_n| \le K \edges_n / \sigma_n} \ge 1-\eta\]
for every $n$ large enough. This is ensured e.g. by Proposition~\ref{prop:queues_exp_Luka} and Remark~\ref{rem:queues_exp_hauteur_typique_foret}. Then 
the probability that $|\frac{1}{\sigma_n}\RR(x_n) - \frac{\sigma_n}{2 \edges_n}|x_n|| > \varepsilon$ is bounded above by
\[\eta + \Pr{\left|\RR(x_n) - \frac{\sigma_n^2}{2 \edges_n}|x_n|\right| > \varepsilon \sigma_n, 
\LR(x_n) \le K \sigma_n,
\text{ and } |x_n| \le K \edges_n / \sigma_n}.\]

We next proceed similarly to the previous proof, appealing to the spinal decomposition obtained in Lemma~\ref{lem:multi_epines_sans_remise} with $q=1$. Recall the notation from this lemma;
consider the event that $|x_n| = h$ and that for all $0 \le i < h$, the ancestor of $x_n$ at generation $i$ has $k_i$ offspring and the $j_i$'th one is the ancestor of $x_n$ at generation $i+1$; according to~\eqref{eq:epine_un_sommet} its probability is bounded by
\[\frac{\varrho_n + \sum_{i=1}^h (k_i - 1)}{\vertices_n}
\cdot \Pr{\bigcap_{i \le h} \left\{(\xi_{d_n}(i), \chi_{d_n}(i)) = (k_i, j_i)\right\}}.\]
By decomposing according to the height of $x_n$ and taking the worst case, we 
see that the probability that $|\frac{1}{\sigma_n}\RR(x_n) - \frac{\sigma_n}{2 \edges_n}|x_n|| > \varepsilon$ is bounded above by
\[\eta + 
\frac{K \edges_n}{\sigma_n} \frac{\varrho_n + K \sigma_n}{\vertices_n} \sup_{h \le K \edges_n/\sigma_n} 
\Pr{\left|\sum_{i \le h} \left(\xi_{d_n}(i) - \chi_{d_n}(i)\right) - \frac{\sigma_n^2}{2 \edges_n} h\right| > \varepsilon \sigma_n}.\]
From our assumption, $\frac{K \edges_n}{\sigma_n} \frac{\varrho_n + K \sigma_n}{\vertices_n}$ converges to $K (\varrho + K)$. From the triangle inequality, the last probability is bounded above by
the sum of
\[\Pr{\left|\sum_{i \le h} \left(\left(\xi_{d_n}(i) - \chi_{d_n}(i)\right) - \frac{\xi_{d_n}(i) - 1}{2}\right)\right| > \frac{\varepsilon \sigma_n}{2}}\]
and
\[\Pr{\left|\sum_{i \le h} \frac{\xi_{d_n}(i) - 1}{2} - \frac{\sigma_n^2}{2 \edges_n} h\right| > \frac{\varepsilon \sigma_n}{2}}.\]
Recall that the $\xi_{d_n}(i)$'s are identically distributed, with mean and variance given in~\eqref{eq:moyenne_biais_par_la_taille}. Furthermore, as discussed in the end of Section~\ref{sec:epine} these random variables are obtained by successive picks without replacement in an urn, and therefore are negatively correlated. In particular, the variance of their sum is bounded by the sum of their variances, see e.g.~\cite[Proposition~20.6]{Ald85}. 
The Markov inequality then yields for every $h \le K \edges_n / \sigma_n$
\[\Pr{\left|\sum_{i \le h} \frac{\xi_{d_n}(i) - 1}{2} - \frac{\sigma_n^2}{2 \edges_n} h\right| > \frac{\varepsilon \sigma_n}{2}}
\le \frac{h}{\varepsilon^2 \sigma_n^2} \Var\left(\xi_{d_n}(1)-1\right)
\le \frac{K \Delta_n}{\varepsilon^2 \sigma_n},\]
which converges to $0$. Moreover, conditionally on the $\xi_{d_n}(i)$'s, the random variables $\xi_{d_n}(i) - \chi_{d_n}(i)$ are independent and uniformly distributed on $\{0, \dots, \xi_{d_n}(i)-1\}$, with mean $(\xi_{d_n}(i) - 1)/2$ and variance $(\xi_{d_n}(i)^2 - 1)/12$. Similarly,
the Markov inequality applied conditionally on the $\xi_{d_n}(i)$'s yields for every $h \le K \edges_n / \sigma_n$:
\begin{align*}
&\Pr{\left|\sum_{i \le h} \left(\left(\xi_{d_n}(i) - \chi_{d_n}(i)\right) - \frac{\xi_{d_n}(i) - 1}{2}\right)\right| > \frac{\varepsilon \sigma_n}{2}}
\\
&\qquad\le \frac{4}{\varepsilon^2 \sigma_n^2} \cdot \Es{\sum_{i \le h} \frac{\xi_{d_n}(i)^2 - 1}{12}}
\\
&\qquad\le \frac{\Delta_n+1}{3 \varepsilon^2 \sigma_n^2} \cdot h \cdot \Es{\xi_{d_n}(1) - 1}
\\
&\qquad\le K \frac{\Delta_n+1}{3 \varepsilon^2 \sigma_n},
\end{align*}
which also converges to $0$. This completes the proof of~\eqref{eq:ecart_longueur_LR}, which, combined with Proposition~\ref{prop:convergence_Luka} and~\eqref{eq:def_X_discret}, yields the convergence
\[\frac{\sigma_n}{2 \edges_n} \left(\Hfn(\vertices_n U_{(i)})\right)_{0 \le i \le q}
\cvloi
\left(\tX_{U_{(i)}}\right)_{0 \le i \le q}\]
for any $q \ge 1$ fixed.

In order to obtain the full statement of the proposition, we need to prove the following: Assume that the forest $\fn$ reduced to the ancestors of $q$ i.d.d. vertices has $q$ leaves (this occurs with high probability since the height of such a vertex is at most of order $\edges_n / \sigma_n = o(\edges_n)$ as we have seen) and a random number, say, $b \in \{1, \dots, q-1\}$ of branchpoints; then remove these $b$ branchpoints from the reduced forest to obtain a collection of $q+b$ single branches. Then we claim that uniformly for $1 \le i \le q+b$, the length of the $i$'th branch multiplied by $\frac{\sigma_n}{2 \edges_n}$ is close to $\sigma_n^{-1}$ times the number of vertices branching off strictly to the right of this path in the original forest. Since this number is encoded by the {\L}ukasiewicz path, the proposition then follows from Proposition~\ref{prop:convergence_Luka}. Such a comparison follows similarly as in the case $q=1$ above from the spinal decomposition of Lemma~\ref{lem:multi_epines_sans_remise}: now we have to consider not only the length $h$ of a single branch, but those $h_1, \dots, h_{q+b}$ of all the branches, which is compensated by the factor $(\sigma_n / \vertices_n)^{q+b}$ in this lemma, the other terms before the probability are bounded. We leave the details to the reader.
\end{proof}

\section{On the label process}
\label{sec:etiquettes}

The aim of this section is to prove Theorems~\ref{thm:tension_etiquettes} and~\ref{thm:convergence_etiquettes} on the label process $\Lfn$ of a labelled forest $(\fn, \ell)$ sampled uniformly at random in $\LFn$. Let us first prove Theorem~\ref{thm:tension_etiquettes} which asserts that the sequence
\[\left((\sigma_n + \varrho_n)^{-1/2} \Lfn(\vertices_n t) ; t \in [0,1]\right)_{n \ge 1}\]
is tight in the space $\mathscr{C}([0,1], \R)$. We shall rely on Proposition~\ref{prop:Holder_marche_Luka} and adapt the argument from~\cite{Mar18b}.
Then in Section~\ref{sec:serpent_brownien} we prove that in the case $\varrho_n \gg \sigma_n$ of a large number of trees, this process converges in distribution towards a Brownian bridge as stated in Theorem~\ref{thm:convergence_etiquettes}~\eqref{thm:convergence_etiquettes_arbre} by identifying the limit of the finite-dimensional marginals. On the other hand in the case of no macroscopic degree, where both $\varrho_n \sim \varrho \sigma_n$ and $\Delta_n \ll \sigma_n$, we deduce Theorem~\ref{thm:convergence_etiquettes}~\eqref{thm:convergence_etiquettes_carte_disque} from the convergence of the finite-dimensional \emph{random} marginals which is itself proved in Section~\ref{sec:marginales_labels}, appealing to Theorem~\ref{thm:marginales_hauteur}.
Throughout this section we shall make an extensive use of Lemma~\ref{lem:multi_epines_sans_remise}.

\subsection{Tightness of the label process}
\label{sec:tension_etiquettes}

Theorem~\ref{thm:tension_etiquettes} extends~\cite[Proposition~7]{Mar18b} which is restricted to the case of a single tree and in a `finite-variance regime', when $\sigma_n^2$ is of order $\edges_n$. Many arguments generalise here so shall only briefly recall them and focus on the main difference.
First, we shall need a technical result which resembles~\cite[Corollary~3]{Mar18b}, but a slight adaptation is needed here.

Recall the notation $\chi_z \in \{1, \dots, k_{pr(z)}\}$ for the relative position of a vertex $z \in \fn$ amongst its siblings as well as the interval notation $\mathopen{\rrbracket} x, y\rrbracket$ for the simple path going from the vertex $x$ (excluded) to the vertex $y$.

\begin{lem}\label{lem:bon_evenement_tension_labels}
Consider the following event which we denote by $\mathcal{E}_n$: for every pair of vertices $\hat{x}$ and $x$ such that $\hat{x}$ is an ancestor of $x$ and $\#\{z \in \mathopen{\rrbracket} \hat{x}, x\rrbracket : k_{pr(z)} \ge 2\} > 2\cdot16^{2} \ln(4 \vertices_{n})$, it holds that
\[\frac{\#\{z \in \mathopen{\rrbracket} \hat{x}, x\rrbracket : \chi_z = 1 \text{ and } k_{pr(z)} \ge 2\}}{\#\{z \in \mathopen{\rrbracket} \hat{x}, x\rrbracket : k_{pr(z)} \ge 2\}}
\le \frac{3}{4}.\]
Then $\P(\mathcal{E}_n^c) \le 2 (\vertices_n-d_{n}(1))^{- 2}$.
\end{lem}

In words, if the number of vertices with degree different from $1$ tends to infinity, then 
with high probability, there is no branch along which, amongst the individuals which are not single child, as soon as this quantity is at least logarithmic in $\vertices_{n}$, the proportion of individuals which are the left-most (or right-most by symmetry) child of their parent is larger than $3/4$. This will be needed to use a symmetry argument in the proof of Theorem~\ref{thm:tension_etiquettes}; note that we exclude the individuals with only one child since the label increment in this case is null.

\begin{proof}
Let us first explain the values appearing in the lemma. 
Since the statement completely ignores the individuals with only one child, then we may replace the degree sequence $d_n$ by $d_n'(k) = d_n(k) \ind{k\ne1}$ and couple the two labelled forests with degree sequence $d_n$ and $d_n'$, by removing all the vertices with outdegree $1$ in $\fn$. Let us denote by $\vertices_{n}' = \vertices_{n}-d_{n}(1)$, $\edges_{n}' = \edges_{n}-d_{n}(1)$, and $d_{n}'(0) = d_{n}(0)$ the number of vertices, edges, and leaves respectively of the new forest ${\fn}\vphantom{T}'$.
Let us set
\[g_n = \frac{d_{n}'(0)}{8 \vertices_{n}'},
\qquad
c_n = 2g_n + \frac{\vertices_n'-d_{n}'(0)}{\edges_n'},
\qquad
l_n 
= g_n^{-2} \ln(g_n^{-1} (\vertices_n')^2).\]
Then $g_{n} \le 1/8$ and since each inner vertex has degree at least $2$, then $\vertices_n'-d_{n}'(0) \le \edges_n'/2$, and thus $c_{n} \le 3/4$. On the other hand we also have $g_{n} \ge 1/16$ so $l_{n} \le 16^{2} \ln(16 (\vertices_n')^2) \le 2\cdot16^{2} \ln(4 \vertices_{n})$.
We shall prove that the probability to find a pair $(\hat{x}, x)$ such that $\hat{x}$ is an ancestor of $x$ and $\#\{z \in \mathopen{\rrbracket} \hat{x}, x\rrbracket : k_{pr(z)} \ge 2\} > l_{n}$ and
\[\frac{\#\{z \in \mathopen{\rrbracket} \hat{x}, x\rrbracket : \chi_z = 1 \text{ and } k_{pr(z)} \ge 2\}}{\#\{z \in \mathopen{\rrbracket} \hat{x}, x\rrbracket : k_{pr(z)} \ge 2\}}
> c_{n}\]
is smaller than or equal to $2 (\vertices_n')^{- 2}$. Let $\mathcal{E}_{n}' \supset \mathcal{E}_{n}^{c}$ denote this event, this implies our claim.

For a vertex $x$ in the forest ${\fn}\vphantom{T}'$ and $1 \le i \le |x|$, let us denote by $\alpha_i(x)$ the ancestor of $x$ at height $|x|-i+1$, so $\alpha_1(x)$ is $x$ itself, $\alpha_2(x)$ is its parent, etc. We may then rewrite the event $\mathcal{E}_{n}'$ as
\[\mathcal{E}_n'
= \bigcup_{x \in {\fn}\vphantom{T}'} \bigcup_{l_n \le l \le |x|} \bigg\{\sum_{i=1}^l \ind{\chi_{\alpha_i(x)} = 1} > c_n l\bigg\}.\]
Let $x_n$ be a vertex sampled uniformly at random in ${\fn}\vphantom{T}'$; a union bound yields
\[\Pr{\mathcal{E}_n'}
\le \vertices_n' \cdot \Pr{\bigcup_{l = l_n}^{|x_n|} \left\{\sum_{i=1}^l \ind{\chi_{\alpha_i(x_n)} = 1} > c_n l\right\}}.\]
Appealing to Lemma~\ref{lem:multi_epines_sans_remise}, since the ratio $(\varrho_n + \sum_{i=1}^h (k_i - 1))/\vertices_n'$ in~\eqref{eq:epine_un_sommet} is bounded by $1$, then the previous right-hand side is bounded above by
\[\vertices_n' \sum_{h \le \vertices_n'} \sum_{l = l_n}^h \Pr{\sum_{i=1}^l \ind{\chi_{d_n'}(i) = 1} > c_n l}.\]
Recall that $c_n = 2g_n + \frac{\vertices_n'-d_{n}'(0)}{\edges_n'}$, then a union bound yields
\begin{align*}
&\Pr{\sum_{i=1}^l \ind{\chi_{d_n'}(i) = 1} > c_n l}
\\
&= \Pr{\sum_{i=1}^l \ind{\chi_{d_n'}(i) = 1} - l \frac{\vertices_n'-d_{n}'(0)}{\edges_n'} > 2g_n l}
\\
&\le \Pr{\sum_{i=1}^l \ind{\chi_{d_n'}(i) = 1} - \sum_{i=1}^l \frac{1}{\xi_{d_n'}(i)} > g_nl} 
+ \Pr{\sum_{i=1}^l \frac{1}{\xi_{d_n'}(i)} - l \frac{\vertices_n'-d_{n}'(0)}{\edges_n'} > g_nl}.
\end{align*}
Note that each $\xi_{d_n'}(i)^{-1}$ takes values in $[0,1]$ and has mean
\[\Es{\xi_{d_n'}(i)^{-1}} = \sum_{k \ge 1} k^{-1} \frac{k d_n'(k)}{\edges_{n}'} = \frac{\vertices_{n}'-d_{n}'(0)}{\edges_{n}'}.\]
If these variables were obtained by sampling \emph{with} replacement, then we could apply a well known concentration result, see e.g.~\cite[Theorem~2.3]{McD98} to obtain
\[\Pr{\sum_{i=1}^l \xi_{d_n'}(i)^{-1} - l \frac{\vertices_{n}'-d_{n}'(0)}{\edges_{n}'} > g_nl}
\le \exp\left(- 2 g_n^2 l\right).\]
As already noted, this remains true here since this concentration is obtained by controlling the Laplace transform of $\sum_{i=1}^l \xi_{d_n'}(i)^{-1}$ and the expectation of any convex function of this sum is bounded above by the corresponding expectation when sampling with replacement, see e.g.~\cite[Proposition~20.6]{Ald85}.
Further, conditionally on the $\xi_{d_n'}(i)$'s, the variables $\ind{\chi_{d_n'}(i) = 1}$ are independent and Bernoulli distributed, with parameter $\xi_{d_n'}(i)^{-1}$ respectively, so we have similarly
\[\Pr{\sum_{i=1}^l \ind{\chi_{d_n'}(i) = 1} - \sum_{i=1}^l \xi_{d_n'}(i)^{-1} > g_nl}
\le \exp\left(- 2 g_n^2 l\right).\]
Since $g_n \in (0,1/2)$, then it holds that $1 - \exp(- 2 g_n^2) \ge g_n^2$, so we obtain
\begin{align*}
\Pr{\mathcal{E}_n'}
&\le \vertices_n \sum_{h \le \vertices_n'} \sum_{l = l_n}^h 2 \exp\left(- 2 g_n^2 l\right)
\\
&\le 2 (\vertices_n')^2 \frac{\exp(- 2 g_n^2 l_n)}{1 - \exp(- 2 g_n^2)}
\\
&\le 2 (\vertices_n')^2 g_n^{-2} \exp(- 2 g_n^2 l_n).
\end{align*}
Finally, we have chosen $l_n 
= g_n^{-2} \ln(g_n^{-1} (\vertices_n')^2)$,
so $\P(\mathcal{E}_n') 
\le 2 (\vertices_n')^{- 2}$.
\end{proof}

We may now prove the tightness of the label process, relying on this result and Proposition~\ref{prop:Holder_marche_Luka}.

\begin{proof}[Proof of Theorem~\ref{thm:tension_etiquettes}]
Suppose first that the number $\vertices_n-d_{n}(1)$ of vertices with outdegree different from $1$ does not tend to infinity; upon extracting a subsequence, assume that it is uniformly bounded. Then the maximal degree $\Delta_{n}$, and so the maximal label increment along an edge, is bounded (since the number of leaves is), and there are a bounded number of edges along which the label increment is not zero. In this case, if furthermore the number $\varrho_{n}$ of trees is also bounded, then it is clear that the label process is tight, and the scaling factor is of constant order.
On the other hand, if $\vertices_n-d_{n}(1)$ is bounded but $\varrho_{n}$ tends to infinity, then the largest gap between labels in the same tree is uniformly bounded, so it tends to $0$ after the scaling, only the labels of the roots of the trees can give a non zero contribution; the latter precisely converge after scaling to a Brownian bridge as discussed more precisely in the proof of Theorem~\ref{thm:convergence_etiquettes}~\eqref{thm:convergence_etiquettes_arbre} below.

We henceforth assume that $\vertices_n-d_{n}(1) \to \infty$ as $n \to \infty$. 
Fix $\varepsilon > 0$ arbitrarily small, let $C$ be as in Proposition~\ref{prop:Holder_marche_Luka} and let $\mathcal{E}_n$ be the intersection of the event in this proposition and that of Lemma~\ref{lem:bon_evenement_tension_labels}, whose probability is therefore at least $1-2\varepsilon$ for every $n$ large enough. 
We claim that for every $q > 4$, for every $\beta \in (0, q/4-1)$, there exists a constant $C> 0$, which depends on $q$, $\beta$ and possibly $(d_n)_{n \ge 0}$, such that for every $n$ large enough, for every $0 \le s \le t \le 1$, it holds that
\begin{equation}\label{eq:moments_label_Kolmogorov}
\Esc{\left|\Lfn(\vertices_n s) - \Lfn(\vertices_n t)\right|^q}{\mathcal{E}_n}
\le C \cdot (\sigma_n + \varrho_n)^{q/2} \cdot |t - s|^{1+\beta}.
\end{equation}
The standard Kolmogorov criterion then implies that for every $\gamma \in (0, 1/4)$, for another constant $K>0$,
\[\lim_{K \to \infty} \limsup_{n \to \infty} \Prc{\sup_{s \ne t} \frac{|\Lfn(\vertices_n s) - \Lfn(\vertices_n t)|}{(\sigma_n + \varrho_n)^{1/2} \cdot |t - s|^\gamma} > K}{\mathcal{E}_n} = 0.\]
Then the same holds for the unconditioned probability and this implies the tightness as claimed.

Exactly as in the proof of Proposition~\ref{prop:Holder_marche_Luka}, it suffices to prove~\eqref{eq:moments_label_Kolmogorov} for the `one-reduced' forest ${\fn}\vphantom{T}'$ with degree sequence $d_n'(k) = d_n(k) \ind{k\ne1}$ obtained by removing all the vertices with outdegree $1$ in $\fn$. Indeed, as there, these vertices induce a null label increment and they are located uniformly at random in the forest; therefore, with the notation of the proof of Proposition~\ref{prop:Holder_marche_Luka}, the suitable rescaled time intervals $[s,t]$ and $[s_{n}, t_{n}]$ corresponding to the same vertices respectively in $\fn$ and in ${\fn}\vphantom{T}'$ are of the same order in the sense that $\E[|t_{n}-s_{n}|^{1+\beta}] \le c\cdot |t-s|^{1+\beta}$ for some constant $c$ independent of $s$ and $t$; this allows to deduce the bound~\eqref{eq:moments_label_Kolmogorov} from ${\fn}\vphantom{T}'$ to $\fn$ (with another constant $C$). Note that the event $\mathcal{E}_{n}$ is not affected by the reduction.

Let us summarise: we assume henceforth that the degree sequence is such that $d_{n}(1) = 0$ and $\vertices_{n} \to \infty$ as $n\to \infty$, and our aim is to prove~\eqref{eq:moments_label_Kolmogorov}. The rest of the proof follows closely that of~\cite[Proposition~7]{Mar18b}.
We may, and do, suppose that $\vertices_n s$ and $\vertices_n t$ are integers.
Let us view $\fn$ as a tree by attaching all the roots to an extra root vertex, let $x$ and $y$ be the vertices visited at time $\vertices_n s$ and $\vertices_n t$ respectively, and let $\hat{x}$ and $\hat{y}$ be the children of their last common ancestor which are ancestor of $x$ and $y$ respectively. Note that if $x$ and $y$ belong to different trees of $\fn$, then $\hat{x}$ and $\hat{y}$ are the roots of these two trees and then the term $\chi_{\hat{y}} - \chi_{\hat{x}}$ below counts the number of other roots strictly between them, plus one. 
Let us decompose the label increment $\Lfn(\vertices_n s) - \Lfn(\vertices_n t) = \ell(x)-\ell(y)$ as follows:
\[\ell(x)-\ell(y)
= \sum_{z \in \mathopen{\rrbracket} \hat{x}, x \mathclose{\rrbracket}} (\ell(z)-\ell(pr(z)))
+ (\ell(\hat{y})-\ell(\hat{y}))
+ \sum_{z \in \mathopen{\rrbracket} \hat{y}, y \mathclose{\rrbracket}} (\ell(z)-\ell(pr(z))).\]
Then it was argued in~\cite{Mar18b}, see Equation~22 and the next few lines there, that $\E[|\ell(x)-\ell(y)|^q \mid \fn]$ is bounded above by some constant times
\[\bigg(\sum_{z \in \mathopen{\rrbracket} \hat{x}, x \mathclose{\rrbracket}} (k_{pr(z)}-\chi_{z}) + (\chi_{\hat{y}} - \chi_{\hat{x}})\bigg)^{q/2}
+ \bigg(\sum_{z \in \mathopen{\rrbracket} \hat{y}, y \mathclose{\rrbracket}} \chi_{z}\bigg)^{q/2}.\]
Furthermore, from Lemma~\ref{lem:codage_marche_Luka} and then Proposition~\ref{prop:Holder_marche_Luka}, one gets that 
there exists $C > 0$ such that, for every $n \ge 1$, with probability at least $1-\varepsilon$, it holds
\begin{align*}
\sum_{z \in \mathopen{\rrbracket} \hat{x}, x \mathclose{\rrbracket}} (k_{pr(z)}-\chi_{z}) + (\chi_{\hat{y}} - \chi_{\hat{x}})
&= \Wfn(\vertices_n s) - \inf_{r \in [s,t]} \Wfn(\vertices_n r)
\\
&\le C \cdot (\sigma_n + \varrho_n) \cdot |t-s|^{(1-\varepsilon)/2}.
\end{align*}
uniformly for $0 \le s < t \le 1$.
By choosing $\varepsilon$ small enough, we deduce that
\[\bigg(\sum_{z \in \mathopen{\rrbracket} \hat{x}, x \mathclose{\rrbracket}} (k_{pr(z)}-\chi_{z}) + (\chi_{\hat{y}} - \chi_{\hat{x}})\bigg)^{q/2}
\le C^{q/2} \cdot (\sigma_n + \varrho_n)^{q/2} \cdot |t-s|^{1+\beta},\]
uniformly for $0 \le s < t \le 1$ with probability $1-\varepsilon$.

We next want a similar bound for the moments of the other term $\sum_{z \in \mathopen{\rrbracket} \hat{y}, y \mathclose{\rrbracket}} \chi_{z}$. We would like to proceed symmetrically, using the `mirror forest' obtained by flipping the order of the children of every vertex, but there is a difference with the first term: the quantity $\sum_{z \in \mathopen{\rrbracket} \hat{x}, x \mathclose{\rrbracket}} (k_{pr(z)}-\chi_{z})$ counts the individuals which lie strictly to the right of the branch $\mathopen{\rrbracket} \hat{x}, x \mathclose{\rrbracket}$ in the sense that their parent belongs to $\mathopen{\llbracket} \hat{x}, x \mathclose{\llbracket}$ but not themselves, and they lie after this branch in the depth-first search order; on the other hand, the term $\sum_{z \in \mathopen{\rrbracket} \hat{y}, y \mathclose{\rrbracket}} \chi_{z}$ counts similarly the number, say $\LL(\mathopen{\rrbracket} \hat{y}, y \mathclose{\rrbracket})$ of individuals which lie strictly to the left of the branch $\mathopen{\rrbracket} \hat{y}, y \mathclose{\rrbracket}$ \emph{plus} the length $\# \mathopen{\rrbracket} \hat{y}, y \mathclose{\rrbracket}$ of this branch. In other words, whilst individuals which are the right-most child of their parent are not counted in the first term (since $\ell(zk_{z}) = \ell(z)$), those which are the left-most child of their parent \emph{are} counted in the second term (since $\ell(z1) \ne \ell(z)$ in general); Lemma~\ref{lem:bon_evenement_tension_labels} allows us to control this number as follows.
Recall that we assume $d_{n}(1)=0$. Under the event $\mathcal{E}_n$, if the branch $\mathopen{\rrbracket} \hat{y}, y \mathclose{\rrbracket}$ has length greater than $l_n = 2\cdot16^{2} \ln(4 \vertices_{n})$, then the proportion of individuals which are the first child of their parent is at most $3/4$; all other vertices (a proportion at least $1/4$) contribute to $\LL(\mathopen{\rrbracket} \hat{y}, y \mathclose{\rrbracket})$ so there are at most $\LL(\mathopen{\rrbracket} \hat{y}, y \mathclose{\rrbracket})$ of them and therefore
on the event $\mathcal{E}_n$, it holds that
\begin{align*}
\# \mathopen{\rrbracket} \hat{y}, y \mathclose{\rrbracket} + \LL(\mathopen{\rrbracket} \hat{y}, y \mathclose{\rrbracket})
&\le 4 \LL(\mathopen{\rrbracket} \hat{y}, y \mathclose{\rrbracket}) \ind{\# \mathopen{\rrbracket} \hat{y}, y \mathclose{\rrbracket} \ge l_n}
+ \left(l_n + \LL(\mathopen{\rrbracket} \hat{y}, y \mathclose{\rrbracket})\right) \ind{\# \mathopen{\rrbracket} \hat{y}, y \mathclose{\rrbracket} < l_n}
\\
&\le l_n + 5 \LL(\mathopen{\rrbracket} \hat{y}, y \mathclose{\rrbracket}).
\end{align*}
Now, using the mirror forest, the previous bound on the left branch yields
\[\E[\LL(\mathopen{\rrbracket} \hat{y}, y \mathclose{\rrbracket})^{q/2}] \le C (\sigma_n + \varrho_n)^{q/2} |t-s|^{1+\beta}\]
for some $C>0$ and so, for some other constant $C > 0$,
\begin{align*}
\Es{\bigg(\sum_{z \in \mathopen{\rrbracket} \hat{y}, y \mathclose{\rrbracket}} \chi_{z}\bigg)^{q/2} \1_{\mathcal{E}_n}}
&= \Es{\big(\# \mathopen{\rrbracket} \hat{y}, y \mathclose{\rrbracket} + \LL(\mathopen{\rrbracket} \hat{y}, y \mathclose{\rrbracket})\big)^{q/2} \1_{\mathcal{E}_n}}
\\
&\le C \left(l_n^{q/2} + (\sigma_n + \varrho_n)^{q/2} |t-s|^{1+\beta}\right).
\end{align*}
Recall that we assume $d_{n}(1) = 0$, so $\sigma_{n}^{2} + \varrho_{n} \ge \edges_{n}+ \varrho_{n} = \vertices_{n}$ so $\vertices_n^{-1/2} (\sigma_n + \varrho_n)$ is bounded away from $0$. Finally, recall also that we have assumed $\vertices_n s$ and $\vertices_n t$ to be integers so $|t-s| \ge \vertices_n^{-1}$, and thus, for $\beta < q/4-1$, we have $(\sigma_n + \varrho_n)^{q/2} |t - s|^{1+\beta} \ge (\sigma_n + \varrho_n)^{q/2} \vertices_n^{-(1+\beta)}$ which is bounded below by some constant times a positive power of $\vertices_n$, so, for $n$ large enough (but independently of $s$ and $t$), it is larger than $l_n$ which is of order $\ln \vertices_n$. It follows that for another constant $C > 0$,
\[\Es{\bigg(\sum_{z \in \mathopen{\rrbracket} \hat{y}, y \mathclose{\rrbracket}} \chi_{z}\bigg)^{q/2} \1_{\mathcal{E}_n}}
\le C \cdot (\sigma_n + \varrho_n)^{q/2} \cdot |t - s|^{1+\beta},\]
which completes the proof of~\eqref{eq:moments_label_Kolmogorov}.
\end{proof}

\subsection{Brownian labelled forests}
\label{sec:serpent_brownien}

Recall from Section~\ref{sec:enonce_tension_convergence_etiquettes}, in particular~\eqref{eq:decomposition_labels}, that the label process $\Lfn$ can be written as
\[\Lfn(k) = \tLfn(k) + \bfn(1-\infWfn(k)),
\qquad (0\le k \le V_n)\]
where, conditionally given $\fn$, the two processes $\tLfn(\cdot)$ and $\bfn(1-\infWfn(\cdot))$ are independent. On the one hand $\tLfn$ is the concatenation of the label process of each tree taken individually, so where all labels have been shifted so that each root has label $0$. On the other hand  $\bfn(1-\infWfn(k))$ gives the value of the label of the root vertex of the tree containing the $k$'th vertex, with $\infWfn(k) = \min_{0 \le i \le k} \Wfn(i)$ and $(\bfn(i))_{1 \le i \le \varrho_n}$ which is independent of $\fn$ and uniformly distributed in $\mathscr{B}_{\varrho_n}^{\ge -1}$.

We claim that if $\lim_{n \to \infty} \sigma_n^{-1} \varrho_n = \infty$, then the convergence in distribution
\[\left((2\varrho_n)^{-1/2} \left(\bfn(\varrho_n t), \tLfn(\vertices_n t)\right); t \in [0,1]\right) \cvloi ((\bb_t, 0); t \in [0,1]),\]
holds in $\mathscr{C}([0,1],\R^2)$, where $\bb$ is an independent standard Brownian bridge from $0$ to $0$ with duration $1$. This yields Theorem~\ref{thm:convergence_etiquettes}~\eqref{thm:convergence_etiquettes_arbre} as noted just after the statement of this theorem.

\begin{proof}[Proof of Theorem~\ref{thm:convergence_etiquettes}~\eqref{thm:convergence_etiquettes_arbre}]
First, consider the random walk $S$ with step distribution $\P(S_1 = i) = 2^{-(i+2)}$ for all $i \ge -1$, which is centred and with variance $2$. One can check that for any $k \ge 1$, the law of the bridge obtained by conditioning $S$ to satisfy $S_k = 0$ has the uniform distribution in $\mathscr{B}_k^{\ge -1}$. Therefore the process $\bfn$ has the law of such a random walk bridge with length $\varrho_n$ and a conditional version of Donsker's invariance principle, see e.g.~\cite[Lemma~10]{Bet10} for a detailed proof, shows that $(2\varrho_n)^{-1/2} \bfn(\varrho_n \cdot)$ converges in distribution towards $\bb$. 
Recall Proposition~\ref{prop:convergence_Luka}, we conclude that the convergence in distribution 
\[\left((2\varrho_n)^{-1/2} \bfn(1-\infWfn(\lfloor \vertices_n t \rfloor)); t \in [0,1]\right) \cvloi (\bb_t; t \in [0,1])\]
holds in $\mathscr{C}([0,1],\R)$.

We next claim that $\varrho_n^{-1/2} \tLfn$ converges in probability to the null process, which yields our result by the decomposition~\eqref{eq:decomposition_labels}. By this decomposition and the preceding convergence, tightness follows from Theorem~\ref{thm:tension_etiquettes} so it remains to prove the convergence of the (one-dimensional) marginals. 
Fix $k \ge 1$, and  $X_k = (X_{k,1}, \dots, X_{k,k})$ a uniform random bridge in $\mathscr{B}^{\ge -1}_k$, and set $X_{k,0}=0$; using the representation as a bridge of the random walk $S$, one easily infers that for every $q \ge 2$, there exists $C(q) > 0$ such that and every $i,j \in \{0, \dots, k\}$,
\begin{equation}\label{eq:borne_moments_pont_uniforme}
\Es{\left|X_{k,i} - X_{k,j}\right|^q} \le C(q) \cdot |i-j|^{q/2},
\end{equation}
see e.g. Le~Gall \& Miermont~\cite[Lemma 1]{LGM11}.
Fix $t \in [0,1]$; building on this bound, as in the preceding proof, we obtain
\[\Esc{\left|\tLfn(\lfloor\vertices_n t\rfloor)\right|^q}{\fn} \le C(q) \cdot \left|\Wfn(\lfloor\vertices_n t\rfloor) - \infWfn\lfloor\vertices_n t\rfloor\right|^{q/2}.\]
See e.g.~\cite{Mar18b} Equation~22 and the next few lines there, with $u=0$ and $v=\lfloor\vertices_n t\rfloor$; only the case of a single tree is considered there but the argument extends readily. By Proposition~\ref{prop:convergence_Luka}, the right-hand side divided by $\varrho_n^{q/2}$ converges in probability to $0$ which yields the convergence in probability to $0$ of $\varrho_n^{-1/2} \tLfn(\lfloor\vertices_n t\rfloor)$ and the proof is complete.
\end{proof}

Let us next turn to the second part of Theorem~\ref{thm:convergence_etiquettes}. 
Recall from Section~\ref{sec:enonce_tension_convergence_etiquettes} the process $Z^\varrho$ which describes Brownian labels on a forest of Brownian trees coded by the first-passage bridge $X^\varrho$; it takes the form
\[Z^\varrho_t = \tildep{Z}^\varrho_t + \sqrt{3} \cdot \bb^\varrho_{- \underlinep{X}^\varrho_t}
\qquad\text{for every}\qquad 0 \le t \le 1,\]
where $\tildep{Z}^\varrho$ is defined conditionally on $\tX = X^\varrho - \underlinep{X}^\varrho$, where $\underlinep{X}$ is the running infimum of $X^\varrho$, as a centred Gaussian process with covariance 
$\E[\tildep{Z}^\varrho_s \tildep{Z}^\varrho_t \mid \tX] = \min_{r \in [s,t]} \tX_r$ for every $0 \le s \le t \le 1$, and where $\bb^\varrho$ is an independent standard Brownian bridge from $0$ to $0$ with duration $\varrho$.

Then Theorem~\ref{thm:convergence_etiquettes}~\eqref{thm:convergence_etiquettes_carte_disque} claims that if $\lim_{n \to \infty} \sigma_n^{-1} \varrho_n = \varrho$ with $\varrho \in [0,\infty)$ and $\lim_{n \to \infty} \sigma_n^{-1} \Delta_n = 0$ then the convergence in distribution
\[\left(\left(\frac{3}{2\sigma_n}\right)^{1/2} \Lfn(\vertices_n t) ; t \in [0,1]\right) \cvloi (Z^\varrho_t; t \in [0,1])\]
holds in $\mathscr{C}([0,1],\R)$.
Since we already proved tightness of this sequence, it only remains to characterise the subsequential limits.

\begin{prop}\label{prop:marginales_labels}
Suppose that $\lim_{n \to \infty} \sigma_n^{-1} \varrho_n = \varrho$ for some $\varrho \in [0,\infty)$ and $\lim_{n \to \infty} \sigma_n^{-1} \Delta_n = 0$. For every $n,q \ge 1$, sample $(\tn, \ell)$ uniformly at random in $\LFn$ and, independently, sample $U_1, \dots, U_q$ uniformly at random in $[0,1]$. Sample $\tildep{Z}^\varrho$ independently of $(U_1, \dots, U_q)$, then we have
\begin{equation}\label{eq:cv_label_unif_multi}
\left(\frac{3}{2\sigma_n}\right)^{1/2} \left(\tLfn(\vertices_n U_1), \dots, \tLfn(\vertices_n U_q)\right)
\cvloi
\left(\tildep{Z}^\varrho_{U_1}, \dots, \tildep{Z}^\varrho_{U_q}\right)
\end{equation}
jointly with the convergence of $\Wfn$ and $\Hfn$ in Proposition~\ref{prop:convergence_Luka} and Theorem~\ref{thm:marginales_hauteur}.
\end{prop}

Let us differ the proof of the proposition to the next subsection and finish now the proof of Theorem~\ref{thm:convergence_etiquettes}.

\begin{proof}[Proof of Theorem~\ref{thm:convergence_etiquettes}~\eqref{thm:convergence_etiquettes_carte_disque}]
Suppose that $\lim_{n \to \infty} \sigma_n^{-1} \varrho_n = \varrho \in [0, \infty)$ and $\lim_{n \to \infty} \sigma_n^{-1} \Delta_n = 0$. 
Then the sequence $(\sigma_n^{-1} \Wfn(\lfloor \vertices_n t \rfloor))_{t \in [0,1]}$ converges in distribution towards $X^\varrho$ by Proposition~\ref{prop:convergence_Luka} and, as in the previous proof, the sequence $((2\varrho_n)^{-1/2} \bfn(\varrho_n t))_{n \ge 1}$ converges in distribution towards $\bb$ whenever $\varrho_n \to \infty$. Since the two are independent, we deduce that, when $\varrho > 0$,
\[\left(\left(\frac{3}{2\sigma_n}\right)^{1/2} \bfn(1 - \infWfn(\vertices_n t))\right)_{t \in [0,1]} 
\cvloi \left((3 \varrho)^{1/2} \bb_{- \varrho^{-1} \underlinep{X}^\varrho_t}\right)_{t \in [0,1]}\]
in $\mathscr{C}([0,1],\R)$ and the limit has the same law as $(\sqrt{3} \bb^\varrho_{- \underlinep{X}^\varrho_t})_{t \in [0,1]}$; when $\varrho = 0$, the limit is instead the null process.

Next, still as in the previous proof, by Theorem~\ref{thm:tension_etiquettes}, the sequence $(\sigma_n^{-1/2} \tLfn)_{n \ge 1}$ is tight; then this equicontinuity combined with the uniform continuity of the process $\tildep{Z}^\varrho$ allows to approximate the marginals evaluated at deterministic times by sampling sufficiently many i.i.d. uniform random times. Hence Proposition~\ref{prop:marginales_labels} yields actually a convergence in $\mathscr{C}([0,1], \R)$.
Combined with Proposition~\ref{prop:convergence_Luka} we obtain the joint convergence
\[\left(\frac{1}{\sigma_n} \Wfn(\vertices_n t), \left(\frac{3}{2\sigma_n}\right)^{1/2} \tLfn(\vertices_n t)\right)_{t \in [0,1]}
\cvloi (X^\varrho_t, \tildep{Z}^\varrho_t)_{t \in [0,1]}\]
in $\mathscr{C}([0,1],\R^2)$.
Recall the decomposition~\eqref{eq:decomposition_labels} of $\Lfn$; we deduce from the conditional independence of $\bfn$ and $\tLfn$, that
\[\left(\left(\frac{3}{2\sigma_n}\right)^{1/2} \Lfn(\vertices_n t)\right)_{t \in [0,1]}
\cvloi (\tildep{Z}^\varrho_t + 3^{1/2} \bb^\varrho_{- \underlinep{X}^\varrho_t})_{t \in [0,1]},\]
and the right-hand side is the definition of the process $Z^\varrho$.
\end{proof}

\subsection{Marginals of the label process}
\label{sec:marginales_labels}

Let us now prove Proposition~\ref{prop:marginales_labels} which is the last remaining ingredient to Theorem~\ref{thm:convergence_etiquettes} and therefore to Theorem~\ref{thm:convergence_carte_disque}. We start with the case $q=1$ and comment then on the general case.

We shall need the following technical estimate. Let $x_n$ denote a vertex in $\fn$, with height $|x_n|$ and for every $1 \le p \le |x_n|$, let $k_p(x_n)$ denote the number of offspring of the ancestor of $x_n$ at height $p-1$ and let $j_p(x_n)$ the index of the offspring of this ancestor which is also an ancestor of $x_n$. Conditionally on $\fn$ and $x_n$, let $(B_p(x_n) ; 1 \le p \le x_n)$ be independent random variables, where each $B_p(x_n)$ has the law 
of a uniformly random bridge in $\mathscr{B}_{k_p(x_n)}^{\ge-1}$ evaluated at time $j_p(x_n)$. Note that they are centred, let us denote by $\sigma^2_p(x_n)$ the variance of $B_p(x_n)$, which is measurable with respect to $\fn$ and $x_n$.

\begin{lem}\label{lem:convergence_variance_labels}
Assume that $\lim_{n \to \infty} \sigma_n^{-1} \varrho_n = \varrho \in [0, \infty)$ and $\lim_{n \to \infty} \sigma_n^{-1} \Delta_n = 0$. For every $n \ge 1$, sample $(\fn, \ell)$ uniformly at random in $\LFn$ and independently sample a vertex $x_n$ uniformly at random. Then
\[\frac{3 \edges_n}{\sigma_n^2 |x_n|} \sum_{p=1}^{|x_n|} \sigma^2_p(x_n) \cvproba 1.\]
\end{lem}

\begin{proof}
Let us first prove the claim on the event $\mathcal{E}_n(K)$ defined as follows with $K > 0$ large but fixed:
\begin{equation}\label{eq:bon_evenement_marginales_labels}
\mathcal{E}_n(K) = \left\{\frac{\sigma_n}{\edges_n} |x_n| \in [K^{-1}, K] \text{ and } \LR(x_n) \le K \sigma_n\right\}.
\end{equation}
Indeed by combining Proposition~\ref{prop:queues_exp_Luka}, Remark~\ref{rem:queues_exp_hauteur_typique_foret}, as well as Theorem~\ref{thm:marginales_hauteur} for the lower bound on $|x_n|$, we see that $\lim_{K \to \infty} \liminf_{n \to \infty} \Pr{\mathcal{E}_n(K)} = 1$.
Let $\sigma^2(k,j)$ denotes the variance of a uniformly random bridge in $\mathscr{B}_k^{\ge-1}$ evaluated at time $j$, which is known explicitly, see e.g. Marckert \& Miermont~\cite[page 1664\footnote{Note that they consider uniform random bridges in $\mathscr{B}_{k+1}^{\ge-1}$!}]{MM07}: we have
\[\sigma^2(k,j) = \frac{2j(k-j)}{k+1},
\qquad\text{so}\qquad
\sum_{j=1}^k \sigma^2(k,j) = \frac{k(k-1)}{3}.\]
Fix $\delta > 0$, then by~\eqref{eq:epine_un_sommet},
\begin{multline*}
\Pr{\left\{\left|\frac{3 \edges_n}{\sigma_n^2 |x_n|} \sum_{p = 1}^{|x_n|} \sigma^2(x_n) - 1\right| > \delta\right\} \cap \1_{\mathcal{E}_n(K)}}
\\
\le K (\varrho + K + o(1)) \sup_{K^{-1} \edges_n / \sigma_n \le h \le K \edges_n / \sigma_n} \Pr{\left|\frac{3 \edges_n}{\sigma_n^2 h} \sum_{p = 1}^h \sigma^2_p(\xi_{d_n}(p), \chi_{d_n}(p)) - 1\right| > \delta}
\end{multline*}
We calculate the first two moments: to ease notation, let $(\xi, \xi', \chi, \chi')$ have the law of the quadruple $(\xi_{d_n}(1), \xi_{d_n}(2), \chi_{d_n}(1), \chi_{d_n}(2))$. Since $\chi$ has the uniform distribution on $\{1, \dots, \xi\}$, then
\[\Es{\sum_{p = 1}^h \sigma^2(\xi_{d_n}(p), \chi_{d_n}(p))}
= h \cdot \Es{\xi^{-1} \sum_{j=1}^\xi \sigma^2(\xi,j)}
= h \cdot \Es{\frac{\xi-1}{3}}
= \frac{\sigma_n^2 h}{3 \edges_n}.\]
It only remains to prove that the variance is small; we have similarly
\begin{multline*}
\Es{\left(\frac{3 \edges_n}{\sigma_n^2 h} \sum_{p = 1}^h \sigma^2(\xi_{d_n}(p), \chi_{d_n}(p))\right)^2}
\\= \left(\frac{3 \edges_n}{\sigma_n^2 h}\right)^2 \left(h \Es{\sigma^2(\xi, \chi)^2} + h(h-1) \Es{\sigma^2(\xi, \chi) \sigma^2(\xi', \chi')}\right).
\end{multline*}
Since the $\chi$'s are independent conditionally on the $\xi$'s, we have
\[\Es{\sigma^2(\xi, \chi) \sigma^2(\xi', \chi')}
= \Es{\frac{(\xi-1)(\xi'-1)}{9}}
= \left(\frac{\sigma_n^2}{3 \edges_n}\right)^2 (1+o(1)),\]
where we used that two samples without replacement decorrelate as the number of possible picks tends to infinity. We conclude that 
\[\Es{\frac{3 \edges_n}{\sigma_n^2 h} \sum_{p = 1}^h \sigma^2(\xi_{d_n}(p), \chi_{d_n}(p))} = 1,\]
and, since $\sigma^2(\xi, \chi) \le \xi \le \Delta_n$, then this random variable has variance
\begin{align*}
&\left(\frac{3 \edges_n}{\sigma_n^2 h}\right)^2 \left(h \Es{\sigma^2(\xi, \chi)^2} + h(h-1) \Es{\sigma^2(\xi, \chi) \sigma^2(\xi', \chi')}\right) - 1
\\
&\qquad\le \left(\frac{3 \edges_n}{\sigma_n^2 h}\right)^2 \left(h \Delta_n \frac{\sigma_n^2}{3 \edges_n} + h(h-1) \left(\frac{\sigma_n^2}{3 \edges_n}\right)^2 (1+o(1))\right) - 1
\\
&\qquad\le \frac{3 \edges_n}{\sigma_n^2 h} \Delta_n + o(1),
\end{align*}
which converges to $0$ uniformly in $h \in [K^{-1} \edges_n / \sigma_n, K \edges_n / \sigma_n]$ since we assume that $\Delta_n = o(\sigma_n)$. This proves that $\frac{3 \edges_n}{\sigma_n^2 |x_n|} \sum_{p = 1}^{|x_n|} \sigma^2_p(x_n)$ converges in probability to $1$ on the event $\mathcal{E}_n(K)$ and therefore also unconditionally by letting further $K$ tend to infinity.
\end{proof}

We may now prove Proposition~\ref{prop:marginales_labels}.

\begin{proof}[Proof of Proposition~\ref{prop:marginales_labels}]
Let us start with the case $q=1$.
Let $U$ have the uniform distribution on $[0,1]$ and let $x_n$ be the uniformly random vertex visited at the time $\lfloor \vertices_n U \rfloor$ in lexicographical order, with label $\ell(x_n) = \Lfn(\lfloor \vertices_n U \rfloor)$ and height $|x_n| = \Hfn(\lfloor \vertices_n U \rfloor)$. 
Let us write
\[\left(\frac{3}{2\sigma_n}\right)^{1/2} \ell(x_n)
= \left(\frac{\sigma_n}{2 \edges_n} |x_n|\right)^{1/2} \left(\frac{3 \edges_n}{\sigma_n^2 |x_n|}\right)^{1/2} \ell(x_n).\]
Recall from Theorem~\ref{thm:marginales_hauteur} that $\frac{\sigma_n}{2 \edges_n} |x_n|$ converges in distribution towards $\tX_U$, it is therefore equivalent to show that, jointly with this convergence, we have
\begin{equation}\label{eq:cv_label_unif}
\left(\frac{3 \edges_n}{\sigma_n^2 |x_n|}\right)^{1/2} \ell(x_n) \quad\mathop{\Longrightarrow}_{n \to \infty}\quad \mathscr{N}(0, 1),
\end{equation}
where $\mathscr{N}(0, 1)$ denotes the standard Gaussian distribution and `$\Rightarrow$' is a slight abuse of notation to refer to the weak convergence of the law of the random variable. 
The idea is to decompose $\ell(x_n)$ as the sum of the label increments $(B_p(x_n) ; 1 \le p \le |x_n|)$ between two consecutive ancestors, as defined before the preceding lemma. To ease notation, let us set $\Sigma_n^2 = \sum_{p=1}^{|x_n|} \sigma^2_p(x_n)$, then the Central Limit Theorem for triangular arrays shows that the conditional law of $\Sigma_n^{-1} \ell(x_n)$ converges to the standard Gaussian distribution as soon as for every $\varepsilon > 0$,
\begin{equation}\label{eq:Lyapunov_TCL}
\Sigma_n^{-2} \sum_{p=1}^{|x_n|} \Esc{B_p(x_n)^2 \ind{|B_p(x_n)| > \varepsilon \Sigma_n}}{\fn, x_n} \cv 0.
\end{equation}
Our claim~\eqref{eq:cv_label_unif} then follows from~\eqref{eq:Lyapunov_TCL} and Lemma~\ref{lem:convergence_variance_labels}. As in the proof of this lemma we shall implicitly work on the intersection of the event $\mathcal{E}_n(K)$ defined in~\eqref{eq:bon_evenement_marginales_labels} and $\{|\Sigma_n^2 \frac{3 \edges_n}{\sigma_n^2 |x_n|} - 1|\le K^{-1}\}$ with $K>0$ fixed, whose probability tends to $1$ when first $n\to\infty$ and then $K\to\infty$. Note that this allows to replace equivalently $\Sigma^2_n$ by $\sigma_n$ in~\eqref{eq:Lyapunov_TCL}. As proved in~\cite[Lemma~1]{LGM11}, for every $r \ge 1$, there exists a constant $K_{(r)} > 0$ such that for every $k \ge j \ge 1$, the $(2r)$'th moment of a uniformly random bridge in $\mathscr{B}_k^{\ge-1}$ evaluated at time $j$ is bounded above by $K_{(r)} |k-j|^r$.
Fix $\varepsilon > 0$, then, by first using the Cauchy--Schwarz and the Markov inequalities,
\begin{align*}
\Esc{B_p(x_n)^2 \ind{|B_p(x_n)| > \varepsilon \Sigma_n}}{\fn, x_n}^2
&\le \Esc{B_p(x_n)^4}{\fn, x_n} \frac{\E[|B_p(x_n)|^2 \mid \fn, x_n]}{\varepsilon^2 \Sigma_n^2}
\\
&\le K_{(2)} |k_p(x_n)-j_p(x_n)|^2 \frac{K_{(1)} |k_p(x_n)-j_p(x_n)|}{\varepsilon^2 \Sigma_n^2}
\\
&\le \frac{K'}{\varepsilon^2} \frac{|k_p(x_n)-j_p(x_n)|^3}{\sigma_n}
,\end{align*}
where $K' > 0$ depends on $K$.
Therefore the left hand side in~\eqref{eq:Lyapunov_TCL} is upper bounded by some constant that depends on $K$ and $\varepsilon$ times
\[\frac{1}{\sigma_n} \sum_{p=1}^{|x_n|} \sqrt{\frac{|k_p(x_n)-j_p(x_n)|^3}{\sigma_n}}
\le \sqrt{\frac{\Delta_n}{\sigma_n}} \frac{1}{\sigma_n} \sum_{p=1}^{|x_n|} |k_p(x_n)-j_p(x_n)|,\]
where the inequality is obtained by upper bounding $|k_p(x_n)-j_p(x_n)|$ by $\Delta_n$. Observe that the last sum precisely equals the value of the {\L}ukasiewicz path at time $\lfloor \vertices_n U \rfloor$, which is of order $\sigma_n$ with high probability. Since $\Delta_n/\sigma_n \to 0$, then we conclude that indeed~\eqref{eq:Lyapunov_TCL} holds on an event which is measurable with respect to $\fn$ and $x_n$ and whose probability is arbitrarily close to $1$. Combined with Lemma~\ref{lem:convergence_variance_labels} this concludes the proof of~\eqref{eq:cv_label_unif} and thus of the proposition for $q=1$.

We next sketch the proof of the claim for $q=2$, the general case $q \ge 3$ is exactly the same and only requires more notation. We shall need the following result, which claims that in the case of no macroscopic degree, there is no macroscopic label increment in the sense that the maximal difference along an edge is small:
\begin{equation}\label{eq:deplacement_max_autour_site}
\sigma_n^{-1/2} \max_{x \in \fn} \left|\max_{1 \le j \le k_x} \ell(xj) - \min_{1 \le j \le k_x} \ell(xj)\right| \cvproba 0.
\end{equation}
This extends Proposition~2 in~\cite{Mar18b} to which we refer for a proof, just replace $N_\mathbf{n}$ there by $\sigma_n^2$.
Let $x_n$ and $y_n$ be independent uniform random vertices of $\fn$ and let us consider the three branches consisting of: their common ancestors on one side, the ancestors of $x_n$ only on another side, and finally those of $y_n$ only. According to Theorem~\ref{thm:marginales_hauteur}, their lengths jointly converge in distribution, when rescaled by a factor $\sigma_n/(2 \edges_n)$.
By~\eqref{eq:deplacement_max_autour_site}, the label increments between the last common ancestor of $x_n$ and $y_n$ and its offspring which are ancestors of $x_n$ and $y_n$ respectively are small compared to our scaling, and the total label increment on each of the three remaining branches are independent; it only remains to prove that the law of these increments, multiplied by $(3 \edges_n/\sigma_n^2)^{1/2}$ and divided by the square-root of their length converges to the standard Gaussian law, as in~\eqref{eq:cv_label_unif}. This can be obtained in the very same way as in the case $q=1$, appealing to Lemma~\ref{lem:multi_epines_sans_remise} to compare the content of each branch with the sequence $(\xi_{d_n}(p), \chi_{d_n}(p))_p$. As in the proof of Theorem~\ref{thm:marginales_hauteur}, the fact that we now have three branches is compensated by the factor $(\sigma_n / \vertices_n)^3$ in Lemma~\ref{lem:multi_epines_sans_remise}, we leave the details to the reader.
\end{proof}

\section{Stable Boltzmann maps}
\label{sec:BGW_Boltzmann}

In this last section, we study size-conditioned stable Boltzmann random maps by relying on our main results. 
In Sections~\ref{sec:marche_stable} and~\ref{sec:Luka_stable}, we present the precise setup and the assumptions we shall make on such laws, by relying on the {\L}ukasiewicz path of the associated labelled forest. 
Then in Section~\ref{sec:loi_sauts_Luka_stable} we state and prove the main result of this section, Theorem~\ref{thm:GW}, which shows that the random degree distribution satisfies the main assumptions from the introduction, according to the value of the stability index. We then easily derive the behaviour of the maps in Sections~\ref{sec:Boltzmann} and~\ref{sec:Boltzmann_sous_crit}, see Theorems~\ref{thm:Boltzmann} and~\eqref{thm:Boltzmann_sous_crit}.
 
Throughout this section, we shall divide by real numbers which depend on an integer $n$, and consider conditional probabilities with respect to events which depend on $n$; we shall therefore always implicitly restrict ourselves to those values of $n$ for which such quantities are well defined and statements such as `as $n \to \infty$' should be understood along the appropriate subsequence of integers.

\subsection{On stable domains of attraction}
\label{sec:marche_stable}

Throughout this section, we work with a sequence $(X_i)_{i \ge 1}$ of i.i.d. copies of a random variable $X$ with distribution, say $\nu$, supported by (a subset of) $\Z_{\ge -1}$ with $\nu(-1) \ne 0$, with finite first moment and $\E[X_1] = 0$. For every integer $n \ge 1$, we let $S_n = X_1 + \dots + X_n$, and set $S_0 = 0$. 
For every integers $n \ge 1$ and $k \ge -1$, let $K_k(n) = \#\{1 \le i \le n : X_i = k\}$ be the number of jumps of size $k$ up to time $n$; for a subset $A \subset \Z_{\ge -1}$, set then $K_A(n) = \sum_{k \in A} K_k(n)$. Finally, for $\varrho \ge 1$, let $\zeta(S, \varrho) = \inf\{i \ge 1 : S_i = - \varrho\}$, and simply write $\zeta(S)$ for $\zeta(S,1)$.

We recall that a measurable function $l : [0, \infty) \to [0, \infty)$ is said to be \emph{slowly varying} (at infinity) if for every $c > 0$, we have $\lim_{x \to \infty} l(cx)/l(x) = 1$. 
For $\alpha \in [1,2]$, we say that $\nu$ satisfies $\mathrm{(H_\alpha)}$ if the tail distribution can be written as
\[\Pr{X \ge n} = n^{-\alpha} L_1(n),\]
where $L_1$ is slowly varying. We also include the case where $X$ has finite variance in $\mathrm{(H_2)}$. Finally, when $\alpha = 1$, we say that $\nu$ satisfies $\mathrm{(H_1^{loc})}$ when the mass function can be written as
\[\Pr{X = n} = n^{-2} L_1(n),\]
where $L_1$ is slowly varying.
The assumption $\mathrm{(H_\alpha)}$ corresponds to the \emph{domain of attraction} of a stable law with index $\alpha$, and $\mathrm{(H_1^{loc})}$ is more restrictive than $\mathrm{(H_1)}$. 

When $\alpha \in (1,2]$, it is well known that there exists a sequence $(a_n)_{n \ge 1}$ such that $(n^{-1/\alpha} a_n)_{n \ge 1}$ is slowly varying
and $a_n^{-1} (X_1 + \dots + X_n)$ converges in distribution to $\mathscr{X}^{(\alpha)}$ with Laplace transform given by $\E[\exp(- t \mathscr{X}^{(\alpha)})] = \exp(t^\alpha)$ for $t > 0$. Note that $\mathscr{X}^{(2)}$ has the Gaussian distribution with variance $2$; as a matter of fact, if $X$ has finite variance $\sigma^2$, then we may take $a_n = (n \sigma^2 / 2)^{1/2}$. Moreover, there exists another slowly varying function $L$ such that for every $n \ge 1$, we have $\Var(X \ind{X \le n}) = n^{2-\alpha} L(n)$. This function is related to $L_1$ by
\begin{equation}\label{eq:ratio_fonctions_var_lente}
\lim_{n \to \infty} \frac{L_1(n)}{L(n)} 
= \lim_{n \to \infty} \frac{n^2 \P(X \ge n)}{\Var(X \ind{X \le n})} 
= \frac{2-\alpha}{\alpha},
\end{equation}
see Feller~\cite[Chapter~XVII, Equation~5.16]{Fel71}. 
Finally, according to~\cite[Equation~7]{Kor17}, we have
\begin{equation}\label{eq:constante_Levy}
\lim_{n \to \infty} \frac{n L(a_n)}{a_n^\alpha} = \frac{1}{(2-\alpha) \Gamma(-\alpha)},
\end{equation}
where, by continuity, the limit is interpreted as equal to $2$ if $\alpha=2$.

In the case $\alpha=1$, when $\mathrm{(H_1)}$ is the domain of attraction of a Cauchy distribution, in addition to the sequence $(a_n)_{n \ge 1}$, there exists another sequence $(b_n)_{n \ge 1}$ such that both $(n^{-1} a_n)_{n \ge 1}$ and $(n^{-1} b_n)_{n \ge 1}$ are slowly varying, with $b_n \to -\infty$ and $b_n / a_n \to -\infty$, and now $a_n^{-1} (X_1 + \dots + X_n - b_n)$ converges in distribution to $\mathscr{X}^{(1)}$ with Laplace transform given by $\E[\exp(- t \mathscr{X}^{(1)})] = \exp(t \ln t)$ for $t > 0$. An example to have in mind when dealing with this rather unusual regime is given by Kortchemski \& Richier~\cite{KR19}: take $\nu(n) \sim \frac{c}{n^2 \ln(n)^2}$, then $a_n \sim \frac{c n}{\ln(n)^2}$ and $b_n \sim - \frac{c n}{\ln(n)}$.

\subsection{Local limit theorems and conditioned paths}
\label{sec:Luka_stable}

Let us now describe our model of random paths, these assumptions and notations shall be used throughout this section. 
We shall assume that $\nu$ satisfies either $\mathrm{(H_\alpha)}$ for some $\alpha \in [1,2]$, or $\mathrm{(H_1^{loc})}$, and then the sequences $(a_n)_{n \ge 1}$ and $(b_n)_{n \ge 1}$ will always be as in the preceding subsection. 
Fix $n \ge 1$, $\varrho_n \ge 1$, and a set $A \subset \Z_{\ge -1}$ with $\nu(A) > 0$, then:
\begin{itemize}
\item Let us denote by $W^{\varrho_n}$ the path $S$ stopped when first hitting $-\varrho_n$; if $\varrho_n = 1$, simply write $W$.

\item If $\nu$ satisfies $\mathrm{(H_\alpha)}$ for some $\alpha \in (1,2]$, then we let $W^{\varrho_n}_{n,A}$ be the path $W^{\varrho_n}$ conditioned to have made $n$ jumps with value in the set $A$.

\item If $\nu$ satisfies $\mathrm{(H_1^{loc})}$, then we only consider $A = \Z_{\ge -1}$ and we define $W^{\varrho_n}_n=W^{\varrho_n}_{n,\Z_{\ge -1}}$ similarly so it reduces to conditioning the walk $S$ to first hit $-\varrho_n$ at time $n$ and stopping it there.

\item If $\nu$ satisfies $\mathrm{(H_1)}$, then again $A = \Z_{\ge -1}$ but now $\varrho_n = 1$, and we let $W_{\ge n}$ be the walk $S$ conditioned to stay nonnegative at least up to time $n$ and stopped when first hitting $-1$ (at a random time).
\end{itemize}
We recalled how to construct the path $W^{\varrho_n}_{n}$ by cyclically shifting a bridge, i.e. the walk $S$ up to time $n$, conditioned to satisfy $S_n = -\varrho_n$. This construction can be generalised to any subset $A$, by cyclically shifting the walk $S$ up to its $n$'th jump with value in $A$, and conditioned to be at $-\varrho_n$ at this moment, see Kortchemski~\cite[Lemma~6.4]{Kor12} for a detailed proof in the case $\varrho_n = 1$ and $A = \{-1\}$; it extends here: replacing `$=-1$' by `$\in A$' does not change anything, and when $\varrho_n \ge 2$, there are $\varrho_n$ cyclically shifted bridges which are first-passage bridges, but this factor $\varrho_n$ cancels since the cycle lemma is used twice.

The study of such bridges relies strongly on local limit theorems, which are used to compare them to unconditioned random walks.
By~\cite[Equation~50]{Kor12} the following holds: suppose that $\nu$ satisfies $\mathrm{(H_\alpha)}$ with $\alpha \in (1, 2]$, fix $A \subset \Z_{\ge -1}$ such that $\nu(A) > 0$, and, if $\nu$ has infinite variance, suppose that either $A$ or $\Z_{\ge -1} \setminus A$ is finite. Then for any sequence $(\varrho_n)_{n\ge 1}$ such that $\limsup_{n \to \infty} a_n^{-1} \varrho_n < \infty$, we have
\begin{equation}\label{eq:LLT_GW}
\lim_{n \to \infty} \left|n \cdot \Pr{K_A(\zeta(S, \varrho_n)) = n} - \frac{\nu(A)^{1/\alpha} \varrho_n}{a_n} \cdot p_1\left(- \frac{\nu(A)^{1/\alpha} \varrho_n}{a_n}\right)\right| = 0,
\end{equation}
where $p_1$ is the density of $\mathscr{X}^{(\alpha)}$.
We shall mostly be interested in the cases $A = \Z_{\ge -1}$, $A = \{-1\}$, and $A = \Z_{\ge 0}$. Nevertheless, Th\'evenin~\cite{The20} proposed a different method, relying on a multivariate local limit theorem, in order to lift the restriction that either $A$ or its complement should be finite; Proposition~6.5 there corresponds to~\eqref{eq:LLT_GW} in the case $\rho_n=1$ on which the latter work focuses, but the general case can be adapted similarly.

In the Cauchy regime $\alpha = 1$, such a local limit theorem is not known for general sets $A$, 
which is the reason why we restrict ourselves to the case $A = \Z_{\ge -1}$. 
Then $K_{\Z_{\ge -1}}(k) = k$ so what we consider in~\eqref{eq:LLT_GW} is simply the first hitting time of $-\varrho_n$. After cyclic shift, the probability that the latter equals $n$ is given by $n^{-1} \varrho_n \P(S_n = - \varrho_n)$. When $\nu$ satisfies $\mathrm{(H_1^{loc})}$, we read from the recent work of Berger~\cite[Theorem~2.4]{Ber19} that, when $\varrho_n = O(a_n) = o(|b_n|)$, we have
\begin{equation}\label{eq:LLT_GW_Cauchy_loc}
\Pr{\zeta(S, \varrho_n) = n}
\equivalent 
\varrho_n \cdot \frac{L_1(|b_n| - \varrho_n)}{(|b_n| - \varrho_n)^2},
\end{equation}
see~\cite[Equation~2.10]{Ber19} with $p = \alpha = 1$ and $x = |b_n| - \varrho_n$.
When $\nu$ only satisfies $\mathrm{(H_1)}$, Kortchemski \& Richier~\cite[Proposition~12]{KR19} proved that, in the case $\varrho_n = 1$,
\begin{equation}\label{eq:LLT_GW_Cauchy_gen}
\Pr{\zeta(S) \ge n}
\equivalent 
\Lambda(n) \cdot \frac{L_1(|b_n|)}{|b_n|},
\end{equation}
where $\Lambda$ is some other slowly varying function unimportant here.

\subsection{On the empirical jump distribution}
\label{sec:loi_sauts_Luka_stable}

We study in this subsection the random paths $W^{\varrho_n}_{n,A}$ and $W_{\ge n}$ described above. 
Let $\zeta(W^{\varrho_n})$ denote the number of steps of the path $W^{\varrho_n}$, and for every subset $B \subset \Z_{\ge -1}$, let
\[J_B(W^{\varrho_n}) = \#\{1 \le i \le \zeta(W^{\varrho_n}) : W^{\varrho_n}(i) - W^{\varrho_n}(i-1) \in B\},\]
and define similar quantities for $W^{\varrho_n}_{n,A}$ and $W_{\ge n}$.
Our first result is a strong law of large numbers.

\begin{lem}\label{lem:scaling_GW}
Let $B$ any subset of $\Z_{\ge -1}$. Assume that either $\nu$ satisfies $\mathrm{(H_\alpha)}$ for some $\alpha \in (1, 2]$ and then take $A \subset \Z_{\ge -1}$ arbitrary (with $\nu(A) > 0$), or that $\nu$ satisfies $\mathrm{(H_1^{loc})}$ and then take $A = \Z_{\ge -1}$. Finally assume that $\limsup_{n \to \infty} a_n^{-1} \varrho_n < \infty$. Then we have
\[n^{-1} J_B(W^{\varrho_n}_{n,A}) \cvps \frac{\nu(B)}{\nu(A)}.\]
If $\nu$ satisfies $\mathrm{(H_1)}$, then
\[\zeta(W_{\ge n})^{-1} J_B(W_{\ge n}) \cvps \nu(B).\]
\end{lem}

In particular, when $\alpha > 1$, recalling that $n^{-1/\alpha} a_n$ is slowly varying, we have for $B = \Z_{\ge -1}$:
\[n^{-1} \zeta(W^{\varrho_n}_{n,A}) \cvps \nu(A)^{-1}
\qquad\text{and so}\qquad
a_n^{-1} a_{\zeta(W^{\varrho_n}_{n,A})} \cvps \nu(A)^{-1/\alpha}.\]
When $\nu$ satisfies $\mathrm{(H_1^{loc})}$, we simply have $\zeta(W^{\varrho_n}_n) = n$; finally, when $\nu$ satisfies $\mathrm{(H_1)}$, we obtain appealing to~\cite[Theorem~30]{KR19} that
\[|b_n|^{-1} |b_{\zeta(W^{\varrho_n}_n)}| \cvloi I,\]
where $I$ has the law $\P(I \ge x) = x^{-1}$ for all $x \ge 1$.
Therefore in any case the natural scaling factor for our conditioned paths, which should involve their total length $\zeta$, may be replaced by a factor depending only on $n$.  The more general statement, for an arbitrary set $B$, shall be used later when dealing with random maps.

\begin{proof}
Let us detail the case $\alpha \in (1,2]$. 
Fix $\delta > 0$ small, we claim that there exist $c, C > 0$ such that for every $n$, we have
\begin{equation}\label{eq:concentration_feuilles}
\Pr{\left|\frac{n}{J_B(W^{\varrho_n}_{n,A})}- \frac{\nu(A)}{\nu(B)}\right| > \frac{\delta}{n^{1/4}}} \le C \e^{- c n^{1/2}}.
\end{equation}
Since the right-hand side is summable, this indeed shows that $n / J_B(W^{\varrho_n}_{n,A})$ converges to $\nu(A) / \nu(B)$ almost surely. Let us write this conditional probability as
\[\frac{1}{\P(J_A(W^{\varrho_n}) = n)} \cdot
\Pr{\left|\frac{J_A(W^{\varrho_n})}{J_B(W^{\varrho_n})}- \frac{\nu(A)}{\nu(B)}\right| > \frac{\delta}{n^{1/4}} \text{ and } J_A(W^{\varrho_n}) = n}.\]
Recall the local limit estimate in~\eqref{eq:LLT_GW}; it is known that $p_1$ is continuous and positive, so in particular bounded away from $0$ and $\infty$ on any compact interval. Hence $\varrho_n^{-1} n a_n \P(J_A(W^{\varrho_n}) = n)$ is bounded away from $0$ and $\infty$ since $\varrho_n = O(a_n)$. Moreover, $\varrho_n \ge 1$ and $a_n = O(n^{3/2})$ and finally $\zeta(W^{\varrho_n}) \ge J_A(W^{\varrho_n})$, so 
the probability that $|\frac{n}{J_B(W^{\varrho_n}_{n,A})}- \frac{\nu(A)}{\nu(B)}| > \frac{\delta}{n^{1/4}}$ is bounded above by some constant times
\[n^{5/2}  \Pr{\left|\frac{J_A(W^{\varrho_n})}{J_B(W^{\varrho_n})}- \frac{\nu(A)}{\nu(B)}\right| > \frac{\delta}{n^{1/4}} \text{ and } \zeta(W^{\varrho_n}) \ge n},\]
and it remains to bound this last probability. Straightforward calculations show that if $n$ is large enough (so e.g. $\delta n^{-1/4} < \nu(B) / 2$), then the event
\[\left\{\left|\frac{J_A(W^{\varrho_n})}{\zeta(W^{\varrho_n})} - \nu(A)\right| \le \frac{\delta}{n^{1/4}}\right\}
\cap \left\{\left|\frac{J_B(W^{\varrho_n})}{\zeta(W^{\varrho_n})} - \nu(B)\right| \le \frac{\delta}{n^{1/4}}\right\}\]
is included in
\[\left\{\left|\frac{J_A(W^{\varrho_n})}{J_B(W^{\varrho_n})} - \frac{\nu(A)}{\nu(B)}\right| \le \frac{\delta'}{n^{1/4}}\right\},\]
for some explicit $\delta'$ which depends on $\delta$, $\nu(A)$ and $\nu(B)$.
We may write
\begin{align*}
&\Pr{\left|\frac{J_A(W^{\varrho_n})}{\zeta(W^{\varrho_n})}- \nu(A)\right| > \frac{\delta}{n^{1/4}} \text{ and } \zeta(W^{\varrho_n}) \ge n}
\\
&\qquad\le \sum_{N \ge n} \Pr{\left|\frac{J_A(W^{\varrho_n})}{N}- \nu(A)\right| > \frac{\delta}{N^{1/4}} \text{ and } \zeta(W^{\varrho_n}) = N}
\\
&\qquad\le \sum_{N \ge n} \Pr{\left|\frac{\#\{1 \le i \le N : X_i \in A\}}{N}- \nu(A)\right| > \frac{\delta}{N^{1/4}}}
\\
&\qquad\le \sum_{N \ge n} 2 \exp(- 2 \delta^2 N^{1/2}),
\end{align*}
where the last bound follows from the Chernoff bound for binomial distributions.
The last sum is bounded by some constant times $n^{1/2} \exp(- 2 \delta^2 n^{1/2})$; the same holds with the set $B$, so we obtain after a union bound,
\[\Pr{\left|\frac{n}{J_B(W^{\varrho_n}_{n,A})}- \frac{\nu(A)}{\nu(B)}\right| > \frac{\delta}{n^{1/4}}}
\le K n^3 \e^{- 2 \delta^2 n^{1/2}},\]
for some $K > 0$ and the proof in the case $\alpha \in (1,2]$ is complete.

The argument is very similar in the case $\alpha=1$, appealing to~\eqref{eq:LLT_GW_Cauchy_loc} and~\eqref{eq:LLT_GW_Cauchy_gen}, we leave the details to the reader.
\end{proof}

In the next theorem, we prove that the empirical jump distribution of our conditioned paths fits in our general framework. For any path $P = (P_i)_{i \ge 0}$, we let $\Delta(P) = \max\{i \ge 1 : P_i - P_{i-1}\}$ be the largest jump, and we let $\Delta'(P)$ be the second largest jump.

\begin{thm}\label{thm:GW}
Assume that $\nu$ satisfies $\mathrm{(H_\alpha)}$ for some $\alpha \in [1, 2]$ or $\mathrm{(H_1^{loc})}$ and fix $A \subset \Z_{\ge -1}$ with $\nu(A) > 0$ arbitrary if $\alpha>1$, but take $A = \Z_{\ge -1}$ otherwise. Let $(\varrho_n)_{n \ge 1}$ be such that $\lim_{n \to \infty} a_n^{-1} \varrho_n = \varrho \nu(A)^{-1/\alpha}$ for some $\varrho \in [0, \infty)$.
\begin{enumerate}
\item\label{thm:GW_Gaussien} If $\nu$ satisfies $\mathrm{(H_2)}$, then 
\[a_n^{-2} \sum_{k \ge 1} k (k+1) J_k(W^{\varrho_n}_{n,A}) \cvproba \frac{2}{\nu(A)}
\qquad\text{and}\qquad
a_n^{-1} \Delta(W^{\varrho_n}_{n,A}) \cvproba 0.\]

\item\label{thm:GW_stable} If $\nu$ satisfies $\mathrm{(H_\alpha)}$ with $\alpha \in (1,2)$, then $a_n^{-2} \sum_{k \ge 1} k (k+1) J_k(W^{\varrho_n}_{n,A})$ converges in distribution to a random variable $Y_\alpha^\varrho$ whose law  does not depend on $\nu$.

\item\label{thm:GW_Cauchy_loc} If $\nu$ satisfies $\mathrm{(H_1^{loc})}$, recall that we take $A = \Z_{\ge -1}$, then
\[|b_n|^{-1} \Delta(W^{\varrho_n}_n) \cvproba 1
\quad\text{and}\quad
|b_n|^{-2} \sum_{k \le \Delta'(W^{\varrho_n}_n)} k (k+1) J_k(W^{\varrho_n}_n) \cvproba 0.\]

\item\label{thm:GW_Cauchy_gen} Similarly, if $\nu$ satisfies $\mathrm{(H_1)}$ and $\varrho_n = 1$, let $I$ be distributed as $P(I \ge x) = x^{-1}$ for every $x \ge 1$, then
\[|b_n|^{-1} \Delta(W_{\ge n}) \cvloi I
\quad\text{and}\quad
|b_n|^{-2} \sum_{k \le \Delta'(W_{\ge n})} k (k+1) J_k(W_{\ge n}) \cvproba 0.\]
\end{enumerate}
\end{thm}

In the case $\varrho_n = 1$, when moreover $\nu$ has finite variance, Theorem~\ref{thm:GW}~\eqref{thm:GW_Gaussien} was first obtained by Broutin \& Marckert~\cite{BM14} for $A = \Z_{\ge -1}$ and generalised to any $A$ in~\cite{Mar18b}. The last two statements when $\alpha = 1$ shall follow easily from~\cite{KR19}. A key idea to prove the first two statements, when $\alpha > 1$, as in~\cite{BM14, Mar18b}, is first to observe that the claims are invariant under cyclic shift, so we may consider a random walk bridge instead of a first-passage bridge, and then to compare the law of these bridges with that of the unconditioned random walk $S$ for which the claims are easy to prove. 
Precisely, when $\nu$ satisfies $\mathrm{(H_\alpha)}$ for some $\alpha \in (1, 2]$ and $A \subset \Z_{\ge -1}$, let us denote for every real $t \ge 1$ by
\[J^-_{A, t}(S) = \inf\left\{k \ge 1 : J_A((S_i)_{i \le k}) = \lfloor t\rfloor\right\}\]
the instant at which the walk $S$ makes its $\lfloor t\rfloor$'th step in $A$. Then, as we discussed in the preceding subsection, the path $W^{\varrho_n}_{n,A}$ has the law of the Vervaat transform of the path $S$ up to time $J^-_{A, n}(S)$ conditioned to be at $-\varrho_n$ at this time. Further, from the Markov property applied to $S$ at time $J^-_{A, n/2}(S)$ and the local limit theorem, one can see that there exists a constant $C > 0$ such that for every $n \ge 1$ and every event $\mathscr{E}_{A, n/2}(S)$ which is measurable with respect to the $J^-_{A, n/2}(S)$ first steps of the path $S$, we have that
\begin{equation}\label{eq:absolue_continuite_GW}
\limsup_{n \to \infty} \Prc{\mathscr{E}_{A, n/2}(S)}{S_{J^-_{A, n}(S)} = -\varrho_n}
\le C \limsup_{n \to \infty} \Pr{\mathscr{E}_{A, n/2}(S)},
\end{equation}
see e.g. Kortchemski~\cite{Kor12}, Lemmas~6.10 and~6.11 and Equation~44 there.

We may now prove Theorem~\ref{thm:GW}. Let us start with the Gaussian regime $\alpha=2$. 

\begin{proof}[Proof of Theorem~\ref{thm:GW}~\eqref{thm:GW_Gaussien}]
According to the preceding discussion, it suffices to prove the convergences
\begin{equation}\label{eq:degres_GW_Gaussien}
\begin{gathered}
a_n^{-2} \sum_{k \ge 1} k (k+1) J_k\left((S_i)_{i \le J^-_{A, n}(S)})\right) \cvproba \frac{2}{\nu(A)},
\\
a_n^{-1} \Delta\left((S_i)_{i \le J^-_{A, n}(S)})\right) \cvproba 0,
\end{gathered}
\end{equation}
under the conditional probability $\P(\, \cdot \mid S_{J^-_{A, n}(S)} = -\varrho_n)$. 
Moreover, if we cut this bridge at time $J^-_{A, n/2}(S)$, then the time- and space-reversal of the second part has the same law as the first one. Therefore it suffices to prove that~\eqref{eq:degres_GW_Gaussien} holds when $n$ is replaced by $n/2$.
By~\eqref{eq:absolue_continuite_GW}, it suffices finally to prove that~\eqref{eq:degres_GW_Gaussien} holds for the unconditioned random walk.

Recall the two slowly varying functions $L$ and $L_1$, which are respectively given by $L_1(x) = x^2 \P(X \ge x)$ and $L(x) = \Var(X \ind{|X| \le x})$ for every $x > 0$; recall from~\eqref{eq:ratio_fonctions_var_lente} and~\eqref{eq:constante_Levy} that $L_1 / L$ converges to $0$ and $n a_n^{-2} L(a_n)$ converges to $2$. Then for every $\varepsilon > 0$, it holds that
\[\Pr{a_n^{-1} \max_{1 \le i \le n} X_i \ge \varepsilon}
\le n \Pr{X \ge \varepsilon a_n}
= n (\varepsilon a_n)^{-2} L_1(\varepsilon a_n)
\cv 0,\]
where we used the fact that $L_1$ is slowly varying, so $L_1(\varepsilon a_n) \sim L_1(a_n) = o(L(a_n))$. Concerning the first convergence, we aim at showing that $a_n^{-2} \sum_{1 \le i \le n} X_i (X_i+1)$ converges in probability to $2$, which is equivalent to the fact that $a_n^{-2} \sum_{1 \le i \le n} X_i^2$ converges in probability to $2$ since $n^{-1} \sum_{1 \le i \le n} X_i$ converges in probability to $0$ by the law of large numbers, and $n = O(a_n^2)$. Let us fix $\varepsilon > 0$, then we have that
\[\Es{a_n^{-2} \sum_{1 \le i \le n} X_i^2 \ind{|X_i| \le \varepsilon a_n}}
= n a_n^{-2} \Es{X^2 \ind{|X| \le \varepsilon a_n}}
= n a_n^{-2} L(a_n) (1+o(1))
,\]
which converges to $2$
and, similarly,
\[\Var\left(a_n^{-2} \sum_{1 \le i \le n} X_i^2 \ind{|X_i| \le \varepsilon a_n}\right)
= n a_n^{-4} \Var\left(X^2 \ind{|X| \le \varepsilon a_n}\right)
\le \varepsilon^2 n a_n^{-2} \Es{X^2 \ind{|X| \le \varepsilon a_n}}
\]
which converges to $2 \varepsilon^2$.
We have shown that with high probability, we have $|X_i| \le \varepsilon a_n$ for every $i \le n$, so we conclude that indeed $a_n^{-2} \sum_{1 \le i \le n} X_i^2$ converges in probability to $2$. We have thus shown that
\[a_n^{-2} \sum_{k \ge 1} k (k+1) J_k((S_i)_{i \le n}) \cvproba 2
\qquad\text{and}\qquad
a_n^{-1} \Delta((S_i)_{i \le n})) \cvproba 0,\]
for the unconditioned random walk, and so, since $a_{cn} \sim c^{1/2} a_n$,
\[a_n^{-2} \sum_{k \ge 1} k (k+1) J_k((S_i)_{i \le n / \nu(A)}) \cvproba \frac{2}{\nu(A)}
\quad\text{and}\quad
a_n^{-1} \Delta((S_i)_{i \le n / \nu(A)})) \cvproba 0.\]
Since $n^{-1} J^-_{A, n}(S)$ converges almost surely to $\nu(A)^{-1}$ by Lemma~\ref{lem:scaling_GW}, we obtain~\eqref{eq:degres_GW_Gaussien} for the unconditioned random walk and the proof is complete.
\end{proof}

We next consider the regime $1 < \alpha < 2$.

\begin{proof}[Proof of Theorem~\ref{thm:GW}~\eqref{thm:GW_stable}]
The claim for the unconditioned random walk $S$ is easy: let $\mathscr{X}$ be the $\alpha$-stable Lévy process whose law at time $1$ is $\mathscr{X}^{(\alpha)}$, then by a classical result on random walks, the convergence in distribution of $a_n^{-1} S_n$ towards $\mathscr{X}^{(\alpha)}$ is equivalent to that of $(a_n^{-1} S_{\lfloor n t\rfloor})_{t \ge 0}$ towards $(\mathscr{X}_t)_{t \ge 0}$ in the Skorokhod's $J_1$ topology. Let $\Delta \mathscr{X}_s = \mathscr{X}_s - \mathscr{X}_{s-} \ge 0$, then the sum $\mathscr{S}_t = \sum_{s \le t} (\Delta \mathscr{X}_s)^2$ is well defined; in fact it is well known that the process $(\mathscr{S}_t)_{t \ge 0}$ is an $\alpha/2$-stable subordinator, whose law can be easily derived from that of $\mathscr{X}$.
Furthermore it is easily seen to be the limit in distribution of $(a_n^{-2} \sum_{i \le \lfloor n t\rfloor} X_i (X_i + 1))_{t \ge 0}$; note that no centring is needed here since $\alpha/2 < 1$. A simple consequence of the fact that $\mathscr{S}$ is a pure jump process is that the non-increasing rearrangement of the vector $(a_n^{-2} X_i (X_i + 1))_{1 \le i \le n}$ converges in distribution in the $\ell^1$ topology towards the decreasing rearrangement of the nonzero jumps of $(\mathscr{S}_t)_{t \in [0,1]}$.

Let us return to our random bridges; one may adapt the arguments from~\cite{Kor12} when $\varrho_n = 1$ to obtain the convergence in distribution under $\P(\, \cdot \mid S_{J^-_{A, n}(S)} = -\varrho_n)$ of the paths $(\nu(A)^{1/\alpha} a_n^{-1} \sum_{1 \le i \le \lfloor J^-_{A, n}(S) t \rfloor} X_i)_{t \in [0,1]}$ towards the bridge $(\mathscr{X}^\varrho_t)_{t \in [0,1]}$ which is informally the process $(\mathscr{X}_t)_{t \in [0,1]}$ conditioned to be at $-\varrho$ at time $1$. 
This implies the convergence of the $N$ largest values amongst $(X_i (X_i + 1))_{1 \le i \le n}$ towards the $N$ largest values amongst $((\Delta \mathscr{X}^\varrho_t)^2)_{0 \le t \le 1}$ for every $N$; by~\eqref{eq:absolue_continuite_GW} and the preceding paragraph, if $N$ is chosen large enough, the sum of all the other jumps is small so we obtain under the law $\P(\, \cdot \mid S_{J^-_{A, n}(S)} = -\varrho_n)$,
\[\nu(A)^{2/\alpha} a_n^{-2} \sum_{1 \le i \le n} X_i (X_i + 1) \cvloi \sum_{t \in [0,1]} (\Delta \mathscr{X}^\varrho_t)^2.\]
We conclude as in the preceding proof from the space-time reversal invariance.
\end{proof}

We finally consider the Cauchy regime. Let us start with the local conditioning. Recall that here we assume that $A = \Z_{\ge -1}$, and $W^{\varrho_n}_n$ is simply the walk $S$ conditioned to first hit $-\varrho_n$ at time $n$.

\begin{proof}[Proof of Theorem~\ref{thm:GW}~\eqref{thm:GW_Cauchy_loc}]
In the case $\varrho_n = 1$, Kortchemski \& Richier~\cite[Theorem~3]{KR19} found the joint limit of the $N$ largest jumps of $W^1_n$ for any $N$ fixed: the largest one is equivalent to $|b_n|$, whereas the others are of order $a_n = o(|b_n|)$ which implies our first claim. This can be generalised to any $\varrho_n = O(a_n) = o(b_n)$; indeed, the key is Proposition~20 there which still applies when one takes the $X^{(n)}_i$'s to be the jumps of the path $S$ (denoted by $W$ there!) conditioned to $S_n = - \varrho_n$ instead of $S_n = -1$: the only feature of the case $\varrho_n = 1$ which is used is that $\P(S_n = -\varrho_n) \sim n\varrho_n\cdot \P(X = |b_n|)$, which remains valid as soon as $\varrho_n = O(a_n) = o(b_n)$ by~\eqref{eq:LLT_GW_Cauchy_loc}.

By this extension of~\cite[Proposition~20]{KR19} the following holds as $n \to \infty$: let $U_n = \inf\{i \le n : X_i = \max_{j \le n} X_j\}$ be the first time at which the path $S$ up to time $n$ makes its largest jump, then the law under $\P(\, \cdot \mid S_n = -\varrho_n)$ of the vector $(X_1, \dots, X_{U_n-1}, X_{U_n+1}, \dots, X_n)$ is close in total variation to $n-1$ i.d.d. copies of $X$.
From this, the proof of Theorem~21 from~\cite{KR19} can be extended to show that if we first run the walk $S$ up to time $n-1$ and then send it to $-\varrho_n$, and then we construct $Z^{\varrho_n}_n$ as the Vervaat transform of this path $(S_0, \dots, S_{n-1}, -\varrho_n)$, then
\begin{equation}\label{eq:one_big_jump_Cauchy_loc}
d_{TV}\left((W^{\varrho_n}_n(i))_{0 \le i \le n}, (Z^{\varrho_n}_n(i))_{0 \le i \le n}\right) \cv 0.
\end{equation}
Finally, from these results, the proof of Theorem~3 in~\cite{KR19} remains unchanged and our first claim follows.
Concerning the second claim, it is equivalent to showing that $|b_n|^{-2} \sum_{i \le n} X_i^2 \ind{i \ne U_n}$ converges to $0$ under $\P(\, \cdot \mid S_n = - \varrho_n)$. As in the case $1 < \alpha < 2$, we have that $a_n^{-2}(X_1^2 + \dots + X_n^2)$ converges in distribution towards some $1/2$-stable random variable. Our claim then follows by~\eqref{eq:one_big_jump_Cauchy_loc} since $a_n = o(|b_n|)$.
\end{proof}

It only remains to consider the tail conditioning. Recall that here $A = \Z_{\ge - 1}$ and $\varrho_n = 1$.

\begin{proof}[Proof of Theorem~\ref{thm:GW}~\eqref{thm:GW_Cauchy_gen}]
Compared to the previous proof, we do not need any adaption here and the first claim now directly follows from~\cite[Theorem~6]{KR19}.
The second claim is more subtle. First, we have a similar approximation to~\eqref{eq:one_big_jump_Cauchy_loc} given by~\cite[Theorem~27]{KR19}; what replaces $Z^{\varrho_n}_n$ there is the process $\vec{Z}^{(n)}$ defined as follows: first $I_n$ is the last weak ladder time of $(S_i)_{i \le n}$, then conditionally on $\{I_n = j\}$, the path $\vec{Z}^{(n)}$ consists in three independent parts:
\begin{enumerate}
\item First $(\vec{Z}^{(n)}_i)_{i <  j}$ has the law of $(S_i)_{i < j}$ conditioned to satisfy $\min_{i \le j} S_i \ge 0$.

\item Then we make a big jump $\vec{Z}^{(n)}_j - \vec{Z}^{(n)}_{j-1}$, sampled from $\P(\, \cdot \mid X \ge |b_n|)$.

\item Finally $(\vec{Z}^{(n)}_{j+i} - \vec{Z}^{(n)}_j)_{i \ge 0}$ continues as the unconditioned random walk $S$.
\end{enumerate}
The big jump will be excluded in our sum and we have seen previously that the sum of $N$ copies of $X^2$ is of order $a_N^2$ as $N \to \infty$. Therefore, if we consider only the jumps of $W_{\ge n}$ after its big jump, then there are less than $\zeta(W_{\ge n})$ of them, and $|b_n|^{-1} \zeta(W_{\ge n})$ converges in distribution to $I$ as $n \to \infty$ by~\cite[Proposition~3.1]{KR19}, so the sum of the square of these jumps is at most of order $a_{|b_n|}^2$. We recall that $a_n$ is defined by $n \P(X \ge a_n) \to 1$; since $\P(X \ge x) = x^{-1} L_1(x)$ and that $\E[|X|] < \infty$, then $L_1(x) \to 0$, therefore $a_n = o(n)$, from which we deduce that the sum of the square of the jumps of the last part is at most of order $a_{|b_n|}^2 = o(|b_n|^2)$. 
Finally, let us consider the jumps of the first part, up to time $I_n-1$. As in the preceding proof, the sum of the square of the first $I_n-1$ jumps of $\vec{Z}^{(n)}$ grows like $a_{I_n}^2 = o(|b_n|^2)$, and precisely this path converges after scaling to the meander given by a $1/2$-stable process conditioned to be nonnegative at least up to time $1$.
\end{proof}

\subsection{Boltzmann planar maps}
\label{sec:Boltzmann}

Let us briefly recall the definition of Boltzmann distributions on planar maps, as introduced by Marckert \& Miermont~\cite{MM07}. Let us also refer to~\cite{BM17} for more details.
For every $\varrho \ge 1$, let $\Map^{\varrho}$ be the set of all finite rooted bipartite maps $M$ with perimeter $2\varrho$, and let $\PMap^{\varrho}$ be the set of all such rooted and pointed maps $(M, x_\star)$. Recall that $\Map^1$ and $\PMap^1$ can be seen as the set of maps and pointed maps without boundary by gluing the two boundary edges, we then drop the exponent $1$.
Let us fix a sequence of nonnegative real numbers $\q = (q_i ; i \ge 0)$ which, in order to avoid trivialities, satisfies $q_i > 0$ for at least one $i \ge 2$; we define a measure $\mathbf{w}^{\varrho}$ on $\Map^{\varrho}$ by setting
\[\mathbf{w}^{\varrho}(M) = \prod_{f \text{ inner face}} q_{\mathrm{deg}(f)/2},
\qquad M \in \Map^{\varrho},\]
where $\mathrm{deg}(f)$ is the degree of the face $f$. We set $\mathbf{W}^{\varrho} = \mathbf{w}^{\varrho}(\Map^{\varrho})$. We define similarly a measure $\mathbf{w}^{\varrho}_\star$ on $\PMap^{\varrho}$ by $\mathbf{w}^{\varrho}_\star((M, x_\star)) = \mathbf{w}^{\varrho}(M)$ for every $(M, x_\star) \in \PMap^{\varrho}$ and we put $\mathbf{W}^{\varrho}_\star = \mathbf{w}^{\varrho}_\star(\PMap^{\varrho})$.
We say that $\q$ is \emph{admissible} when $\mathbf{W}_\star = \mathbf{W}_\star^1$ is finite; this seems stronger than requiring $\mathbf{W}^1$ to be finite, but it is not, see~\cite{BCM19}; moreover, this implies that $\mathbf{W}^\varrho_\star$ (and so $\mathbf{W}^\varrho$) is finite for any $\varrho \ge 1$.
If $\q$ is admissible, we set
\[\P^\varrho(\cdot) = \frac{1}{\mathbf{W}^\varrho} \mathbf{w}^\varrho(\cdot)
\qquad\text{and}\qquad
\P^{\varrho, \star}(\cdot) = \frac{1}{\mathbf{W}^\varrho_\star} \mathbf{w}^\varrho_\star(\cdot).\]
For every integers $n, \varrho \ge 1$, let $\Map^\varrho_{E=n}$, $\Map^\varrho_{V=n}$ and $\Map^\varrho_{F=n}$ be the subsets of $\Map$ of maps with respectively $n$ edges, $n+1$ vertices, and $n$ inner faces. More generally, for every $A \subset \N$, let $\Map_{(F,A)=n}$ be the subset of $\Map$ of maps with $n$ inner faces whose degree belongs to $2A$ (and possibly other faces, but with a degree in $2\N \setminus 2A$). For every $S \in \{E, V, F\} \cup \bigcup_{A \subset \N} \{(F,A)\}$ and every $n \ge 2$, we define
\[\P^\varrho_{S=n}(M) = \P^\varrho(M \mid M \in \Map^\varrho_{S=n}),
\qquad M \in \Map^\varrho_{S=n},\]
the law of a rooted Boltzmann map with perimeter $2 \varrho$ conditioned to have `size' $n$. We also let $\P_{E \ge n}$ be the law of a rooted Boltzmann map without boundary conditioned to have at least $n$ edges.
We define similarly the laws $\P^{\varrho, \star}_{S=n}$ on such pointed maps in $\PMap^\varrho_{S=n}$ and $\P^\star_{E \ge n}$ on $\bigcup_{k \ge n} \PMap_{E=k}$.

According to~\cite{MM07} (see~\cite{Mar18b} for a presentation closer to the present work), the admissibility of a sequence can be checked analytically: $\q$ is admissible if and only if the power series $g_\q : x \mapsto 1 + \sum_{k \ge 1} \binom{2k-1}{k-1} q_k x^k$ (which is convex and strictly increasing) possesses a fixed point $Z_\q$ such that $g_\q'(Z_\q)\le1$; in this case it holds that $Z_\q = (\mathbf{W}_\star+1)/2$. 
When $\q$ is admissible, the sequence
\[\mu_\q(k) = (Z_\q)^{k-1} \binom{2k-1}{k-1} q_k
\qquad (k \ge 0),
\]
where $\mu_\q(0)$ is understood as $(Z_\q)^{-1}$, thus defines a probability measure with mean smaller than or equal to one, and we say that $\q$ is \emph{critical} when $\mu_\q$ has mean exactly one.

Sample a random labelled forest $(T^\varrho, \ell)$ as follows: first $T^\varrho$ has the law of a Bienaymé--Galton--Watson forest with $\varrho$ trees and offspring distribution $\mu_\q$; this means that its {\L}ukasiewicz path has the law of the first-passage bridge $W^\varrho$ of the preceding subsections associated with the distribution $\nu_\q(\cdot) = \mu_\q(\cdot+1)$.
Then, conditionally on $T^\varrho$, the random labelling $\ell$ is obtained by sampling the label of the roots uniformly at random in $\mathscr{B}_{\varrho}^{\ge -1}$ as defined in~\eqref{eq:def_pont} and, independently, for every branchpoint with outdegree $k \ge 1$, the label increments between itself and its offspring are sampled uniformly at random in $\mathscr{B}_{k}^{\ge -1}$.

Then we first construct a negative pointed map from this forest as discussed in Section~\ref{sec:bijection}, and then we re-root it at one of the $2\varrho$ edges along the boundary (keeping the external face to the right) chosen uniformly at random, then this new pointed map has the law $\P^{\varrho, \star}$. This can be shown in the case $\varrho = 1$ by adapting the arguments from~\cite{MM07} which rely on the coding of~\cite{BDG04}, and the generalisation to $\varrho \ge 2$ can be obtained similarly, with a straightforward analogue of \cite[Lemma~10]{BM17} for the re-rooting along the boundary.

Moreover, the pointed map has the law $\P^{\varrho, \star}_{S=n}$ when the forest has the law of such a Bienaymé--Galton--Watson forest conditioned to have $n$ vertices with outdegree in the set $B_S \subset \Z_+$ given by
\[B_E = \Z_+,
\qquad
B_V = \{0\},
\qquad
B_F = \N
\qquad\text{and}\qquad
B_{(F,A)} = A.\]
We may therefore rely on the preceding sections to obtain information about (pointed) Boltzmann maps. 

We let $\nu_\q(k) = \mu_\q(k+1)$ for every $k \ge -1$ which is centred if and only if $\q$ is critical. Let the sequences $(a_n)_{n\ge 1}$ and $(b_n)_{n\ge 1}$ be as in Section~\ref{sec:Luka_stable}: when $\alpha>1$, the sequence $(n^{-1/\alpha} a_n)_{n\ge 1}$ 
is slowly varying and $a_n^{-1} (X_1 + \dots + X_n)$ converges in distribution to an $\alpha$-stable random variable, where the $X_i$'s are i.i.d. with law $\nu_\q$; when $\alpha=1$, both $(n^{-1} a_n)_{n \ge 1}$ and $(n^{-1} b_n)_{n \ge 1}$ are slowly varying, with $b_n \to -\infty$ and $b_n / a_n \to -\infty$, and now $a_n^{-1} (X_1 + \dots + X_n - b_n)$ converges in distribution.

\begin{thm}\label{thm:Boltzmann}
Assume that $\q$ is an admissible and critical sequence such that the law $\nu_\q$ satisfies $\mathrm{(H_\alpha)}$ for some $\alpha \in [1, 2]$ or $\mathrm{(H_1^{loc})}$. Let $(\varrho_n)_{n\ge1}$ be such that $\limsup_{n \to \infty} a_n^{-1} \varrho_n < \infty$.
\begin{enumerate}
\item\label{thm:Boltzmann_stable} Fix $S \in \{E, V, F\} \cup \bigcup_{A \subset \N} \{(F,A)\}$.
Suppose that $\nu_\q$ satisfies $\mathrm{(H_\alpha)}$ with $\alpha \in (1, 2]$. For every $n \ge 1$, let $M^{\varrho_n}_{S=n}$ have the law $\P^{\varrho_n}_{S=n}$. Then from every increasing sequence of integers, one can extract a subsequence along which the spaces
\[\left(V(M^{\varrho_n}_{S=n}), a_n^{-1/2} \dgr, \pgr\right)_{n \ge 1}\]
converge in distribution in the Gromov--Hausdorff--Prokhorov topology to a limit with a nonzero diameter.

\item\label{thm:Boltzmann_Gaussien} Suppose furthermore that $\alpha = 2$; let $c_{\q, S} = (\mu_\q(B_S) / 2)^{1/2}$ and suppose that there exists $\varrho \in [0, \infty)$ such that $\lim_{n \to \infty} a_n^{-1} \varrho_n = \varrho / c_{\q, S}$, then the convergence in distribution
\[\left(V(M^{\varrho_n}_{S=n}) \left(\frac{3 c_{\q, S}}{2 a_n}\right)^{1/2} \dgr, \pgr\right) \cvloi (\Bmap^\varrho, \dBmap^\varrho, \pBmap^\varrho)\]
holds in the Gromov--Hausdorff--Prokhorov topology, where $\Bmap^0$ is the Brownian sphere and $\Bmap^\varrho$ is the Brownian disk with perimeter $\varrho$ if $\varrho>0$.

\item\label{thm:Boltzmann_Cauchy_loc} Suppose that $\nu_\q$ satisfies $\mathrm{(H_1^{loc})}$. For every $n \ge 1$, let $M^{\varrho_n}_{E=n}$ have the law $\P^{\varrho_n}_{E=n}$. Then the convergence in distribution
\[\left(V(M^{\varrho_n}_{E=n}), |2b_n|^{-1/2} \dgr, \pgr\right) \cvloi (\CRT_{X^0}, \dCRT_{X^0}, \pCRT_{X^0})\]
holds in the Gromov--Hausdorff--Prokhorov topology, where $X^0$ is the standard Brownian excursion. 

\item\label{thm:Boltzmann_Cauchy_gen} Suppose that $\nu_\q$ satisfies $\mathrm{(H_1)}$. For every $n \ge 1$, let $M_{E \ge n}$ have the law $\P_{E \ge n}$. Then the convergence in distribution
\[\left(V(M_{E \ge n}), |2b_n|^{-1/2} \dgr, \pgr\right) \cvloi (\CRT_{I^{1/2} X^0}, \dCRT_{I^{1/2} X^0}, \pCRT_{I^{1/2} X^0})\]
holds in the Gromov--Hausdorff--Prokhorov topology, where $X^0$ is the standard Brownian excursion, and $I$ is independently sampled from the law $\P(I \ge x) = x^{-1}$ for all $x \ge 1$. 
\end{enumerate}
Finally, the same results hold when the maps are obtained by forgetting the distinguished vertex in the pointed versions of the laws.
\end{thm}

For maps without boundary, the case where $\nu_\q$ satisfies $\mathrm{(H_\alpha)}$ with $\alpha \in (1, 2)$ was first introduced by Le~Gall \& Miermont~\cite{LGM11} in a more restricted case, where $\nu_\q(n) \sim \mathrm{Cst}\cdot n^{-1-\alpha}$, and extended to $\mathrm{(H_\alpha)}$ in~\cite{Mar18a}. More is known for this model than just the above statement; in particular the height and label process of the associated labelled tree converge in distribution. This could be extended to maps with a boundary in a similar way.
Le~Gall \& Miermont~\cite{LGM11} also proved that the subsequential limits all have Hausdorff dimension $2\alpha$ and showing that these limits all have the same topology and further the same distribution (at $\alpha$ fixed) is  the object of current investigation.

\begin{proof}
The statements for random pointed maps follow easily from Theorem~\ref{thm:GW} and the results from the introduction. Indeed the statement~\eqref{thm:Boltzmann_stable} for $\alpha \in (1, 2)$ follows from Theorem~\ref{thm:tension_cartes} and Theorem~\ref{thm:GW}~\eqref{thm:GW_stable}: According to the latter, the factor $\sigma_n^{1/2}$ in Theorem~\ref{thm:tension_cartes} for the random degree sequence of $M^{\varrho_n}_{S=n}$ is of order $a_n^{1/2}$; since furthermore conditionally on these degrees the map has the uniform distribution on the set of all possible maps, we may then apply Theorem~\ref{thm:tension_cartes} conditionally on the degrees, with $a_n^{1/2}$ instead of $\sigma_n^{1/2}$, and then average with respect to the degrees.
The statement~\eqref{thm:Boltzmann_Gaussien} for $\alpha = 2$ follows similarly from Theorem~\ref{thm:convergence_carte_disque} and Theorem~\ref{thm:GW}~\eqref{thm:GW_Gaussien}; here the scaling constant $(3 / (2 \sigma_n))^{1/2}$ is given by $((9 \mu_\q(B_S))/(8 a_n^2))^{1/4}$.
The statements~\eqref{thm:Boltzmann_Cauchy_loc} and~\eqref{thm:Boltzmann_Cauchy_gen} finally follow from Theorem~\ref{thm:convergence_cartes_CRT} and Theorems~\ref{thm:GW}~\eqref{thm:GW_Cauchy_loc} and~\ref{thm:GW}~\eqref{thm:GW_Cauchy_gen} respectively.

In the next proposition, we compare pointed and non-pointed maps, which allows to transfer these results to non-pointed maps.
\end{proof}

Let $\phi : \PMap \to \Map$ be the projection $\phi((M, x_\star)) = M$ which `forgets the distinguished vertex'. We stress that, unless in the case $S = V$ for which there is no bias, the law $\P^\varrho_{S=n}$ and the push-forward $\phi_\ast \P^{\varrho, \star}_{S=n}$ differ at $n$ fixed. Nonetheless, this bias disappears at the limit as shown in several works~\cite{Abr16, BJM14, BM17, Mar18b, Mar18a}.

\begin{prop}\label{prop:biais_cartes_Boltzmann_pointees}
Under the assumption of Theorem~\ref{thm:Boltzmann} \eqref{thm:Boltzmann_stable}, \eqref{thm:Boltzmann_Cauchy_loc}, and \eqref{thm:Boltzmann_Cauchy_gen} respectively, each corresponding total variation distance
\[\left\|\P^\varrho_{S=n} - \phi_* \P^{\varrho, \star}_{S=n}\right\|_{TV},
\qquad
\left\|\P^\varrho_{E=n} - \phi_* \P^{\varrho, \star}_{E=n}\right\|_{TV},
\quad\text{and}\quad
\left\|\P_{E\ge n} - \phi_* \P^{\star}_{E\ge n}\right\|_{TV},\]
converges to $0$ as $n \to \infty$.
\end{prop}

\begin{proof}
This result generalises Proposition~12 in~\cite{Mar18b} whose proof extends readily here: for notational convenience, suppose that we are in the first case of the theorem. Let $\Lambda(T^{\varrho_n}_{n,B_S})$ denote the number of leaves of the forest $T^{\varrho_n}_{n,B_S}$, which equals the number of vertices minus one in the associated map. Then the above total variation distance is bounded above by
\[\Es{\left|\frac{1}{\Lambda(T^{\varrho_n}_{n,B_S})+1} \frac{1}{\E[\frac{1}{\Lambda(T^{\varrho_n}_{n,B_S})+1}]} - 1\right|}.\]
see e.g. the proof of Proposition~12 in~\cite{Mar18b}. It thus suffices to prove that this expectation tends to $0$, which is the content of~\cite[Lemma~8]{Mar18b}. The proof of this lemma is fairly general once we have the exponential concentration of the proportion of leaves we obtained in the proof of Lemma~\ref{lem:scaling_GW}: take the sets $A$ and $B$ in~\eqref{eq:concentration_feuilles} to be respectively $B_S$ and $\{-1\}$ here. We refer to~\cite{Mar18b} for details.
\end{proof}

\subsection{Subcritical maps}
\label{sec:Boltzmann_sous_crit}

Let us end this paper with a condensation result similar to Theorem~\ref{thm:Boltzmann}~\eqref{thm:Boltzmann_Cauchy_loc} and~\eqref{thm:Boltzmann_Cauchy_gen} for \emph{subcritical} maps, that is, when the weight sequence $\q$ is so that the law $\mu_\q$ has mean 
\[m_\q = \sum_{k \ge 0} k \mu_\q(k) < 1.\]
In order to directly use the results available in the literature, we shall only work with maps without boundary, so $\varrho_n = 1$. The first convergence below, for pointed maps, is the main result of Janson \& Stef{\'a}nsson~\cite{JS15}.

\begin{thm}\label{thm:Boltzmann_sous_crit}
Assume that $\q$ is an admissible sequence such that $m_\q < 1$.

\begin{enumerate}
\item Suppose that there exists a slowly varying function $L$ and an index $\beta > 1$ such that we have
\[\mu_\q(k) = k^{-(1+\beta)} L(k),\qquad k \ge 1.\]
For every $n \ge 1$, let $M_{E=n}$ have the law $\P_{E=n}$. Then the convergence in distribution
\[\left(V(M_{E=n}), (2 n)^{-1/2} \dgr, \pgr\right) \cvloi (\CRT_{\gamma^{1/2} X^0}, \dCRT_{\gamma^{1/2} X^0}, \pCRT_{\gamma^{1/2} X^0})\]
holds in the Gromov--Hausdorff--Prokhorov topology, where $X^0$ is the standard Brownian excursion and $\gamma = 1 - m_\q$. 

\item Under the weaker assumption that there exists a slowly varying function $L$ and an index $\beta > 1$ such that
\[\mu_\q([k, \infty)) = k^{-\beta} L(k),\qquad k \ge 1,\]
if now $M_{E \ge n}$ has the law $\P_{E \ge n}$, then the convergence in distribution
\[\left(V(M_{E \ge n}), (2 n)^{-1/2} \dgr, \pgr\right) \cvloi (\CRT_{I^{1/2} X^0}, \dCRT_{I^{1/2} X^0}, \pCRT_{I^{1/2} X^0})\]
holds in the Gromov--Hausdorff--Prokhorov topology, where $X^0$ is the standard Brownian excursion, and $I$ is independently sampled from the law $\P(I \ge x) = (\frac{1-m_\q}{x})^{\beta}$ for all $x \ge 1-m_\q$. 
\end{enumerate}

Finally, the same results hold when the maps are obtained by forgetting the distinguished vertex in the pointed versions of the laws.
\end{thm}

\begin{proof}
The proof goes exactly as that of Theorem~\ref{thm:Boltzmann} in the case $\alpha = 1$. First, the analogue of Theorem~\ref{thm:GW}~\eqref{thm:GW_Cauchy_loc} is provided by Kortchemski~\cite{Kor15}: according to Theorem~1 there, with the notation of Theorem~\ref{thm:GW} here, we have in the first case,
\[n^{-1} \left(\Delta(W_n), \Delta'(W_n)\right) \cvproba (1-m_\q,0).\]
This implies that $n^{-2} \sum_{k \le \Delta'(W_n)} k (k+1) J_k(W_n) \le n^{-1} (1+\Delta'(W_n))^2$ converges in probability to $0$. The analogue of Theorem~\ref{thm:GW}~\eqref{thm:GW_Cauchy_gen} in the second case is provided by Kortchemski \& Richier~\cite{KR20}: the convergence in distribution
\[n^{-1} \left(\Delta(W_n), \Delta'(W_n)\right) \cvloi (I,0)\]
follows from Proposition~9 there, and again this implies that $n^{-2} \sum_{k \le \Delta'(W_n)} k (k+1) J_k(W_n)$ converges in probability to $0$. Theorem~\ref{thm:convergence_cartes_CRT} allows us to conclude in the case of pointed maps.

In order to consider non-pointed maps, we need an analogous result to Proposition~\ref{prop:biais_cartes_Boltzmann_pointees}, which holds as soon as the proportion of leaves in the tree coded by $W_n$ or $W_{\ge n}$ satisfies an exponential concentration as in~\eqref{eq:concentration_feuilles}. The proof is easily adapted: we only need the analogue of Equations~\eqref{eq:LLT_GW_Cauchy_gen} and~\eqref{eq:LLT_GW_Cauchy_loc} with $\varrho_n=1$ on the hitting time of $-1$; the former, in the case of the tail conditioning, is given by~\cite[Equation~9]{KR20} and references therein, and the latter, in the case of the local conditioning, is given by~\cite[Equation~14]{Kor15}. 
\end{proof}

\appendix
\section{Proof of the spinal decomposition}
\label{sec:preuve_epine}

In this appendix we prove Lemma~\ref{lem:multi_epines_sans_remise} on the geometry of random forests spanned by $q \ge 1$ random vertices. We start with the simpler case $q=1$, the case $q \ge 2$ is very similar but requires more notation.

\subsection{Proof of the one-point decomposition}

Let us first consider the simpler case of a single random vertex. In this case, we have $b=0$ and $c=1$ so the upper bound reads simply
\[\frac{\varrho_n + \sum_{i=1}^h (k_i - 1)}{\vertices_n}
\cdot \Pr{\bigcap_{i \le h} \left\{(\xi_{d_n}(i), \chi_{d_n}(i)) = (k_i, j_i)\right\}}.\]

\begin{proof}[Proof of Lemma~\ref{lem:multi_epines_sans_remise} for $q = 1$]
The proof relies on the fact that there is a one-to-one correspondence between a plane forest with $\varrho$ trees and a distinguished vertex $x$ on the one hand, and on the other hand the triplet given by the knowledge of which of the $\varrho$ trees contains $x$, the vector $\Cont(x)$, and the plane forest obtained by removing all the strict ancestors of $x$; note that this forest contains $\varrho + \sum_{0 \le i < |x|} (k_{a_i(x)} - 1)$ trees.
Recall that for any sequence $\theta = (\theta_\ell)_{\ell \ge 0}$ of nonnegative integers with finite sum $|\theta|$, the number of plane forests having exactly $\theta_\ell$ vertices with $\ell$ offspring for every $\ell \ge 0$ is given by
\[\# \ensembles{F}(\theta) = \frac{r}{|\theta|} \binom{|\theta|}{(\theta_\ell)_{\ell \ge 0}},\]
where $r= \sum_{\ell \ge 0} (1-\ell) \theta_\ell$ is the number of roots.
Fix $(k_i, j_i)_{i=1}^h$ a sequence of positive integers such that $1 \le j_i \le k_i$ for each $i$. For every $\ell \ge 1$, let $m_\ell = \#\{1 \le i \le h : k_i = \ell\}$ and assume that $m_\ell \le d_n(\ell)$; let also $m_0 = 0$ and $\m = (m_\ell)_{\ell \ge 0}$. 
By the preceding bijection, if $x_n$ is a vertex chosen uniformly at random, then we have
\begin{align*}
\Pr{\Cont(x_n) = (k_i, j_i)_{i=1}^h}
&= \frac{\varrho_n \#\ensembles{F}(d_n - \m)}{\vertices_n \#\ensembles{F}(d_n)}
\\
&= \frac{\varrho_n \frac{\varrho_n + \sum_{i=1}^h (k_i - 1)}{\vertices_n-h} \binom{\vertices_n-h}{(d_n(\ell)-m_\ell)_{\ell\ge 0}}}{\vertices_n \frac{\varrho_n}{\vertices_n} \binom{\vertices_n}{(d_n(\ell))_{\ell\ge 0}}}
\\
&= \frac{\varrho_n + \sum_{i=1}^h (k_i - 1)}{\vertices_n-h} \frac{(\vertices_n-h)!}{\vertices_n!} \prod_{\ell \ge 1} \frac{d_n(\ell)!}{(d_n(\ell)-m_\ell)!}.
\end{align*}
Let us set
\begin{align*}
P((k_i, j_i)_{i \le h}) 
&= \Pr{\bigcap_{i \le h} \left\{(\xi_{d_n}(i), \chi_{d_n}(i)) = (k_i, j_i)\right\}} 
\\
&= \frac{\prod_{\ell \ge 1} \ell^{- m_\ell} (\ell d_n(\ell)) (\ell d_n(\ell) - 1) \cdots (\ell d_n(\ell) - m_\ell + 1)}{\edges_n (\edges_n - 1) \cdots (\edges_n - h + 1)}
\\
&= \frac{(\edges_n - h)!}{\edges_n!} \prod_{\ell \ge 1} \ell^{- m_\ell} \frac{(\ell d_n(\ell))!}{(\ell d_n(\ell) - m_\ell)!}.
\end{align*}
Then the probability that $\Cont(u_n) = (k_i, j_i)_{i=1}^h$ equals $P((k_i, j_i)_{i \le h})$ times
\[\frac{\varrho_n + \sum_{i=1}^h (k_i - 1)}{\vertices_n-h} \frac{(\vertices_n-h)!}{\vertices_n!}
\frac{\edges_n!}{(\edges_n - h)!} 
\left(\prod_{\ell \ge 1} \ell^{m_\ell} \frac{d_n(\ell)! (\ell d_n(\ell) - m_\ell)!}{(d_n(\ell)-m_\ell)! (\ell d_n(\ell))!}\right)
.\]
Observe that $\vertices_n \ge \edges_n + 1$, so
\[\frac{(\vertices_n-h)!}{\vertices_n!} \frac{\edges_n!}{(\edges_n - h)!}
= \frac{\vertices_n-h}{\vertices_n} \prod_{i=0}^{h-1} \frac{\edges_n - i}{\vertices_n - 1 - i}
\le \frac{\vertices_n-h}{\vertices_n}.\]
Moreover, we claim that $\ell^{m_\ell} \frac{d_n(\ell)! (\ell d_n(\ell) - m_\ell)!}{(d_n(\ell)-m_\ell)! (\ell d_n(\ell))!} \le 1$ for every $\ell \ge 1$. Indeed, it equals $1$ when $\ell=1$; we suppose next that $\ell \ge 2$, so in particular $m_\ell \le d_n(\ell) \le \ell d_n(\ell) / 2$. It is simple to check that for every $x \in [0, 1/2]$, we have $(1-x)^{-1} \le 2^{2x}$, and for every $x \in [0, 1]$, we have $1-x \le 2^{-x}$. It follows that for every $\ell \ge 2$ such that $d_n(\ell) \ne 0$,
\begin{align*}
\ell^{m_\ell} \frac{d_n(\ell)! (\ell d_n(\ell)-m_\ell)!}{(d_n(\ell)-m_\ell)! (\ell d_n(\ell))!}
&= \frac{d_n(\ell)!}{(d_n(\ell)-m_\ell)! d_n(\ell)^{m_\ell}} \frac{(\ell d_n(\ell)-m_\ell)! (\ell d_n(\ell))^{m_\ell}}{(\ell d_n(\ell))!}
\\
&= \prod_{i = 0}^{m_\ell-1} \frac{d_n(\ell) - i}{d_n(\ell)} \frac{\ell d_n(\ell)}{\ell d_n(\ell) - i}
\\
&\le \prod_{i = 0}^{m_\ell-1} 
\exp\left(\ln 2 \left(-\frac{i}{d_n(\ell)} + \frac{2i}{\ell d_n(\ell)}\right)\right),
\end{align*}
which is indeed bounded by $1$ for $\ell \ge 2$. 
We have thus shown that
\[\Pr{\Cont(u_n) = (k_i, j_i)_{i=1}^h} \le \frac{\varrho_n + \sum_{i=1}^h (k_i - 1)}{\vertices_n} P((k_i, j_i)_{i \le h}),\]
which corresponds to our claim in the case $q=1$.
\end{proof}

\subsection{Proof of the multi-point decomposition}

We next turn to the proof in the general case; the difference is that one has to deal with the branchpoints of the reduced forest, which, we recall, are not considered in the vector of content.

\begin{proof}[Proof of Lemma~\ref{lem:multi_epines_sans_remise} for $q \ge 2$]
Sample $x_{n,1}, \dots, x_{n,q}$ i.i.d. uniformly random vertices in $\fn$ and assume that none is an ancestor of another. As in the case $q=1$, there is a one-to-one correspondence between the plane forest $\fn$ and $x_{n,1}, \dots, x_{n,q}$ on the one hand, and on the other hand the plane forest obtained by removing their strict ancestors, the vector $\Cont(x_{n,1}, \dots, x_{n,q})$, and the following data: first the number $c$ of trees containing at least one of the $x_{n,i}$'s and the knowledge of which ones, second for each of the $b$ branchpoints: their total number $r$ of children in $\fn$, their number $d$ of children which are ancestors of at least one $x_{n,j}$, and the relative positions $z_i$'s of these children.

Therefore, fix $(k_i, j_i)_{i=1}^h$, also $c \in \{1, \dots, q\}$, $b \in \{0, \dots, q-1\}$, and for every $p \in \{1, \dots, b\}$, fix $r(p) \ge d(p) \ge 2$ and integers $1 \le z_{p,1} < \dots < z_{p, d(p)} \le r(p)$, and let us consider the following event: first $\Cont(x_{n,1}, \dots, x_{n,q}) = (k_i, j_i)_{i=1}^h$, second, in the forest spanned by $x_{n,1}, \dots, x_{n,q}$, we have $c$ trees, $b$ branchpoints, for every $p \le b$, the $p$'th branchpoint in lexicographical order has $r(p)$ children in total in the original forest $\fn$, amongst which $d(p)$ belong to the reduced forest, and the relative positions of the latter are given by the $z_{p,i}$'s.
If we set $m_0 = 0$ and for every $\ell \ge 1$, we let $m_\ell = \#\{1 \le i \le h : k_i = \ell\}$ and $\overline{m}_\ell = m_\ell + \sum_{p=1}^b \ind{r(p) = \ell}$, then on this event, the complement of the reduced forest is a forest with degree sequence $(d_n(\ell)-\overline{m}_\ell)_{\ell\ge 0}$; note that it contains $R = \varrho_n - c + q + \sum_{p=1}^b (r(p) - d(p)) + \sum_{i=1}^h (k_i - 1)$ trees and $\vertices_n - h - b$ vertices. Therefore the probability of our event is given by
\begin{align*}
\frac{\binom{\varrho_n}{c} \#\ensembles{F}(d_n - \overline{\m})}{\vertices_n^q \#\ensembles{F}(d_n)}
&= \frac{\binom{\varrho_n}{c} \frac{R}{\vertices_n-h-b} \binom{\vertices_n-h-b}{(d_n(\ell)-\overline{m}_\ell)_{\ell\ge 0}}}{\vertices_n^q \frac{\varrho_n}{\vertices_n} \binom{\vertices_n}{(d_n(\ell))_{\ell\ge 0}}}
\\
&= \frac{\binom{\varrho_n}{c} R}{\varrho_n \vertices_n^{q-1} (\vertices_n-h-b)}
\frac{(\vertices_n-h-b)!}{\vertices_n!}
\prod_{\ell \ge 1} \frac{(d_n(\ell))!}{(d_n(\ell)-\overline{m}_\ell)!}.
\end{align*}
Now as previously, 
\[\frac{(\vertices_n-h-b)!}{\vertices_n!}
\le \frac{(\vertices_n-h-b)!}{(\vertices_n - h)!} \frac{\vertices_n-h}{\vertices_n} \frac{(\edges_n - h)!}{\edges_n!}
= \frac{(\vertices_n-h-b)!}{\vertices_n (\vertices_n - h - 1)!} \frac{(\edges_n - h)!}{\edges_n!}.\]
Next,
\begin{align*}
\prod_{\ell \ge 1} \frac{(d_n(\ell))!}{(d_n(\ell)-\overline{m}_\ell)!}
&= \prod_{\ell \ge 1} \frac{(d_n(\ell))!}{(d_n(\ell)-m_\ell)!}
\prod_{\ell \ge 1} \frac{(d_n(\ell)-m_\ell)!}{(d_n(\ell)-m_\ell-\sum_{p=1}^b \ind{r(p) = \ell})!}
\\
&\le \prod_{\ell \ge 1} \frac{(d_n(\ell))!}{(d_n(\ell)-m_\ell)!}
\prod_{p=1}^b d_n(r(p)).
\end{align*}
Finally, we have seen that
\[\P\bigg(\bigcap_{i \le h} \left\{(\xi_{d_n}(i), \chi_{d_n}(i)) = (k_i, j_i)\right\}\bigg)
= \frac{(\edges_n - h)!}{\edges_n!} \prod_{\ell \ge 1} \ell^{- m_\ell} \frac{(\ell d_n(\ell))!}{(\ell d_n(\ell) - m_\ell)!},\]
which we shall denote by $P((k_i, j_i)_{i \le h})$.
We infer that
\begin{align*}
&\frac{\binom{\varrho_n}{c} \#\ensembles{F}(d_n - \overline{\m})}{\vertices_n^q \#\ensembles{F}(d_n)}
\\
&\le \frac{\binom{\varrho_n}{c} R}{\varrho_n \vertices_n^{q-1} (\vertices_n-h-b)}
\frac{(\vertices_n-h-b)!}{\vertices_n (\vertices_n - h - 1)!} \frac{(\edges_n - h)!}{\edges_n!}
\prod_{\ell \ge 1} \frac{(d_n(\ell))!}{(d_n(\ell)-m_\ell)!}
\prod_{p=1}^b d_n(r(p))
\\
&\le \frac{\binom{\varrho_n}{c} R}{\varrho_n \vertices_n^q} \frac{(\vertices_n-h-b-1)!}{(\vertices_n - h - 1)!}
\underbrace{\prod_{\ell \ge 1} \ell^{m_\ell} \frac{(\ell d_n(\ell) - m_\ell)! (d_n(\ell))!}{(\ell d_n(\ell))! (d_n(\ell)-m_\ell)!}}_{\displaystyle \le 1}
\prod_{p=1}^b d_n(r(p))
\cdot P((k_i, j_i)_{i \le h}),
\end{align*}
where the fact that the first product is bounded by $1$ was shown at the very end of the proof of the case $q=1$. 

Since $c \ge 1$, $b \le q-1$, and $2 \le d(p) \le r(p) \le \Delta_n$, we have
\[- c + q + \sum_{p=1}^b (r(p) - d(p))
\le - 1 + q + b (\Delta_n-1)
\le (q-1) \Delta_n,\]
hence
\[R = \varrho_n - c + q + \sum_{p=1}^b (r(p) - d(p)) + \sum_{i=1}^h (k_i - 1)
\le \varrho_n + (q-1) \Delta_n + \sum_{i=1}^h (k_i - 1).\]
Moreover, from our assumption that $h, q \le \vertices_n/4$,
\[\frac{(\vertices_n-h-b-1)!}{(\vertices_n - h - 1)!}
\le \left(\frac{1}{\vertices_n - h - b}\right)^b
\le \left(\frac{2}{\vertices_n}\right)^b,\]
and finally, since $\binom{\varrho_n}{c} \le \varrho_n^c$, we obtain
\[\frac{\binom{\varrho_n}{c} \#\ensembles{F}(d_n - \overline{\m})}{\vertices_n^q \#\ensembles{F}(d_n)}
\le 
2^b \varrho_n^c \frac{\varrho_n + (q-1) \Delta_n + \sum_{i=1}^h (k_i - 1)}{\varrho_n \vertices_n^{q+b}} P((k_i, j_i)_{i \le h})
\cdot \prod_{p=1}^b d_n(r(p))
.\]

On the left is the probability that $\Cont(x_{n,1}, \dots, x_{n,q}) = (k_i, j_i)_{i=1}^h$ and that the reduced forest has $c$ trees, $q$ leaves, and $b$ branchpoints and that for every $p \le b$, the $p$'th branchpoint has $r(p)$ children, amongst which $d(p) \ge 2$ belong to the reduced forest, with relative positions given by $1 \le z_{p,1} < \dots < z_{p, d(p)} \le r(p)$. 
We now want to sum over all possible $z$'s, all vectors $d$'s and $r$'s; note that 
the quantities on the right before the product sign do not depend on them so we only focus on the term $\prod_{p=1}^b d_n(r(p))$. First note that the sum of this product over all the possible $z$'s is bounded by
\begin{align*}
\prod_{p=1}^b \binom{r(p)}{d(p)} d_n(r(p))
&\le \prod_{p=1}^b \Delta_n^{d(p)-2} r(p)(r(p) - 1) d_n(r(p))
\\
&\le \Delta_n^{q - c - b} \prod_{p=1}^b r(p)(r(p) - 1) d_n(r(p)),
\end{align*}
where for the last inequality, we note that, since the $d(p)$'s are the number of children of the branchpoints in the reduced forest, which contains $c$ trees and $q$ leaves, then we have $q = c + \sum_{p=1}^b (d(p)-1) = c + b + \sum_{p=1}^b (d(p)-2)$. 
Note also the very crude bound: there are less than $b q \le q^2$ such vectors $(d(1), \dots, d(b))$. We finally want to sum the last bound times $q^2$ over all the vectors $(r(1), \dots, r(b))$; we have
\[\sum_{r(1), \dots, r(b) \le \Delta_n} \prod_{p=1}^b r(p)(r(p) - 1) d_n(r(p))
= \bigg(\sum_{r \le \Delta_n} r (r-1) d_n(r)\bigg)^b
= \sigma_n^{2b}.\]
Since $\Delta_n \ge 2$, then $\Delta_n^2 \le 2 \Delta_n (\Delta_n - 1) \le 2\sigma_n^2$, so finally, since $q - c - b \le 0$, then $2^{(q - c - b)/2} \le 1$ and so
\[\sum_{r(1), \dots, r(b) \le \Delta_n} q^2 \prod_{p=1}^b \binom{r(p)}{d(p)} d_n(r(p))
\le q^2 \Delta_n^{q - c - b} \sigma_n^{2b}
\le q^2 \sigma_n^{q - c + b}.\]
We conclude that the probability that $\Cont(x_{n,1}, \dots, x_{n,q}) = (k_i, j_i)_{i=1}^h$ and the reduced forest has $q$ leaves, $c$ trees, and $b$ branchpoints is bounded by
\begin{multline*}
2^b \varrho_n^c \frac{\varrho_n + (q-1) \Delta_n + \sum_{i=1}^h (k_i - 1)}{\varrho_n \vertices_n^{q+b}} P((k_i, j_i)_{i \le h}) q^2 \sigma_n^{q - c + b}
\\
= q^2 2^b \left(\frac{\sigma_n}{\vertices_n}\right)^{q+b} \left(\frac{\varrho_n}{\sigma_n}\right)^{c-1} \frac{\varrho_n + (q-1) \Delta_n + \sum_{i=1}^h (k_i - 1)}{\sigma_n} P((k_i, j_i)_{i \le h}),
\end{multline*}
and the proof is complete since $b \le q-1$.
\end{proof}


{\linespread{1}\selectfont

}

\end{document}